\documentclass[a4paper,10pt,leqno]{article}
\usepackage[margin=2.5cm]{geometry}
\usepackage[all]{xy}
\usepackage{mathtools}
\usepackage{amssymb}
\usepackage{amsxtra}
\usepackage{amsthm}
\usepackage{calrsfs}
\usepackage[T1]{fontenc}
\usepackage{lmodern}
\usepackage[pagebackref,breaklinks]{hyperref}
\usepackage[lite,abbrev]{amsrefs}
\theoremstyle{definition}
\newtheorem{construction}[subsubsection]{Construction}
\newtheorem{definition}[subsubsection]{Definition}
\newtheorem{sdefinition}[subsection]{Definition}
\newtheorem{notation}[subsubsection]{Notation}
\newtheorem{example}[subsubsection]{Example}
\newtheorem{remark}[subsubsection]{Remark}
\newtheorem{sremark}[subsection]{Remark}

\theoremstyle{plain}
\newtheorem{lemma}[subsubsection]{Lemma}
\newtheorem{prop}[subsubsection]{Proposition}

\newtheorem{theorem}[subsubsection]{Theorem}
\newtheorem{stheorem}[subsection]{Theorem}
\newtheorem{cor}[subsubsection]{Corollary}
\newtheorem{scor}[subsection]{Corollary}

\newcommand{\one}{\mathbf{1}}
\newcommand{\Z}{\mathbb{Z}}
\newcommand{\Q}{\mathbb{Q}}
\newcommand{\R}{\mathbb{R}}
\newcommand{\C}{\mathbb{C}}
\newcommand{\F}{\mathbb{F}}
\newcommand{\cA}{\mathcal{A}}
\newcommand{\cB}{\mathcal{B}}
\newcommand{\cC}{\mathcal{C}}
\newcommand{\cD}{\mathcal{D}}
\newcommand{\cE}{\mathcal{E}}
\newcommand{\cF}{\mathcal{F}}
\newcommand{\cG}{\mathcal{G}}
\newcommand{\cH}{\mathcal{H}}
\newcommand{\cL}{\mathcal{L}}
\newcommand{\cO}{\mathcal{O}}
\newcommand{\cS}{\mathcal{S}}
\newcommand{\cT}{\mathcal{T}}
\newcommand{\cHom}{\mathcal{H}om}
\newcommand{\bIsom}{\mathbf{Isom}}
\newcommand{\bHom}{\mathbf{Hom}}
\newcommand{\op}{\mathrm{op}}
\newcommand{\Sym}{\mathrm{Sym}}
\newcommand{\Gal}{\mathrm{Gal}}
\newcommand{\id}{\mathrm{id}}
\newcommand{\ev}{\mathrm{ev}}
\newcommand{\adj}{\mathrm{adj}}
\newcommand{\rD}{\mathrm{D}}
\newcommand{\Dbar}{\overline{D}}
\newcommand{\rR}{\mathrm{R}}
\newcommand{\Perm}{\mathfrak{S}}
\newcommand{\fm}{\mathfrak{m}}
\newcommand{\Perv}{\mathrm{Perv}}
\newcommand{\Aut}{\mathrm{Aut}}
\newcommand{\Map}{\mathrm{Map}}
\newcommand{\Ext}{\mathrm{Ext}}
\newcommand{\Ker}{\mathrm{Ker}}
\newcommand{\Coker}{\mathrm{Coker}}
\newcommand{\Img}{\mathrm{Im}}
\newcommand{\Hom}{\mathrm{Hom}}
\newcommand{\End}{\mathrm{End}}
\newcommand{\Fun}{\mathrm{Fun}}
\newcommand{\Nat}{\mathrm{Nat}}
\newcommand{\Spec}{\mathrm{Spec}}
\newcommand{\Frac}{\mathrm{Frac}}
\newcommand{\GS}{\mathrm{GS}}
\newcommand{\GW}{\mathrm{GW}}
\newcommand{\IH}{\mathrm{IH}}
\newcommand{\Sh}{\mathrm{Sh}}
\newcommand{\IC}{\mathrm{IC}}
\newcommand{\Ob}{\mathrm{Ob}}
\newcommand{\Nil}{\mathrm{Nil}}
\newcommand{\Fr}{\mathrm{Frob}}
\newcommand{\gr}{\mathrm{gr}}
\newcommand{\lisse}{\mathrm{lisse}}
\newcommand{\sd}{\mathrm{sd}}
\newcommand{\dl}{\mathrm{d}}
\newcommand{\tr}{\mathrm{tr}}
\newcommand{\red}{\mathrm{red}}
\newcommand{\diag}{\mathrm{diag}}
\newcommand{\orth}{\mathrm{orth}}
\newcommand{\rss}{\mathrm{ss}}
\newcommand{\dec}{\mathrm{spl}}
\newcommand{\rH}{\mathrm{H}}
\newcommand{\rI}{\mathrm{I}}
\newcommand{\rK}{\mathrm{K}}
\newcommand{\Ptau}{\prescript{P}{}{\tau}}
\newcommand{\PR}{\prescript{P}{}{\rR}}
\newcommand{\PH}{\prescript{P}{}{\rH}}
\newcommand{\PcD}{\prescript{P}{}{\cD}}
\newcommand{\pH}{\prescript{p}{}{\mathcal{H}}}
\newcommand{\pR}{\prescript{p}{}{\rR}}
\newcommand{\simto}{\xrightarrow{\sim}}
\newcommand{\Qlb}{\overline{\Q}_\ell}
\newcommand{\Fqb}{\overline{\F}_q}
\newcommand{\res}{\mathbin{|}}

\numberwithin{equation}{subsection}

\begin{document}
\title{Parity and symmetry in intersection and ordinary cohomology}
\author{Shenghao Sun\thanks{Yau Mathematical Sciences Center, Tsinghua University, Beijing 100084, China;
email: \texttt{shsun@math.tsinghua.edu.cn}. Partially supported by ANR grant
G-FIB.}\and Weizhe Zheng\thanks{Morningside Center of Mathematics, Academy
of Mathematics and Systems Science, Chinese Academy of Sciences, Beijing
100190, China; email: \texttt{wzheng@math.ac.cn}. Partially supported by
China's Recruitment Program of Global Experts; National Natural Science
Foundation of China Grant 11321101; National Center for Mathematics and
Interdisciplinary Sciences and Hua Loo-Keng Key Laboratory of Mathematics,
Chinese Academy of Sciences.} }
\date{}
\maketitle

\begin{flushright}
\emph{To the memory of Torsten Ekedahl}
\end{flushright}

\begin{abstract}
In this paper, we show that the Galois representations provided by
$\ell$-adic cohomology of proper smooth varieties, and more generally by
$\ell$-adic intersection cohomology of proper varieties, over any field,
are orthogonal or symplectic according to the degree. We deduce this from
a preservation result of orthogonal and symplectic pure perverse sheaves
by proper direct image. We show moreover that the subgroup of the
Grothendieck group generated by orthogonal pure perverse sheaves of even
weights and symplectic pure perverse sheaves of odd weights are preserved
by Grothendieck's six operations. Over a finite field, we deduce parity
and symmetry results for Jordan blocks appearing in the Frobenius action
on intersection cohomology of proper varieties, and virtual parity results
for the Frobenius action on ordinary cohomology of arbitrary varieties.
\end{abstract}

\section{Introduction}

The $n$-th cohomology of a compact K\"ahler manifold $X$ is equipped with a
pure Hodge structure of weight $n$:
\[\rH^n(X,\Q)\otimes_\Q \C = \bigoplus_{p+q=n}\rH^{pq},\]
where $\rH^{pq}\simeq \rH^q(X,\Omega^p_X)$ satisfies
$\overline{\rH^{pq}}=\rH^{qp}$. In particular, $\rH^n(X,\Q)$ is
even-dimensional for $n$ odd. Hodge decomposition and Hodge symmetry extend
to proper smooth schemes over $\C$ \cite[Proposition 5.3]{DeligneL} by
Chow's lemma and resolution of singularities. Thus, in this case,
$\rH^n(X(\C),\Q)$ is also even-dimensional for $n$ odd. Moreover, the pure
Hodge structure of weight $n$ on $\rH^n(X(\C),\Q)$ is polarizable, in the
sense that there exists a morphism of Hodge structures
\[\rH^n(X(\C),\Q)\otimes \rH^n(X(\C),\Q) \to \Q(-n),\]
symmetric for $n$ even and alternating for $n$ odd, satisfying certain
positivity conditions, which implies that the pairing is perfect.

Now let $\bar k$ be a separably closed field of characteristic $p\ge 0$ and
let $\ell\neq p$ be a prime number. For a projective smooth scheme $X$ of
finite type over $\bar k$, hard Lefschetz theorem \cite[Th\'eor\`eme
4.1.1]{WeilII} and Poincar\'e duality equip the $n$-th $\ell$-adic
cohomology $\rH^n(X,\Q_\ell)$ of $X$ with a nondegenerate bilinear form that
is symmetric for $n$ even and alternating for $n$ odd. In particular,
$\rH^n(X,\Q_\ell)$ is even-dimensional for $n$ odd. In a remark following
\cite[Corollaire 4.1.5]{WeilII}, Deligne predicts that the evenness of the
odd-degree Betti numbers should hold more generally for proper smooth
schemes over $\bar k$. This is recently shown by Suh \cite[Corollary
2.2.3]{Suh} using crystalline cohomology.

The goal of this article is to study problems of parity and symmetry in more
general settings, including symmetry of Galois actions on cohomology. Our
approach is different from that of Suh as we do not use $p$-adic cohomology.

For a general scheme $X$ of finite type over $\bar k$, $\rH^n(X,\Q_\ell)$ is
not ``pure'' and not necessarily even-dimensional for $n$ odd. Before going
into results for such mixed situations, let us first state our results in
the pure case for intersection cohomology.

\begin{stheorem}\label{t.i1}
Let $k$ be an arbitrary field of characteristic $p\ge 0$ and let $\bar k$ be
its separable closure. Let $X$ be a proper, equidimensional scheme over $k$.
Then, for $n$ even (resp.\ odd), the $n$-th $\ell$-adic intersection
cohomology group admits a $\Gal(\bar k/k)$-equivariant symmetric (resp.\
alternating) perfect pairing
\[\IH^n(X_{\bar k},\Q_\ell)\otimes \IH^n(X_{\bar k},\Q_\ell)\to \Q_\ell(-n).\]
Here $\Gal(\bar k/k)$ denotes the Galois group of $k$ and $X_{\bar
k}=X\otimes_k \bar k$.
\end{stheorem}

By definition, $\IH^n(X_{\bar k},\Q_\ell)=\rH^{n-d}(X_{\bar k},\IC_X)$,
where $\IC_X=j_{!*}(\Q_\ell[d])$, $d=\dim(X)$, $j\colon U\to X$ is an open
dense immersion such that $U_{\red}$ is regular. For $X$ proper smooth, we
have $\IH^n(X_{\bar k},\Q_\ell)=\rH^n(X_{\bar k},\Q_\ell)$, and the theorem
takes the following form. The statement was suggested to us by Takeshi
Saito. One may compare such pairings with polarizations of pure Hodge
structures mentioned at the beginning of the introduction.

\begin{scor}
Let $X$ be a proper smooth scheme over $k$. Then, for $n$ even (resp.\ odd),
the $n$-th $\ell$-adic cohomology group admits a $\Gal(\bar
k/k)$-equivariant symmetric (resp.\ alternating) perfect pairing
\[\rH^n(X_{\bar k},\Q_\ell)\otimes \rH^n(X_{\bar k},\Q_\ell)\to \Q_\ell(-n).\]
\end{scor}

Ignoring Galois actions, we obtain the following corollary. For $X$ proper
smooth, this gives another proof of Suh's result mentioned earlier.

\begin{scor}
Let $X$ be a proper, equidimensional scheme over $k$. Then $\IH^n(X_{\bar
k},\Q_\ell)$ is even-dimensional for $n$ odd.
\end{scor}

To demonstrate the strength of Theorem \ref{t.i1}, we give a reformulation
in the case where $k=\F_q$ is a finite field. In this case, the Galois
action is determined by Frobenius action. We let $\Fr_q\in \Gal(\Fqb/\F_q)$
denote the geometric Frobenius $x\mapsto x^{1/q}$. The eigenvalues of
$\Fr_q$ acting on $\IH^n(X_{\Fqb},\Q_\ell)$ are \emph{$q$-Weil integers of
weight $n$}, by which we mean algebraic integers $\lambda$ such that for
every embedding $\alpha\colon \Q(\lambda)\to \C$, we have $\lvert
\alpha(\lambda)\rvert^2=q^n$. We let $\mu_\lambda\in \Z_{\ge 0}$ denote the
multiplicity of the eigenvalue $\lambda$ for the action of $\Fr_q$ on
$\IH^n(X_{\Fqb},\Q_\ell)$. In other words, we put
\begin{equation}\label{e.mu}
\det(1-T\Fr_q\mid\IH^n(X_{\Fqb},\Q_\ell))=\prod_\lambda
(1-\lambda T)^{\mu_\lambda}.
\end{equation}
For $e\ge 1$, we let $\mu_{\lambda,e}\in \Z_{\ge 0}$ denote the number of
$e\times e$ Jordan blocks with eigenvalue $\lambda$ in the Jordan normal
form of $\Fr_q$ acting on $\IH^n(X_{\Fqb},\Qlb)$. We have
$\mu_{\lambda}=\sum_{e\ge 1} e\mu_{\lambda,e}$.

\begin{scor}\label{c.IH}
Let $X$ be a proper, equidimensional scheme over $\F_q$. In the above
notation, $\mu_{\lambda,e}=\mu_{q^n/\lambda,e}$. Moreover,
$\mu_{\sqrt{q^n},e}$ and $\mu_{-\sqrt{q^n},e}$ are even for $n+e$ even. In
particular, $\mu_{\lambda}=\mu_{q^n/\lambda}$ and, for $n$ odd,
$\mu_{\sqrt{q^n}}$ and $\mu_{-\sqrt{q^n}}$ are even.
\end{scor}

The last statement of Corollary \ref{c.IH} implies that for $n$ odd,
\begin{gather*}
\det(\Fr_q\mid \IH^n(X_{\Fqb},\Q_\ell))=q^{nb_n/2},\\
\det(1-T\Fr_q\mid \IH^n(X_{\Fqb},\Q_\ell))=q^{nb_n/2}T^{b_n}\det(1-q^{-n}T^{-1}\Fr_q\mid \IH^n(X_{\Fqb},\Q_\ell)),
\end{gather*}
where $b_n=\dim \IH^n(X_{\Fqb},\Q_\ell)$. Moreover, it implies that $\dim
\IH^n(X_{\Fqb},\Q_\ell)=\sum_\lambda \mu_\lambda$ is even.

\begin{sremark}
Some special cases of the last statement of Corollary \ref{c.IH} were
previously known.
\begin{enumerate}
\item Gabber's theorem on independence of $\ell$ for intersection
    cohomology \cite[Theorem~1]{Gabber} states that \eqref{e.mu} belongs
    to $\Z[T]$ and is independent of $\ell$. The fact that \eqref{e.mu}
    belongs to $\Q[T]$ implies $\mu_\lambda=\mu_{\lambda'}$ for $\lambda$
    and $\lambda'$ in the same $\Gal(\overline{\Q}/\Q)$-orbit. In
    particular, $\mu_\lambda=\mu_{q^n/\lambda}$, and, if $q$ is \emph{not}
    a square and $n$ is odd, $\mu_{\sqrt{q^n}}=\mu_{-\sqrt{q^n}}$, so that
    $\dim\IH^n(X_{\Fqb},\Q_\ell)=\sum_\lambda \mu_\lambda$ is even in this
    case.

\item For $X$ proper smooth, the fact that $\mu_{\sqrt{q^n}}$ and
    $\mu_{-\sqrt{q^n}}$ are even for $n$ odd follows from a theorem of Suh
    \cite[Theorem 3.3.1]{Suh}.
\end{enumerate}
\end{sremark}

\begin{sremark}
The first two statements of Corollary \ref{c.IH} are consistent with the
conjectural semisimplicity of the Frobenius action on
$\IH^n(X_{\Fqb},\Q_\ell)$ (namely, $\mu_{\lambda,e}=0$ for $e\ge 2$), which
would follow from the standard conjectures. To see this implication, let
$X'\to X$ be a surjective generically finite morphism such that $X'$ is
projective smooth over $\F_q$, which exists by de Jong's alterations
\cite[Theorem 4.1]{dJ}. Then $\IH^n(X_{\Fqb},\Q_\ell)$ as a
$\Gal(\Fqb/\F_q)$-module is a direct summand of
$\rH^n(X'_{\Fqb},\Q_\ell)$.\footnote{This argument is also used in Gabber's
proof of the integrality of \eqref{e.mu}.} The semisimplicity of the
Frobenius action on $\rH^n(X'_{\Fqb},\Q_\ell)$ would follow from the
Lefschetz type standard conjecture for $X'$ and the Hodge type standard
conjecture for $X'\times X'$ \cite[Theorem 5.6 (2)]{Kleiman}.
\end{sremark}

To prove Theorem \ref{t.i1}, we may assume that $k$ is finitely generated
over its prime field. We will keep this assumption in the rest of the
introduction. This includes notably the case of a number field. We deduce
Theorem \ref{t.i1} from a relative result with coefficients. In the case
where $k$ is a finite field, the coefficients are pure perverse sheaves. In
the general case, we apply the formalism of pure \emph{horizontal} perverse
sheaves of Annette Huber \cite{Huber}, as extended by Sophie Morel
\cite{Morel}. Since the proofs are the same in the two cases, we recommend
readers not familiar with horizontal perverse sheaves to concentrate on the
case of a finite field and to ignore the word ``horizontal''. Unless
otherwise stated, we will only consider the middle perversity. We let $\Qlb$
denote the algebraic closure of $\Q_\ell$.

\begin{sdefinition}\label{d.orth}
Let $X$ be a scheme of finite type over $k$ and let $A\in \rD^b_c(X,\Qlb)$
be a horizontal perverse sheaf on $X$, pure of weight $w$. We say that $A$
is \emph{orthogonal}, if there exists a symmetric perfect pairing $A\otimes
A\to K_X(-w)$. We say that $A$ is \emph{symplectic}, if there exists an
alternating perfect pairing $A\otimes A\to K_X(-w)$.
\end{sdefinition}

Here $K_X=\rR a_X^! \Qlb$ is the dualizing complex on $X$, where $a_X\colon
X\to \Spec(k)$ is the structural morphism.

\begin{stheorem}[Special case of Theorem \ref{t.mainh}]\label{t.i2}
Let $f\colon X\to Y$ be a proper morphism of schemes of finite type over $k$
and let $A\in \rD^b_c(X,\Qlb)$ be an orthogonal (resp.\ symplectic) pure
horizontal perverse sheaf on $X$. Then
\begin{equation}\label{e.decomp}
\rR f_*A\simeq \bigoplus_n(\pR^n f_* A)[-n]
\end{equation}
and $\pR^n f_* A$ is orthogonal (resp.\ symplectic) for $n$ even and
symplectic (resp.\ orthogonal) for $n$ odd.
\end{stheorem}

Recall that the Beilinson-Bernstein-Deligne-Gabber decomposition theorem
\cite[Th\'eor\`eme 5.4.5]{BBD} implies that \eqref{e.decomp} holds after
base change to the algebraic closure $\bar k$ of $k$.

Theorem \ref{t.i1} follows from Theorem \ref{t.i2} applied to the morphism
$a_X\colon X\to \Spec(k)$. Even if one is only interested in Theorem
\ref{t.i1}, our proof leads one to consider the relative situation of
Theorem \ref{t.i2}.

Next we state results for operations that do not necessarily preserve pure
(horizontal) complexes. For a scheme $X$ of finite type over $k$, we let
$\rK_\orth(X,\Qlb)$ denote the subgroup of the Grothendieck group
$\rK(X,\Qlb)$ of $\rD^b_c(X,\Qlb)$ generated by orthogonal pure horizontal
perverse sheaves of even weights and symplectic pure horizontal perverse
sheaves of odd weights.

\begin{stheorem}[Special case of Theorem \ref{t.Kh}]\label{t.i3}
Grothendieck's six operations preserve $\rK_\orth$.
\end{stheorem}

Note that the preservation of $\rK_\orth$ by each of the six operations is
nontrivial. The crucial case turns out to be the preservation by the
extension by zero functor $j_!$ for certain open immersions $j$ (Proposition
\ref{p.NCD}).

As Michel Gros points out, one may compare these theorems to Morihiko
Saito's theory of mixed Hodge modules \cite{Saito}. By definition a mixed
Hodge module admits a weight filtration for which the graded pieces are
\emph{polarizable} pure Hodge modules. One may compare Definition
\ref{d.orth} to polarizable pure Hodge modules.

Let us state a consequence of Theorem \ref{t.i3} on Galois action on
cohomology in the case where $k=\F_q$ is a finite field. Let $X$ be a scheme
of finite type over $\F_q$. The eigenvalues of $\Fr_q$ acting on
$\rH^*(X_{\Fqb},\Q_\ell)$ are $q$-Weil integers\footnote{The integrality is
a special case of \cite[Variante 5.1]{Zint}.} of integral weights. We let
$m_\lambda\in \Z$ denote the multiplicity of the eigenvalue $\lambda$. In
other words, we put
\begin{equation}\label{e.m}
\prod_n\det(1-T\Fr_q\mid\rH^n(X_{\Fqb},\Q_\ell))^{(-1)^{n}}
=\prod_\lambda(1-\lambda T)^{m_\lambda}.
\end{equation}
Applying the theorem to $(a_{X})_*$, where $a_X\colon X\to \Spec(\F_q)$, we
obtain the following.

\begin{scor}\label{c.i3}
Let $X$ be a scheme of finite type over $\F_q$. In the above notation,
$m_{\lambda}=m_{q^w/\lambda}$ for every $q$-Weil integer $\lambda$ of weight
$w$, and for $w$ odd, $m_{\sqrt{q^w}}$ and $m_{-\sqrt{q^w}}$ are even. In
particular, for $w$ odd, $\sum_{\lambda} m_\lambda$ (where $\lambda$ runs
through $q$-Weil integers of weight $w$), namely the dimension of the weight
$w$ part of $\rH^*(X_{\Fqb},\Q_\ell)$, is even.
\end{scor}

Theorem \ref{t.i3} also implies analogues of Corollary \ref{c.i3} for
compactly supported cohomology $\rH^*_c(X,\Q_\ell)$, and, if $X$ is
equidimensional, intersection cohomology $\IH^*(X,\Q_\ell)$ and compactly
supported intersection cohomology $\IH^*_c(X,\Q_\ell)$. In the case of
$\rH^*_c(X,\Q_\ell)$, the analogue of \eqref{e.m} is the inverse of the zeta
function and the analogue of Corollary \ref{c.i3} was established by Suh
\cite[Theorem 3.3.1]{Suh} using rigid cohomology.

\begin{sremark}
Some special cases of Corollary \ref{c.i3} were previously known. By
Gabber's theorem on independence of $\ell$ \cite[Theorem~2]{Gabber},
\eqref{e.m} belongs to $\Q(T)$ and is independent of $\ell$. The fact that
\eqref{e.m} belongs to $\Q(T)$ implies that $m_\lambda=m_{\lambda'}$ for
$\lambda$ and $\lambda'$ in the same $\Gal(\overline{\Q}/\Q)$-orbit. In
particular, $m_\lambda=m_{q^w/\lambda}$ for every $q$-Weil integer $\lambda$
of weight $w$, and, if $q$ is \emph{not} a square and $w$ is odd,
$m_{\sqrt{q^w}}=m_{-\sqrt{q^w}}$ so that $\sum_{\lambda} m_\lambda$ (where
$\lambda$ runs through $q$-Weil integers of weight $w$) is even in this
case.
\end{sremark}

One ingredient in the proof of Theorem \ref{t.i3} is de Jong's alterations.
Note that even for a finite \'etale cover $f\colon X\to Y$, one cannot
recover the parity of an object on $Y$ from the parity of its pullback to
$X$. More precisely, for an element $A$ in the Grothendieck group of mixed
horizontal perverse sheaves on $Y$ such that $f^*A\in\rK_{\orth}$, we do not
have $A\in \rK_{\orth}$ in general. We use equivariant alterations to
compensate for this loss of information. In order to better deal with the
equivariant situation, we will work systematically with Deligne-Mumford
stacks in the main text. We note however that the proofs of Theorems
\ref{t.i1} and \ref{t.i2} (and the corollaries to Theorem \ref{t.i1}) do not
depend on stacks, and readers only interested in these results may, in the
corresponding portions of the text (Sections \ref{s.gp}, \ref{s.pc} and
Subsection \ref{ss.h}), assume every stack to be a scheme.

The paper is organized as follows. In Section~\ref{s.gp}, we study symmetry
of complexes and perverse sheaves over a general field. In
Section~\ref{s.pc}, we study symmetry and decomposition of pure complexes
over a finite field and prove Theorem \ref{t.i2} in this case. In
Section~\ref{s.G}, we study symmetry in Grothendieck groups over a finite
field and prove Theorem \ref{t.i3} in this case. In Section \ref{s.6}, we
study symmetry of horizontal complexes over a general field finitely
generated over its prime field and finish the proof of the theorems. In
Appendix~\ref{s.app}, we collect some general symmetry properties in
categories with additional structures, which are used in the main body of
the paper.

\subsection*{Conventions} Unless otherwise indicated, we let $X$, $Y$, etc.\
denote Deligne-Mumford stacks \textit{of finite presentation} (i.e.\ of
finite type and quasi-separated) over a base field $k$; this rules out
stacks such as $B\Z$. We recall that for schemes, being of finite
presentation over $k$ is the same as being of finite type over $k$.

We let $\ell$ denote a prime number invertible in $k$, and let
$\rD^b_c(X,\Qlb)$ denote the derived category of $\Qlb$-complexes on $X$. We
refer the reader to \cite{Zheng} for the construction of $\rD^b_c(X,\Qlb)$
and of Grothendieck's six operations. We denote by $a_X\colon X\to\Spec(k)$
the structural morphism, by $K_X\coloneqq\rR a_X^!\Qlb$ the dualizing
complex on $X$, and by $D_X$ the dualizing functor
$D_{K_X}\coloneqq\rR\cHom(-,K_X)$.

As mentioned above, we will only consider the middle perversity unless
otherwise stated. We let $\Perv(X,\Qlb)\subseteq\rD^b_c(X,\Qlb)$ denote the
full subcategory of perverse $\Qlb$-sheaves on $X$. For a separated
quasi-finite morphism $f\colon X\to Y$, the middle extension functor
$f_{!*}\colon\Perv(X,\Qlb)\to \Perv(Y,\Qlb)$ is the image of the
support-forgetting morphism $\pH^0 f_! \to \pH^0 \rR f_*$.

Throughout the article we let $\sigma$ and $\sigma'$ represent elements of
$\{\pm 1\}$.

\subsection*{Acknowledgments}
We thank Luc Illusie for encouragement and enlightening conversations. We
thank Ofer Gabber for many helpful suggestions, and we are grateful to him
and G\'erard Laumon for pointing out a mistake in an earlier draft of this
paper. We are indebted to Takeshi Saito for fruitful suggestions on general
base fields, to Pierre Deligne for generously sharing his knowledge on
$\lambda$-rings, to Jiangxue Fang for bringing our attention to Kashiwara's
conjecture, and to Matthew Young for suggesting a connection to
Grothendieck-Witt groups. The first-named author thanks Torsten Ekedahl,
Bernd Ulrich, and many others for helpful discussions on the MathOverflow
website. The second-named author thanks Michel Brion, Zongbin Chen, Lei Fu,
Michel Gros, Binyong Sun, Yichao Tian, Claire Voisin, Liang Xiao, and Zhiwei
Yun for useful discussions. Part of this work was done during various stays
of the authors at Universit\'e Paris-Sud, Institut des Hautes \'Etudes
Scientifiques, Korea Institute for Advanced Studies, Shanghai Jiao Tong
University, and Hong Kong University of Science and Technology. We thank
these institutions for hospitality and support. We thank the referees for
careful reading of the manuscript and for many useful comments.

\section{Symmetry of complexes and perverse sheaves}\label{s.gp}
In this section, we study symmetry properties of $\Qlb$-complexes, namely
objects of $\rD^b_c(X,\Qlb)$, over an arbitrary field $k$. In Subsection
\ref{ss.gp1}, we define $\sigma$-self-dual complexes and we study their
behavior under operations that commute with duality.  In Subsection
\ref{ss.gp2}, we analyze $\sigma$-self-dual semisimple perverse sheaves. In
this generality none of the results is difficult, but they will be used
quite often in the sequel.

\subsection{Symmetry of complexes}\label{ss.gp1}
Tensor product endows $\rD^b_c(X,\Qlb)$ with the structure of a closed
symmetric monoidal structure. The definition below only makes use of the
symmetry constraint $c_{AB}\colon A\otimes B\simto B\otimes A$ and the
internal mapping object $\rR\cHom$.

\begin{definition}[$\sigma$-self-dual complexes]
Let $A,C\in \rD^b_c(X,\Qlb)$. We say that $A$ is \emph{$1$-self-dual with
respect to $C$} (resp.\ \emph{$-1$-self-dual with respect to $C$}) if there
exists a pairing $A\otimes A\to C$ that is
\begin{itemize}
\item \emph{symmetric} (resp.\ \emph{alternating}), in the sense that the
    following diagram commutes (resp.\ anticommutes)
\begin{equation}\label{e.sym}
\xymatrix@C=1cm{ A\otimes A \ar[rr]^{c_{AA}} \ar[rd] && A\otimes A
\ar[ld] \\
& C, &}
\end{equation}
\item and \emph{perfect}, in the sense that the pairing induces an
    isomorphism $A \simto D_CA\coloneqq\rR\cHom(A,C)$.
\end{itemize}
We say that $A$ is \textit{self-dual with respect to $C$} if there exists
    an isomorphism $A\simto D_CA$.
\end{definition}

The symmetry of the pairing $A\otimes A\to C$ can also be expressed in terms
of the induced morphism $f\colon A\to D_C A$. In fact, the diagram
\eqref{e.sym} $\sigma$-commutes if and only if the diagram
\[
\xymatrix@C=1cm{
A \ar[rr]^f \ar[rd]_-{\text{ev}} && D_CA \\
& D_CD_CA \ar[ur]_-{D_Cf}&}
\]
$\sigma$-commutes (Lemma \ref{lemma2.1.12}).

\begin{remark}
Similarly one can define self-dual and $\sigma$-self-dual
$E_\lambda$-complexes, where $E_\lambda$ is any algebraic extension of
$\Q_\ell$. Note that for $A,C\in \rD^b_c(X,E_\lambda)$, $A$ is self-dual
(resp.\ $\sigma$-self-dual) with respect to $C$ if and only if
$A\otimes_{E_\lambda} \Qlb$ satisfies the same property with respect to
$C\otimes_{E_\lambda} \Qlb$. Indeed, the ``only if'' part is obvious. To see
the ``if'' part, consider $U=\bIsom(A_{\bar k},D_C A_{\bar k})$ as in Lemma
\ref{l.Zar} below. Here $A_{\bar k}$ denotes the pullback of $A$ to $X_{\bar
k}$. Recall that rational points form a Zariski dense subset of any affine
space over an infinite field. If $A\otimes_{E_\lambda} \Qlb$ is self-dual
(resp.\ $\sigma$-self-dual) with respect to $C\otimes_{E_\lambda} \Qlb$,
then $U\cap V$ is nonempty, hence has a $E_\lambda$-point. Here $V\subseteq
\bHom(A_{\bar k},D_C A_{\bar k})$ is represented by the $E_\lambda$-vector
subspace, image of morphisms (resp.\ $\sigma$-symmetric morphisms) $A\to D_C
A$. For the above reason, we will work almost exclusively with
$\Qlb$-complexes.
\end{remark}

\begin{lemma}\label{l.Zar}
Let $A,B\in \rD^b_c(X,E_\lambda)$ such that
$\dim_{E_\lambda}\Hom(A,B)<\infty$. Then there exists a Zariski open
subscheme $U=\bIsom(A,B)$ of the affine space $\bHom(A,B)$ over $E_\lambda$
represented by the $E_\lambda$-vector space $\Hom(A,B)$ such that for any
algebraic extension $E'_\lambda$ of $E_\lambda$, $U(E'_\lambda)$ corresponds
to the set of isomorphisms $A\otimes_{E_\lambda} E'_\lambda\simto
B\otimes_{E_\lambda} E'_\lambda$.
\end{lemma}

\begin{proof}
Assume $A,B\in \rD^{[a,b]}$. Choose a stratification of $X$ by connected
geometrically unibranch substacks such that the restrictions of $\cH^n A$
and $\cH^n B$ to each stratum are lisse sheaves. Choose a geometric point
$x$ in each stratum. Then a morphism $f\colon A\otimes_{E_\lambda}
E'_\lambda \to B\otimes_{E_\lambda} E'_\lambda$ is an isomorphism if and
only if $\cH^n f_x$ is for every $n\in [a,b]$ and every $x$ in the finite
collection. Then $\bIsom(A,B)$ is the intersection of the pullbacks of the
open subsets $\bIsom(\cH^n A_x,\cH^n B_x)\subseteq \bHom(\cH^n A_x,\cH^n
B_x)$.
\end{proof}

We will mostly be interested only in duality with respect to Tate twists
$K_X(-w)$, $w\in \Z$ of the dualizing complex $K_X=\rR a_X^!\Qlb$. In this
case, the evaluation morphism $A\to D_{K_X(-w)} D_{K_X(-w)} A$ is an
isomorphism. The functor $D_{K_X(-w)}$ preserves perverse sheaves. We will
sometimes write $K$ for $K_X$ when no confusion arises.

In the rest of this subsection, we study the behavior of $\sigma$-self-dual
complexes under operations that commute with the dualizing functors (up to
shift and twist). The results are mostly formal, but for completeness we
provide a proof for each result, based on general facts on symmetry in
categories collected in Appendix \ref{s.app}.  Readers willing to accept
these results may skip the proofs.

Most of the proofs consist of showing that the natural isomorphism
representing the commutation of the functor in question with duality is
symmetric in the sense of Definition \ref{d.symF}. It then follows from
Lemma \ref{l.selfdual} that the functor in question preserves
$\sigma$-self-dual objects.

\begin{remark}[Preservation of $\sigma$-self-dual complexes]\label{p.sd}
Let $f\colon X\to Y$ be a morphism. Let $w,w'\in \Z$.
\begin{enumerate}
\item \emph{For $n\in \Z$, Tate twist $A\mapsto A(n)$ carries
    $\sigma$-self-dual $\Qlb$-complexes with respect to $C$ to
    $\sigma$-self-dual $\Qlb$-complexes with respect to $C(2n)$; the shift
    functor $A\mapsto A[n]$ carries $\Qlb$-complexes $\sigma$-self-dual
    with respect to $C$ to $(-1)^n\sigma$-self-dual $\Qlb$-complexes with
    respect to $C[2n]$.}

This follows from Lemma \ref{l.trans}.

\item \emph{$D_X$ carries $\sigma$-self-dual $\Qlb$-complexes with respect
    to $K_X(-w)$ to $\sigma$-self-dual $\Qlb$-complexes with respect to
    $K_X(w)$.}

Since $D_X A\simeq (D_{K_X(-w)} A)(w)$, the assertion follows from (1).

\item \emph{Assume that $f$ is proper. Then $\rR f_*\colon
    \rD^b_c(X,\Qlb)\to \rD^b_c(Y,\Qlb)$ preserves $\sigma$-self-dual
    objects with respect to $K(-w)$. In other words, $\rR f_*$ carries
    $\Qlb$-complexes $\sigma$-self-dual with respect to $K_X(-w)$ to
    $\Qlb$-complexes $\sigma$-self-dual with respect to $K_Y(-w)$.}

Since $\rR f_*$ is a right-lax symmetric functor (Definition \ref{d.sf}),
the morphism $\rR f_* D_X\to D_{\rR f_* K_X}\rR f_*$ is symmetric by
Construction \ref{c.image}. Composing with the adjunction map $\rR f_* K_X
\simeq \rR f_! K_X \to K_Y$, we obtain a symmetric isomorphism $\rR
f_*D_X\simto D_Y \rR f_*$.

\item \emph{Assume that $f$ is a closed immersion and let
    $A\in\rD^b_c(X,\Qlb)$. Then $A$ is $\sigma$-self-dual with respect to
    $K_X(-w)$ if and only if $f_*A$ is $\sigma$-self-dual with respect to
    $K_Y(-w)$.}

Since the functor $f_*$ is fully faithful in this case, the assertion
follows from the proof of (3) above and Lemma \ref{l.selfdual}.

\item \emph{Assume that $f$ is an open immersion and let $A\in
    \Perv(X,\Qlb)$ be a perverse sheaf. Then $A$ is $\sigma$-self-dual
    with respect to $K_X(-w)$ if and only if $f_{!*}A$ is
    $\sigma$-self-dual with respect to $K_Y(-w)$.}

Since $f_!$ is a symmetric functor, the morphism $f_! D_X\to D_{f_!
K_X}f_!$ is symmetric by Construction \ref{c.image}. It follows that the
composite map in the commutative square
\[\xymatrix{f_! D_X\ar[d]\ar[r]^\alpha &\rR f_* D_X\ar[d]^\beta_\simeq \\
D_{f_! K_X}f_!\ar[r] & D_Y f_!}
\]
is symmetric. Here the lower horizontal map is given by the adjunction map
$f_! K_X\to K_Y$. Moreover, we have a commutative square
\[\xymatrix{f_{!*}D_X
    \ar[r]^\gamma_\sim\ar[d] & D_Y f_{!*}\ar[d]\\
    \rR \pR^0 f_* D_X\ar[r]^{\pH^0\beta}_\sim &D_Y \pR^0 f_!}
\]
By Lemma \ref{l.inter}, $\gamma$ is symmetric. Since the functor $f_{!*}$
is fully faithful by Lemma \ref{l.perverse} below, it then suffices to
apply Lemma \ref{l.selfdual}. The ``if'' part also follows from (8) below.

\item \emph{Let $A\in \rD^b_c(X,\Qlb)$ be $\sigma$-self-dual with respect
    to $K_X(-w)$, and let $B\in\rD^b_c(X',\Qlb)$ be $\sigma'$-self-dual
    with respect to $K_{X'}(-w')$. Then the exterior tensor product
    $A\boxtimes B\in \rD^b_c(X\times X',\Qlb)$ is $\sigma
    \sigma'$-self-dual with respect to $K_{X\times X'}(-w-w')$.}

Since $-\boxtimes -$ is a symmetric monoidal functor, the K\"unneth
isomorphism (cf.\ \cite[(1.7.6), Proposition 2.3]{SGA5III})
    \[D_X(-)\boxtimes D_{X'}(-) \simto D_{K_X\boxtimes K_{X'}}
    (-\boxtimes -)\simeq D_{X\times X'}(-\boxtimes -)
    \]
    is symmetric by Construction \ref{c.image}.

\item \emph{If $f$ is smooth, purely of dimension $d$, then $f^*[d]$
    carries $\Qlb$-complexes $\sigma$-self-dual with respect to $K_Y(-w)$
    to $\Qlb$-complexes $(-1)^d\sigma$-self-dual with respect to
    $K_X(-d-w)$.}

Since $f^*$ is a symmetric monoidal functor, the isomorphism
\[f^*T^d D_{Y}\simto D_{f^*K_Y [2d]} f^* T^d\simeq D_{K_X(-d)} f^* T^d
\]
is $(-1)^d$-symmetric by Construction \ref{c.image} and Remark
\ref{r.shift}.

\item \emph{If $X$ and $Y$ are regular, purely of dimension $d$ and $e$
    respectively, then $f^*[d-e]\colon \rD^b_{\lisse}(Y,\Qlb)\to
    \rD^b_{\lisse}(X,\Qlb)$ carries $\Qlb$-complexes $\sigma$-self-dual
    with respect to $K_Y(-w)$ to $\Qlb$-complexes
    $(-1)^{d-e}\sigma$-self-dual with respect to $K_X(-(d-e)-w)$. Here
    $\rD^b_{\lisse}$ denotes the full subcategory of $\rD^b_c$ consisting
    of complexes with lisse cohomology sheaves.}

Let $r=d-e$. As in (7), the natural transformation
\[f^*T^{r} D_Y \to D_{f^*K_Y[2r]}f^*T^{r}\simeq D_{K_X(-r)}f^*T^{r}\]
is $(-1)^r$-symmetric. For $A\in \rD^b_c(Y,\Qlb)$, this natural
transformation can be computed as
\[f^*T^{r} D_Y A\simeq f^*T^{r} D_{K_Y(-r)} A\otimes \rR f^! \Qlb \xrightarrow{\alpha}
\rR f^!D_{K_Y(-r)}T^{r} A\simeq D_{K_X(-r)}f^*T^{r} A.
\]
For $A\in \rD^b_{\lisse}$, $\alpha$ is an isomorphism.

\end{enumerate}
Similar results hold for self-dual $\Qlb$-complexes.
\end{remark}

\begin{lemma}\label{l.perverse}
Let $j\colon U\to X$ be an immersion. Then the functor $j_{!*}\colon
\Perv(U,\Qlb)\to \Perv(X,\Qlb)$ is fully faithful.
\end{lemma}

\begin{proof}
Let $A$ and $B$ be perverse $\Qlb$-sheaves on $U$, and let $\alpha\colon
\Hom(A,B)\to \Hom(j_{!*}A,j_{!*}B)$ be the map induced by $j_{!*}$. As the
composite map
\[\Hom(A,B)\xrightarrow{\alpha} \Hom(j_{!*}A,j_{!*}B)\xrightarrow{\beta} \Hom(\pH^0 j_! A,\pH^0 Rj_* B)\simeq \Hom(j_!A,Rj_*B)
\]
is an isomorphism and $\beta$ is an injection, $\alpha$ is an isomorphism.
\end{proof}

\begin{example}\label{e.IC}
Assume that $X$ is equidimensional. We define the \textit{intersection
complex} of $X$ by $\IC_X=j_{!*}(\Q_\ell[d])$, where $j\colon U\to X$ is a
dominant open immersion such that $U_\red$ is regular and $d=\dim X$. Then
by Parts (5) and (7) of Remark \ref{p.sd}, $\IC_X$ is $(-1)^d$-self-dual
with respect to $K_X(-d)$.
\end{example}

Although we do not need it, let us mention the following stability of
$\sigma$-self-dual complexes under nearby cycles.

\begin{remark}
\emph{Let $S$ be the spectrum of a Henselian discrete valuation
    ring, of generic point $\eta$ and closed point $s$, on which $\ell$ is
    invertible. Let $\mathfrak X$ be a Deligne-Mumford stack of finite
    presentation over $S$. Then the nearby cycle functor $\rR\Psi\colon
    \rD^b_c(\mathfrak X_\eta,\Qlb)\to\rD^b_c(\mathfrak X_s\times_s
    \eta,\Qlb)$ preserves $\sigma$-self-dual objects with respect to
    $K(-w)$.}

Indeed, $\rR\Psi$ is a right-lax symmetric monoidal functor. Hence, by
Construction \ref{c.image}, the composite $\rR\Psi D_{\mathfrak X_\eta}\to
D_{\rR\Psi K_{\mathfrak X_\eta}} \rR\Psi \to D_{\mathfrak X_s} \rR\Psi$,
which is a natural isomorphism (cf.\ \cite[Th\'eor\`eme 4.2]{Illusie}), is
symmetric.
\end{remark}

Remark \ref{p.sd} (6) can be applied to the exterior tensor power functor
$(-)^{\boxtimes m}\colon \rD^b_c(X,\Qlb)\to \rD^b_c(X^m,\Qlb)$, $m\ge 0$. We
now discuss a refinement
\begin{equation}\label{eq.boxpower}
    \rD^b_c(X,\Qlb)\to
    \rD^b_c([X^m/\Perm_m],\Qlb),\quad A\mapsto A^{\boxtimes m}
\end{equation}
given by permutation. Readers not interested in this refinement may skip
this part as it will not be used in the proofs of the results mentioned in
the introduction.

We briefly recall one way to define the symmetric product stack
$[X^m/\Perm_m]$. For every $k$-scheme $S$, $[X^m/\Perm_m](S)$ is the
groupoid of pairs $(T,x)$, where $T$ is a finite \'etale cover of $S$ of
degree $m$ and $x$ is an object of $X(T)$, with isomorphisms of pairs
defined in the obvious way.

\begin{remark}
\emph{The functor \eqref{eq.boxpower} carries $\sigma$-self-dual complexes
with respect to $K_X(-w)$ to $\sigma^m$-self-dual complexes with respect to
$K_{[X^m/\Perm_m]}(-mw)$.}

Indeed, $(-)^{\boxtimes m}$ is a symmetric monoidal functor. Hence, the
    isomorphism
    \[(D_{X} (-))^{\boxtimes m}\simto D_{K_X^{\boxtimes m}}
    ((-)^{\boxtimes m})\simeq D_{[X^m/\Perm_m]}((-)^{\boxtimes m})
    \]
    is symmetric by Construction \ref{c.image}.
\end{remark}

\subsection{Symmetry of perverse sheaves}\label{ss.gp2}
In this subsection, we study $\sigma$-self-dual perverse sheaves. We first
prove a two-out-of-three property, which will play an important role in
later sections. We then discuss a trichotomy for indecomposable perverse
$\Qlb$-sheaves. From this we deduce a criterion for semisimple perverse
$\Qlb$-sheaves to be $\sigma$-self-dual in terms of multiplicities of simple
factors.

\begin{prop}[Two-out-of-three]\label{p.23}
Let $A$ be a perverse $\Qlb$-sheaf such that $A\simeq A'\oplus A''$. If $A$
and $A'$ are $\sigma$-self-dual with respect to $K_X(-w)$, then so is $A''$.
\end{prop}

\begin{proof}
We write $D$ for $D_{K_X(-w)}$. Let $f\colon A\simto DA$ and $g'\colon
A'\simto D A'$ be $\sigma$-symmetric isomorphisms. We let $f'\colon A'\to D
A'$ denote the restriction of $f$ to $A'$, namely the composite
\[A'\xrightarrow{i} A\xrightarrow[\sim]{f} DA \xrightarrow{Di} DA',\]
where $i\colon A'\to A$ is the inclusion. Let $g\colon A\to DA$ be the
direct sum of $g'$ with the zero map $A''\to DA''$. These are
$\sigma$-symmetric morphisms. Consider linear combinations $h_{a,b}=af+bg$
and $h'_{a,b}=af'+bg'$, where $a,b\in \Qlb$. By Lemma \ref{l.Zar}, there are
only finitely many values of $(a:b)$ for which $h_{a,b}$ is not an
isomorphism. The same holds for $h'_{a,b}$. Therefore, there exist $a,b\in
\Qlb$ such that $h_{a,b}$ and $h'_{a,b}$ are isomorphisms. Consider the
orthogonal complement of $A'$ in $A$ with respect to $h_{a,b}$:
\[B=\Ker(A\xrightarrow[\sim]{h_{a,b}}DA \xrightarrow{Di} DA').\]
Then $h_{a,b}$ induces a $\sigma$-symmetric isomorphism $B\simto DB$.
Moreover, $A\simeq A'\oplus B$, so that $B\simeq A''$. Here we used the
Krull-Schmidt theorem \cite[Theorem 1]{Atiyah} and the fact that perverse
sheaves have finite lengths.
\end{proof}

\begin{remark}\label{r.23d}
The two-out-of-three property also holds more trivially for self-dual
complexes. In fact, if we have decompositions of perverse $\Qlb$-sheaves
$A\simeq A'\oplus A''$ and $B\simeq B'\oplus B''$ such that $A\simeq D_X B
(-w)$ and $A'\simeq D_X B'(-w)$, then we have $A''\simeq D_X B''(-w)$ by the
Krull-Schmidt theorem.
\end{remark}

\begin{prop}[Trichotomy]\label{p.schur}
Let $A$ be an indecomposable perverse $\Qlb$-sheaf on $X$. Then exactly one
of the following occurs:
\begin{itemize}
\item $A$ is $1$-self-dual with respect to $K_X(-w)$;
\item $A$ is $-1$-self-dual with respect to $K_X(-w)$;
\item $A$ is not self-dual with respect to $K_X(-w)$.
\end{itemize}
\end{prop}

This follows from general facts (Lemma \ref{l.KS} and Remark \ref{r.KS})
applied to the category of perverse $\Qlb$-sheaves.

\begin{remark}
In the case of a simple perverse $\Qlb$-sheaf, the proof can be somewhat
simplified with the help of Schur's lemma. This case is analogous to a
standard result on complex representations of finite or compact groups
(\cite[Section 13.2, Proposition 38]{SerreRL}, \cite[Proposition
II.6.5]{BD}).
\end{remark}

\begin{remark}
An indecomposable perverse $E_\lambda$-sheaf on $X$, self-dual with respect
to $K_X(-w)$, is either $1$-self-dual or $-1$-self-dual, by Lemma \ref{l.KS}
and Remark \ref{r.KS}. Note that a simple perverse $E_\lambda$-sheaf can be
$1$-self-dual and $-1$-self-dual with respect to $K_X(-w)$ at the same time.
\end{remark}

\begin{cor}\label{c.even}
Let $A\simeq \bigoplus_B B^{n_B}$ be a semisimple perverse $\Qlb$-sheaf on
$X$, where $B$ runs through isomorphism classes of simple perverse
$\Qlb$-sheaves on $X$. Then $A$ is $\sigma$-self-dual with respect to
$K_X(-w)$ if and only if the following conditions hold:
\begin{enumerate}
\item $n_B=n_{(D_X B)(-w)}$ for $B$ not self-dual with respect to
    $K_X(-w)$;
\item $n_B$ is even for $B$ that are $-\sigma$-self-dual with respect to
    $K_X(-w)$.
\end{enumerate}
Moreover, $A$ is self-dual with respect to $K_X(-w)$ if and only if (1)
holds.
\end{cor}

In particular, if $B$ and $B'$ are respectively $1$-self-dual and
$-1$-self-dual simple perverse sheaves on $X$, then $B\oplus B'$ is
self-dual but neither $1$-self-dual nor $-1$-self-dual.

\begin{proof}
It is clear that (1) is equivalent to the condition that $A$ is self-dual.
If (1) and (2) hold, then $A$ is $\sigma$-self-dual by Proposition
\ref{p.schur} and the fact that $B\oplus (D_X B)(-w)$ is $\sigma$-self-dual
with respect to $K_X(-w)$ for all $B$ (Remark \ref{r.selfd} (2)). It remains
to show that if $A$ is $\sigma$-self-dual, then (2) holds. Let $B$ be
$-\sigma$-self-dual. As $B^{\oplus n_B}\simeq B\otimes V$ is
$\sigma$-self-dual, where $V=\Qlb^{\oplus n_B}$, the isomorphism
\[\Hom(B\otimes V, D_X(B\otimes
V)(-w))\simeq \Hom(B,(D_X B)(-w))\otimes \Hom(V,V^*) \simeq \Hom(V,V^*)\]
provides a skew-symmetric $n_B\times n_B$ matrix with entries in $\Qlb$,
which implies that $n_B$ is even. More formally, we can apply the second
part of Lemma \ref{l.selfdual} to the fully faithful functor $F\colon
V\mapsto B\otimes V$ from the category of finite-dimensional $\Qlb$-vector
spaces to $\rD^b_c(X,\Qlb)$. The natural isomorphism $FD\simto D_{K_X(-w)}F$
is $-\sigma$-self-dual.
\end{proof}

\begin{remark}\label{r.ss}
The semisimplification of a $\sigma$-self-dual perverse sheaf is
$\sigma$-self-dual by Lemma \ref{l.ss}. The converse does not hold. See
Example \ref{e.ss} below.
\end{remark}

We will need to consider more generally geometrically semisimple perverse
sheaves, namely perverse sheaves whose pullbacks to $X_{\bar k}$ are
semisimple. Part (1) of the following lemma extends \cite[Corollaire
5.3.11]{BBD} for pure perverse sheaves.

\begin{lemma}\label{l.indec}
Let $A$ be a geometrically semisimple perverse $\Qlb$-sheaf on $X$.
\begin{enumerate}
\item Let $i\colon Y\to X$ be a closed immersion with complementary open
    immersion $j\colon U\to X$. Then $A$ admits a unique decomposition
    $A\simeq j_{!*}j^*A \oplus i_*B$, where $B$ is a perverse sheaf on
    $Y$. Moreover, we have $B\simeq \pH^0 i^* A\simeq \pR^0 i^! A$.

\item $A$ admits a unique decomposition $A\simeq \bigoplus_V A_V$, where
    $V$ runs through irreducible closed substacks of $X$, and the support
    of each indecomposable direct summand of the perverse sheaf $A_V$ is
    $V$.
\end{enumerate}
\end{lemma}

Assume moreover that $A$ is indecomposable. Then, by Part (1) of the lemma,
we have $j_{!*}j^*A\simeq A$ if $U$ intersects with the support of $A$ (and
$j^*A=0$ otherwise). Moreover, the support of $A$ is irreducible, and $A$ is
isomorphic to $f_{!*}(\cF[d])$ for some immersion $f\colon W\to X$ with $W$
regular irreducible of dimension $d$, and some lisse $\Qlb$-sheaf $\cF$ on
$V$.

\begin{proof}
(1) The proof is identical to that of \cite[Corollaire 5.3.11]{BBD}. The
uniqueness of the decomposition is clear. For the existence, it suffices to
check that
\begin{itemize}
\item the adjunction map $\pH^0 j_!j^*A\to A$ factorizes through the
    quotient $j_{!*} j^*A$ of $\pH^0 j_!j^*A$, and the adjunction map
    $A\to \pR^0 j_* j^* A$ factorizes through the sub-object $j_{!*} j^*A$
    of $\pR^0 j_* j^* A$;
\item the composite of the adjunction maps $i_*\pR^0 i^! A\to A \to i_*
    \pH^0 i^* A$ is an isomorphism;
\end{itemize}
and these maps provide a decomposition of $A$. These statements can be
easily checked over $\bar k$.

(2) Again the uniqueness is clear. The existence follows from the fact that
the support of every indecomposable direct summand of $A$ is irreducible.
\end{proof}

\begin{remark}
In Lemma \ref{l.indec}, $A$ is $\sigma$-self-dual with respect to $K_X(-w)$
if and only if each direct summand $A_V$ in the support decomposition is
$\sigma$-self-dual with respect to $K_X(-w)$.
\end{remark}

As an application, we show that for geometrically semisimple perverse
sheaves the property of being $\sigma$-self-dual is local for the Zariski
topology. This Zariski local nature will be useful in Section \ref{s.G}.

\begin{prop}\label{p.cover}
Let $(X_\alpha)_{\alpha\in I}$ be a Zariski open covering of $X$, and let
$A$ and $B$ be geometrically semisimple perverse $\Qlb$-sheaves on $X$. Then
$A\simeq (D_X B)(-w)$ if and only if $A\res{X_\alpha}\simeq (D_{X_\alpha}
B\res{X_\alpha})(-w)$ for every $\alpha\in I$. Moreover, $A$ is
$\sigma$-self-dual with respect to $K_X(-w)$ if and only if
$A\res{X_\alpha}$ is so with respect to $K_{X_\alpha}(-w)$ for every
$\alpha\in I$.
\end{prop}

\begin{proof}
We prove the second assertion, the proof of the first assertion being
simpler. It suffices to show the ``if'' part. Let $j_\alpha\colon
X_\alpha\to X$. By Parts (5) and (8) of Remark \ref{p.sd},
$j_{\alpha!*}j_\alpha^*A$ is $\sigma$-self-dual. Since
$j_{\alpha!*}j_\alpha^*A\simeq \bigoplus_{V} A_V$, where $V$ satisfies
$V\cap X_\alpha\neq \emptyset$, we conclude that each $A_V$ is
$\sigma$-self-dual.

Alternatively we may apply Lemma \ref{l.cover} below. Indeed, by
quasi-compactness, we may assume that $I$ is finite. For $J\subseteq I$
nonempty, $j_{J!*}j_J^* A$ is $\sigma$-self-dual. Thus the same holds for
$A\simeq j_{\emptyset!*}j_{\emptyset}^*A$ by the two-out-of-three property.
\end{proof}

\begin{lemma}\label{l.cover}
Let $(X_\alpha)_{\alpha\in I}$ be a finite Zariski open covering of $X$, and
let $A$ be a geometrically semisimple perverse $\Qlb$-sheaf on $X$. Then
    \[\bigoplus_{\substack{J\subseteq I\\\text{$\#J$ \emph{even}}}} j_{J!*}j_J^* A  \simeq
    \bigoplus_{\substack{J\subseteq I\\\text{$\#J$ \emph{odd}}}} j_{J!*}j_J^* A,
    \]
where $j_J\colon \bigcap_{\alpha\in J}X_\alpha\to X$ is the open immersion.
\end{lemma}

\begin{proof}
We may assume that $A$ is indecomposable. Then both sides are direct sums of
copies of $A$ and the multiplicities are equal:
\[\sum_{\substack{0\le i\le m\\\text{$i$ even}}} \binom{m}{i}=\sum_{\substack{0\le i\le m\\\text{$i$ odd}}} \binom{m}{i}.\]
Here $m\ge 1$ is the number of indices $\alpha\in I$ such that the support
of $A$ intersects with $X_\alpha$.
\end{proof}

\section{Symmetry and decomposition of pure complexes}\label{s.pc}

In this section, we study symmetry of pure perverse sheaves and more
generally of pure complexes that decompose into shifts of perverse sheaves.
We first work over a finite field. In Subsection \ref{ss.pp}, we analyze
$\sigma$-self-dual pure perverse sheaves and give a criterion in terms of
multiplicities of factors. In Subsection \ref{ss.pc}, we study the behavior
of such perverse sheaves under operations that preserve purity. The main
result of this section is the preservation of certain class of complexes
under derived proper direct image (Theorem \ref{t.mainf}), which implies the
finite field case of Theorem \ref{t.i2}. Such preservation results
constitute the starting point of the analysis in Section \ref{s.G} of the
effects of more general operations in the mixed case. In Subsection
\ref{s.var}, we work over a separably closed base field and we prove
preservation results for certain semisimple complexes, by reducing to the
finite field case.

\subsection{Symmetry of pure perverse sheaves over a finite field}\label{ss.pp}

In this subsection and the next we work over a finite field $k=\F_q$. Recall
that $X$, $Y$, etc.\ denote Deligne-Mumford stacks of finite presentation
over $\F_q$. Let $\iota\colon \Qlb\to \C$ be an embedding.

In this subsection, we study $\sigma$-self-dual $\iota$-pure perverse
sheaves. We give a criterion for $\iota$-pure perverse sheaves to be
$\sigma$-self-dual in terms of multiplicities of factors.

For $n\ge 1$, let $E_n$ be the sheaf on $\Spec(\F_q)$ of stalk
$(\Qlb)^n=\bigoplus_{i=1}^n\Qlb e_i$ on which Frobenius $F=\Fr_q$ acts
unipotently with one Jordan block: $F{e_1}=e_1$ and $Fe_i=e_i+e_{i-1}$ for
$i>1$. Recall that any indecomposable $\iota$-pure perverse sheaf $A$ on $X$
is isomorphic to a perverse sheaf of the form $B\otimes a_X^*E_n$, where $B$
is a simple perverse sheaf on $X$, $n\ge 1$, and $a_X\colon X\to
\Spec(\F_q)$.

\begin{prop}\label{p.Jordan}
Let $A$ be a perverse $\Qlb$-sheaf on $X$, isomorphic to $\bigoplus_B
(B\otimes a_X^* E_n)^{m_{B,n}}$, where $B$ runs over simple perverse
$\Qlb$-sheaves on $X$, and $w\in \Z$. Then $A$ is $\sigma$-self-dual with
respect to $K_X(-w)$ if and only if the following conditions hold:
\begin{enumerate}
\item $m_{B,n}=m_{(D_XB)(-w),n}$ for $B$ not self-dual with respect to
    $K_X(-w)$;
\item $m_{B,n}$ is even for $B$ $\sigma$-self-dual with respect to
    $K_X(-w)$ and $n$ even;
\item $m_{B,n}$ is even for $B$ $-\sigma$-self-dual with respect to
    $K_X(-w)$ and $n$ odd.
\end{enumerate}
Moreover, $A$ is self-dual with respect to $K_X(-w)$ if and only if (1)
holds.
\end{prop}

\begin{proof}
The equivalence between (1) and the condition that $A$ is self-dual follows
from the isomorphism $D_X(B\otimes a_X^* E_n)\simeq D_X B\otimes a_X^* E_n$.
For the ``if'' part of the $\sigma$-self-dual case, note that $C\oplus D_X
C(-w)$ is $\sigma$-self dual with respect to $K_X(-w)$, so that by the
trichotomy in Proposition \ref{p.schur}, it suffices to show that $B\otimes
a_X^* E_n$ is $\sigma'$-self-dual (resp.\ $-\sigma'$-self-dual) for $B$
$\sigma'$-self-dual and $n$ odd (resp.\ even). For the ``only if'' part, we
reduce to the case where $m_{B,n}=0$ for all $B$ except for one
$\sigma'$-self-dual $B$. For both parts, consider the functor $F=B\otimes
a_X^*- \colon \Perv(\Spec(\F_q),\Qlb)\to \Perv(X,\Qlb)$, which is fully
faithful by Lemma \ref{l.Hom} below. The natural isomorphism
$FD_{\Spec(\F_q)}\simeq D_{K_X(-w)} F$ is $\sigma'$-symmetric. By Lemma
\ref{l.F}, we are then reduced to the case where $X=\Spec(\F_q)$ and
$B=(\Qlb)_X$, which follows from Lemma \ref{l.la} below.
\end{proof}

\begin{lemma}\label{l.Hom}
Let $B$ be a simple perverse $\Qlb$-sheaf on $X$. Then the functor $B\otimes
a_X^*-$ is fully faithful. In other words, for $\Qlb$-sheaves $E$ and $E'$
on $\Spec(\F_q)$, the map $\alpha\colon \Hom(E,E')\to \Hom(B\otimes a_X^*E,
B\otimes a_X^*E')$ is an isomorphism.
\end{lemma}

\begin{proof}
We have $B\simeq j_{!*}(\cF[d])$, where $j\colon U\to X$ is an immersion,
with $U_{\red}$ connected regular purely of dimension $d$, and where $\cF$
is a simple lisse $\Qlb$-sheaf on $U$. We have $B\otimes a_X^*E\simeq
j_{!*}((\cF\otimes a_U^*E)[d])$. The map $\alpha$ is the composite
\[\Hom(E,E')\xrightarrow{\beta} \Hom(\cF\otimes a_U^*E, \cF\otimes a_U^*E')\xrightarrow{\gamma}\Hom(B\otimes a_X^*E, B\otimes a_X^*E'),\]
where $\gamma$ is an isomorphism by Lemma \ref{l.perverse}. The map $\beta$
is obviously injective. To show that $\beta$ is an isomorphism, we may
assume that $E=E_n$, $E'=E_m$. Since the socle of $\cF\otimes a_U^* E_m$ is
$\cF$,
\[\Hom(\cF,\cF)\simeq \Hom(\cF,\cF\otimes a_U^* E_{m})\]
is one-dimensional. Dually, since the cosocle of $\cF\otimes a_U^* E_n$ is
$\cF$,
\[\Hom(\cF\otimes a_U^* E_{n},\cF)\simeq \Hom(\cF,\cF)\]
is one-dimensional. Thus
\[\dim\Hom(\cF\otimes a_U^*E_n,\cF\otimes a_U^*E_m)\le \min\{n,m\}=\dim\Hom(E_n,E_m).\]
It follows that $\beta$ is an isomorphism.
\end{proof}

\begin{lemma}\label{l.la}
Let $L$ be a field of characteristic $\neq 2$, and let $N_n\in
\mathrm{M}_{n\times n}(L)$ be the matrix defined by $(N_n)_{i,j}=1$ for
$i=j-1$ and $(N_n)_{i,j}=0$ otherwise. Let $m_1,\cdots,m_l\ge0$ be integers,
and let
\[N\coloneqq N(m_1,\dots,m_l)\coloneqq \diag(N_1,\dots,N_1,\dots,N_l,\dots, N_l),\] where each $N_n$ is repeated $m_n$ times. Then there exists an invertible symmetric (resp.\
invertible skew-symmetric) matrix $A$ such that $AN=-N^{T}A$ if and only if
$m_n$ is even for $n$ even (resp.\ $m_n$ is even for $n$ odd).
\end{lemma}

For $L$ of characteristic $0$, the equality is equivalent to $\exp(N)^T A
\exp(N)=A$.

\begin{proof}
We denote the entries of $A$ by $a_{n,c,i}^{n',c',i'}$, where $1\le n,n'\le
l$, $1\le c\le m_n$, $1 \le c'\le m_{n'}$, $1\le i\le n$, $1\le i'\le {n'}$.
Then $AN=-N^{T}A$ if and only if
$a_{n,c,i}^{n',c',i'}=-a_{n,c,i-1}^{n',c',i'+1}$ for $1<i\le n$, $1\le
i'<n'$ and $a_{n,c,i}^{n',c',i'}=0$ for $i+i'\le \max\{n,n'\}$. Let $A_n$ be
the $n\times n$ matrix given by $(A_n)_{ij}=(-1)^{i}$ for $i+j=n+1$ and
$(A_n)_{ij}=0$ otherwise. For $N=N_n$, $n$ odd (resp.\ even), we can take
$A=A_n$. For $N=\diag(N_n,N_n)$, $n$ even (resp.\ odd), we can take
$A=\begin{pmatrix}0&A_n\\-A_n&0\end{pmatrix}$. The ``if'' part follows. The
``only if'' part follows from the case $\sigma=-1$ of Proposition
\ref{p.Pi}, because $P_{1-n}$ is $(-1)^{1-n}$-self-dual (resp.\
$(-1)^n$-self-dual) and $\dim P_{1-n}=m_n$. Let us give a more elementary
proof of the ``only if'' part by induction on $l$. For $l=1$, the assertion
is void (resp.\ $A$ defines a nondegenerate alternating bilinear form on an
$m_1$-dimensional vector space, which implies that $m_1$ is even). For $l\ge
2$, consider the $m_l\times m_l$ submatrices
$B=(a_{l,c,1}^{l,c',l})_{c,c'}$, $C=(a_{l,c,l}^{l,c',1})_{c,c'}$ of $A$. Let
$A'$ be the matrix obtained from $A$ by removing the rows and columns in $A$
that contain entries of $B$ or $C$. Note that for $i'<l$, we have
$a^{n',c',i'}_{l,c,1}=0$, and for $i<l$, we have $a^{l,c',1}_{n,c,i}=0$.
Thus, up to reordering the indices, we have
\[A=\begin{pmatrix}A'&0&B'\\0&0&B\\C'&C&D\end{pmatrix}.\]
It follows that $B^T=\sigma' C=(-1)^{l-1}\sigma' B$, where $\sigma'=1$
(resp.\ $\sigma'=-1$), and $B$ is invertible, so that $m_l$ is even for $l$
even (resp.\ $l$ odd). Moreover, $A'$ is invertible symmetric (resp.\
invertible skew-symmetric) and $A'N'=-N'^TA'$, where
$N'=N(m_1,\dots,m_{l-3},m_{l-2}+m_l,m_{l-1})$ ($N'=N(m_1)$ for $l=2$). The
assertion then follows from induction hypothesis.
\end{proof}

\begin{example}
Let $A$ and $A'$ be sheaves on $X=\Spec(\F_q)$ and let $w\in\Z$. For $n\ge
1$ and $\lambda\in \Qlb^\times$, we let $\mu_{\lambda,n}$ and
$\mu'_{\lambda,n}$ denote the number of $n\times n$ Jordan blocks of
eigenvalue $\lambda$ in the Jordan normal forms of the Frobenius $\Fr_q$
acting on $A_{\Fqb}$ and $A'_{\Fqb}$, respectively. Then
\begin{itemize}
\item $A\simeq (D_X A')(-w)$ if and only if
    $\mu_{\lambda,n}=\mu'_{q^w/\lambda,n}$ for all $n\ge 1$ and all
    $\lambda$. In particular, $A$ is self-dual with respect to $\Qlb(-w)$
    if and only if $\mu_{\lambda,n}=\mu_{q^w/\lambda,n}$ for all $n\ge 1$
    and all $\lambda$. Note that the last condition trivially holds for
    $\lambda= \pm q^{w/2}$.

\item $A$ is $1$-self-dual (resp.\ $-1$-self-dual) with respect to
    $\Qlb(-w)$ if and only if it is self-dual with respect to $\Qlb(-w)$
    and $\mu_{q^{w/2},n}$, $\mu_{-q^{w/2},n}$ are even for $n$ even
    (resp.\ $n$ odd).
\end{itemize}
\end{example}

\begin{example}\label{e.ss}
Let $B$ be a simple perverse sheaf $\iota$-pure of weight $w$, not self-dual
with respect to $K_X(-w)$. Then $A=B^{\oplus 2}\oplus ((D_X B)(-w)\otimes
a_X^* E_2)$ is not self-dual, but the semisimplification of $A$ is both
$1$-self-dual and $-1$-self-dual.
\end{example}

\begin{remark}\leavevmode\label{r.Laff}
\begin{enumerate}
\item An $\iota$-pure complex self-dual with respect to $K_X(-w)$ is
    necessarily of weight $w$.

\item Every simple perverse $\Qlb$-sheaf is $\iota$-pure by a theorem of
    Lafforgue \cite[Corollaire VII.8]{Lafforgue} (with a gap filled by
    Deligne \cite[Th\'eor\`eme 1.6]{Delignegen}; see \cite[Remark
    2.8.1]{SunL} for the case of stacks).

\item The two-out-of-three property (Proposition \ref{p.23}) in the case
    of $\iota$-pure perverse sheaves also follows from the criterion of
    Proposition \ref{p.Jordan}.

\item Since $\iota$-pure perverse sheaves are geometrically semisimple,
    for such perverse sheaves the property of being self-dual (resp.\
    $\sigma$-self-dual) is local for the Zariski topology by Proposition
    \ref{p.cover}.
\end{enumerate}
\end{remark}

\subsection{Symmetry and decomposition of pure complexes over a finite field}\label{ss.pc}
In this subsection, we study the behavior of $\sigma$-self-dual $\iota$-pure
perverse sheaves under operations that preserve purity. The main goal is to
prove the finite field case of Theorem \ref{t.i2} on derived proper direct
image of $\sigma$-self-dual $\iota$-pure perverse sheaves. The behavior of
$\sigma$-self-dual complexes has already been described in Subsection
\ref{ss.gp1}. The focus of this subsection is on decomposition and on the
self-duality of individual perverse cohomology sheaves. To state our
results, it is convenient to introduce the following terminology.

\begin{definition}[Split complexes]
We say that a complex of $\Qlb$-sheaves $A$ is \emph{split} if it is a
direct sum of shifts of perverse sheaves, or, in other words, $A\simeq
\bigoplus_i(\pH^i A)[-i]$.
\end{definition}

\begin{definition}[$\rD^w_{\iota,\sigma}$]\label{d.Diota}
Let $w\in \Z$. We denote by $\rD^w_{\iota,\sigma}(X,\Qlb)\subseteq
\Ob(\rD^b_c(X,\Qlb))$ (resp.\ $\rD^w_{\iota,\sd}(X,\Qlb)\subseteq
\Ob(\rD^b_c(X,\Qlb))$) the subset consisting of split $\iota$-pure complexes
$A$ of weight $w$ such that $\pH^i A$ is $(-1)^{w+i}\sigma$-self-dual
(resp.\ self-dual) with respect to $K_X(-w-i)$ for all $i$. We denote by
$\rD^w_{\iota,\dl}(X,\Qlb)\subseteq\Ob
\big(\rD^b_c(X,\Qlb)\times\rD^b_c(X,\Qlb)\big)$ the subset consisting of
pairs $(A,B)$ of split $\iota$-pure complexes of weight $w$ such that $\pH^i
A$ is isomorphic to $(D_X\pH^i B)(-w-i)$ for all $i$.
\end{definition}

By definition,
$\rD^w_{\iota,\sd}(X,\Qlb)=\Delta^{-1}(\rD^w_{\iota,\dl}(X,\Qlb))$, where
$\Delta\colon \rD^b_c(X,\Qlb)\to \rD^b_c(X,\Qlb)\times \rD^b_c(X,\Qlb)$ is
the diagonal embedding.

Since in this subsection we will only consider operations that preserve
purity, the factor $(-1)^w$ in the definition above is fixed, hence not
essential. We include this factor here to make the definition compatible
with the mixed case studied in Section \ref{s.G}, where the factor is
essential (see Definition \ref{d.K}).

The main result of this section is the following preservation result under
proper direct image, which clearly implies the finite field case of Theorem
\ref{t.i2}.

\begin{theorem}\label{t.mainf}
Let $f\colon X\to Y$ be a proper morphism of Deligne-Mumford stacks of
finite presentation over $\F_q$, where $Y$ has finite inertia. Then $\rR
f_*$ preserves $\rD^w_{\iota,\sigma}$ and $\rD^w_{\iota,\dl}$. In other
words, for $A\in \rD^w_{\iota,\sigma}(X,\Qlb)$ we have $\rR f_*A\in
\rD^w_{\iota,\sigma}(Y,\Qlb)$, and for $(A,B)\in \rD^w_{\iota,\dl}(X,\Qlb)$
we have $(\rR f_*A,\rR f_*B)\in \rD^w_{\iota,\dl}$.
\end{theorem}

The preservation of $\rD^w_{\iota,\dl}$ has the following two consequences,
obtained respectively by considering the diagonal embedding and the first
factor.

\begin{cor}\label{c.mainf}
Let $f\colon X\to Y$ be a proper morphism, where $Y$ has finite inertia.
Then $\rR f_* $ preserves $\rD^w_{\iota,\sd}$.
\end{cor}

\begin{cor}\label{c.decomp}
Let $f\colon X\to Y$ be a proper morphism, where $Y$ has finite inertia.
Then $\rR f_*$ preserves split $\iota$-pure complexes of weight $w$. In
other words, if $A$ is a split $\iota$-pure complex of weight $w$ on $X$,
then $\rR f_*A$ is a split $\iota$-pure complex of weight $w$ on $Y$.
\end{cor}

Corollary \ref{c.decomp} clearly extends to the case where $w\in \R$. Recall
that the Beilinson-Bernstein-Deligne-Gabber decomposition theorem
\cite[Th\'eor\`eme 5.4.5]{BBD} (\cite[Theorem 1.2]{Sun} for the case of
stacks) implies that the pullback of $\rR f_*A$ (or any $\iota$-pure complex
on $Y$) to $Y\otimes_{\F_q}\Fqb$ is split.

\begin{remark}
For $w\in \Z$, let $\rD^w_{\iota,\dec}(X,\Qlb)\subseteq \rD^b_c(X,\Qlb)$
denote the full subcategory consisting of split $\iota$-pure complexes of
weight $w$. Consider the twisted dualizing functor $\Dbar^w_{\iota,X}\colon
\rD^w_{\iota,\dec}(X,\Qlb)^{\op}\to \rD^w_{\iota,\dec}(X,\Qlb)$ carrying $A$
to $\bigoplus_i (D_X\pH^iA)(-w-i) [-i]$. Then $\rD^w_{\iota,\dl}$ is the
collection of pairs $(A,\Dbar^w_{\iota,X}A)$. Thus the preservation of
$\rD^w_{\iota,\dl}$ by $\rR f_*$ is equivalent to the preservation of
$\rD^w_{\iota,\dec}$ and the existence of an isomorphism $\rR
f_*\Dbar^w_{\iota,X} A\simeq \Dbar^w_{\iota,X} \rR f_* A$ for every object
$A$ of $\rD^w_{\iota,\dec}(X,\Qlb)$. Our proof of Theorem \ref{t.mainf}
relies on the two-out-of-three property and the resulting isomorphism is not
necessarily functorial in $A$. Thus our proof does \emph{not} provide a
natural isomorphism between the functors $\rR f_*\Dbar^w_{\iota,X}$ and
$\Dbar^w_{\iota,X} \rR f_*$.
\end{remark}

Let us first recall that the following operations preserve $\iota$-pure
complexes \cite[Stabilit\'es 5.1.14, Corollaire 5.4.3]{BBD} (\cite{Sun} for
the case of stacks). The proof makes use of the fact that these operations
commute with duality (up to shift and twist).

\begin{remark}[Preservation of $\iota$-pure complexes]\label{r.purew}
Let $f\colon X\to Y$ be a morphism, and let $w,w'\in \R$.
\begin{enumerate}
\item For $A\in\rD^b_c(X,\Qlb)$ $\iota$-pure of weight $w$, $A(n)$ is
    $\iota$-pure of weight $w-2n$ and $A[n]$ is $\iota$-pure of weight
    $w+n$ for $n\in \Z$.

\item For $A\in\rD^b_c(X,\Qlb)$, $A$ is $\iota$-pure of weight $w$ if and
    only if $D_X A$ is $\iota$-pure of weight $-w$.

\item If $f$ is an open immersion, the functor $f_{!*}$ preserves
    $\iota$-pure perverse sheaves of weight $w$.

\item Assume that $f$ is smooth. Then $f^*$ preserves $\iota$-pure
    complexes of weight $w$. Moreover, if $f$ is surjective, then $A\in
    \rD^b_c(Y,\Qlb)$ is $\iota$-pure of weight $w$ if and only if $f^*A$
    is so.

\item Assume that $X$ and $Y$ are regular. Then $f^*$ preserves
    $\iota$-pure complexes of weight $w$ in $\rD^b_{\lisse}$. Moreover, if
    $f$ is surjective, then $A\in \rD^b_{\lisse}(Y,\Qlb)$ is $\iota$-pure
    of weight $w$ if and only if $f^*A$ is so.

\item For $A\in\rD^b_c(X,\Qlb)$ which is $\iota$-pure of weight $w$, and
    $A'\in \rD^b_c(X',\Qlb)$ which is $\iota$-pure of weight $w'$,
    $A\boxtimes A' \in \rD^b_c(X\times X',\Qlb)$ is $\iota$-pure of weight
    $w+w'$.

\item For $A\in\rD^b_c(X,\Qlb)$ and $m\ge 1$, $A$ is $\iota$-pure of
    weight $w$ if and only if $A^{\boxtimes m}\in
    \rD^b_c([X^m/\Perm_m],\Qlb)$ is $\iota$-pure of weight $mw$.

\item For a proper morphism $f$, $\rR f_*$ preserves $\iota$-pure
    complexes of weight $w$.

\item Assume that $f$ is a closed immersion and let $A\in
    \rD^b_c(X,\Qlb)$. Then $A$ is $\iota$-pure of weight $w$ if and only
    if $f_*A$ is so.
\end{enumerate}
\end{remark}

Recall that these operations also preserve $\sigma$-self-dual complexes
(Remark \ref{p.sd}). With the exception of derived proper direct image, the
operations also preserve perversity (up to shift). Hence they also preserve
$\rD^w_{\iota,\sigma}$, up to modification of $w$ and $\sigma$. The details
are given below. The case of $(-)^{\boxtimes m}$ requires some additional
arguments and will be given in Proposition \ref{p.boxtimesf} later.

\begin{remark}[Preservation of $\rD^w_{\iota,\sigma}$, easy part]\label{r.Diota}
Let $f\colon X\to Y$ be a morphism, and let $w,w'\in \Z$.
\begin{enumerate}
\item If $A\in \rD^w_{\iota,\sigma}(X,\Qlb)$, then $A[n]\in
    \rD^{w+n}_{\iota,\sigma}(X,\Qlb)$ and $A(n)\in
    \rD^{w-2n}_{\iota,\sigma}(X,\Qlb)$ for $n\in \Z$.

\item $D_X$ carries $\rD^w_{\iota,\sigma}(X,\Qlb)$ to
    $\rD^{-w}_{\iota,\sigma}(X,\Qlb)$.

\item Assume that $X$ is regular and let $\cF$ be a lisse $\Qlb$-sheaf on
    $X$, punctually $\iota$-pure of weight $w$. Then there exists a
    nondegenerate $\sigma$-symmetric pairing $\cF\otimes \cF\to \Qlb(-w)$
    if and only if $\cF$ belongs to $\rD^w_{\iota,(-1)^w\sigma}(X,\Qlb)$.

\item If $f$ is smooth, then $f^*$ preserves $\rD^w_{\iota,\sigma}$.

\item If $X$ and $Y$ are regular, then $f^*$ preserves
    $\rD^w_{\iota,\sigma}\cap \rD^b_{\lisse}$.

\item The functor $-\boxtimes-$ carries
    $\rD^w_{\iota,\sigma}(X,\Qlb)\times \rD^{w'}_{\iota,\sigma'}(X',\Qlb)$
    to $\rD^{w+w'}_{\iota,\sigma \sigma'}(X\times X',\Qlb)$.

\item Assume that $f$ is a closed immersion and let $A\in
    \rD^b_c(X,\Qlb)$. Then $A\in \rD^w_{\iota,\sigma}(X,\Qlb)$ if and only
    if $f_*A\in \rD^w_{\iota,\sigma}(Y,\Qlb)$.

\item Assume that $f$ is an open immersion and let $A\in\Perv(X,\Qlb)$.
    Then $A\in\rD^w_{\iota,\sigma}(X,\Qlb)$ if and only if
    $f_{!*}A\in\rD^w_{\iota,\sigma}(Y,\Qlb)$.
\end{enumerate}
Similar properties hold for $\rD^w_{\iota,\dl}$.
\end{remark}

Next we transcribe the two-out-of-three property (Proposition \ref{p.23})
established earlier in terms of $\rD^w_{\iota,\sigma}$.

\begin{remark}\label{r.23f}
If $A,A',A''\in \rD^b_c(X,\Qlb)$ satisfy
    $A\simeq A'\oplus A''$, and two of the three complexes are in
    $\rD^w_{\iota,\sigma}$, then so is the third one.
A similar property holds for $\rD_{\iota,\dl}$.
\end{remark}

Note that proper direct image $\rR f_*$ does not preserve perversity and in
general there seems to be no canonical way to produce pairings on the
perverse cohomology sheaves $\pR^if_* A$ from pairing on $A$. In the case of
projective direct image, the relative hard Lefschetz theorem provides such
pairings. Let us first fix some terminology on projective morphisms of
Deligne-Mumford stacks.

\begin{definition}\label{d.proj}
Let $f\colon X\to Y$ be a quasi-compact schematic morphism of
Deligne-Mumford stacks. We say that an invertible sheaf $\cL$ on $X$ is
\emph{$f$-ample} if for one (or, equivalently, for every) \'etale surjective
morphism $g\colon Y' \to Y$ where $Y'$ is a scheme, $h^*\cL$ is $f'$-ample
\cite[D\'efinition 4.6.1]{EGA2}. Here $h$ and $f'$ are as shown in the
following Cartesian square
\[\xymatrix{X'\ar[r]^h\ar[d]_{f'} & X\ar[d]^f\\
Y'\ar[r]^g & Y.}
\]
We say that a morphism $f\colon X\to Y$ of Deligne-Mumford stacks is
\emph{quasi-projective} if it is schematic, of finite presentation and if
there exists an $f$-ample invertible sheaf on $X$. We say that $f$ is
\emph{projective} if it is quasi-projective and proper.
\end{definition}

The following is an immediate extension of the case of schemes
\cite[Th\'eor\`eme 5.4.10]{BBD}.

\begin{prop}[Relative hard Lefschetz]\label{p.rhlf}
Let $f\colon X\to Y$ be a projective morphism of Deligne-Mumford stacks of
finite presentation over $\F_q$, and let $\eta\in\rH^2(X,\Qlb(1))$ be the
first Chern class of an $f$-ample invertible sheaf on $X$. Let $A$ be an
$\iota$-pure perverse sheaf on $X$. Then, for $i\ge 0$, the morphism
\[\pR^{-i}f_*(\eta^i\otimes \id_A)\colon \pR^{-i}f_*A \to \pR^{i}f_*A(i)\]
is an isomorphism.
\end{prop}

By Deligne's decomposition theorem \cite{Deligne}, the proposition implies
that $\rR f_* A$ is split.

\begin{prop}\label{p.projf}
Let $f\colon X\to Y$ be a projective morphism, and let $w\in \Z$. Then $\rR
f_*$ preserves $\rD^w_{\iota,\sigma}$ and $\rD^w_{\iota,\dl}$.
\end{prop}

\begin{proof}
We prove the case of $\rD^w_{\iota,\sigma}$, the case of $\rD^w_{\iota,\dl}$
being simpler. It suffices to show that for every $(-1)^w\sigma$-self-dual
$\iota$-pure perverse sheaf $A$ of weight $w$, $\rR f_* A$ belongs to
$\rD^w_{\iota,\sigma}$. Given a $(-1)^w\sigma$-symmetric isomorphism
$A\simto (D_X A)(-w)$, the isomorphism
\[\pR^{-i} f_* A \simto \pR^{-i} f_* (D_XA) (-w)\xrightarrow[\eta^i]{\sim} \pR^i f_* (D_XA) (i-w) \simto (D_Y \pR^{-i} f_* A)(i-w)\]
corresponding to the pairing obtained from
\[\rR f_* A[-i]\otimes \rR f_* A[-i] \to \rR f_* (A\otimes A)[-2i]\to \rR f_* K_X (-w)[-2i] \xrightarrow{\eta^i} \rR f_* K_X (i-w) \to K_Y(i-w) \]
is $(-1)^{w+i}\sigma$-symmetric by Lemma \ref{l.final}.
\end{proof}

\begin{proof}[Proof of Theorem \ref{t.mainf}]
We will prove the case of $\rD^w_{\iota,\sigma}$. The case of
$\rD^w_{\iota,\dl}$ is similar.

Consider the diagram
\[\xymatrix{X\ar[r]^-{f_1} & Y\times_{\bar Y} \bar X\ar[r]\ar[d]& Y\ar[d]\\
&\bar X \ar[r]^{\bar f} &\bar Y}
\]
where $\bar f$ is the morphism of coarse moduli spaces (they exist by the
Keel-Mori theorem \cite{KM}) associated to $f$. Since $f_1$ is proper and
quasi-finite, $\rR f_{1*}$ is $t$-exact for the perverse $t$-structures, and
by Remarks \ref{r.purew} (8) and \ref{p.sd} (3), we see that the theorem
holds for $f_1$. Thus we may assume that $f$ is representable.  We proceed
by induction on the dimension of $X$. Let $A\in
\rD^w_{\iota,\sigma}(X,\Qlb)$; we may assume that $A$ is perverse. Applying
Chow's lemma \cite[Corollaire I.5.7.13]{RG} to the proper morphism $\bar f$
of algebraic spaces, we obtain a projective birational morphism $\bar
g\colon \bar X'\to \bar X$ such that $\bar f \bar g\colon \bar X'\to \bar Y$
is projective. Let $g\colon X'\to X$ be the base change of $\bar g$. Let $U$
be a dense open substack of $X$ such that $g$ induces an isomorphism
$g^{-1}(U)\simto U$. Let $j$ and $j'$ be the open immersions, as shown in
the commutative diagram
\[\xymatrix{&X'\ar[d]^g \\ U\ar[r]^j\ar[ru]^{j'} & X.}\]
By \cite[Corollaire 5.3.11]{BBD} (cf.\ Lemma \ref{l.indec}), we have
\[A\simeq j_{!*}j^* A\oplus B,\]
where $B\in\Perv(X,\Qlb)$ is supported on $X\backslash U$. By Remark
\ref{r.Diota} (4), (8), we have $j_{!*}j^* A\in
\rD^w_{\iota,\sigma}(X,\Qlb)$. By the two-out-of-three property, we have
$B\in\rD^w_{\iota,\sigma}(X,\Qlb)$. As $g$ is projective, by Proposition
\ref{p.projf} we have $\rR g_*j'_{!*}j^* A\in \rD^w_{\iota,\sigma}(X,\Qlb)$,
so
\[\rR g_*j'_{!*}j^* A\simeq j_{!*}j^* A \oplus C,\]
where $C\in \rD^w_{\iota,\sigma}(X,\Qlb)$ is supported on $X\backslash U$.
Applying $\rR f_*$ to $A\oplus C\simeq \rR g_*j'_{!*}j^* A\oplus B$, we
obtain
\[\rR f_* A \oplus \rR f_* C\simeq \rR f_*\rR
g_*j'_{!*}j^* A\oplus \rR f_* B.
\]
By induction hypothesis, $\rR f_* B$ and $\rR f_* C$ belong to
$\rD^w_{\iota,\sigma}(Y,\Qlb)$. As $fg$ is projective, by Proposition
\ref{p.projf} we have $\rR(fg)_*j'_{!*}j^*
A\in\rD^w_{\iota,\sigma}(Y,\Qlb)$. It then follows from the two-out-of-three
property that $\rR f_* A$ belongs to $\rD^w_{\iota,\sigma}(Y,\Qlb)$.
\end{proof}

\begin{remark}[Gabber]\label{r.Chow}
In the case $Y=\Spec(\F_q)$, the proof of Theorem \ref{t.mainf} still makes
use of the relative hard Lefschetz theorem (applied to the morphism $g$).
With the help of a refined Chow's lemma, it is possible to prove this case
of Theorem \ref{t.mainf} using only the absolute hard Lefschetz theorem, at
least in the case of schemes.
\end{remark}

The following is a preservation result for the exterior tensor power functor
$(-)^{\boxtimes m}$. Unlike the functors listed in Remark \ref{r.Diota},
$(-)^{\boxtimes m}$ is not additive and the reduction to the case of
perverse sheaves is not trivial.

\begin{prop}\label{p.boxtimesf}
Let $A$ be an $\iota$-mixed $\Qlb$-complex of integral weights on $X$ such
that for all $n,w\in \Z$, $\gr^W_w\pH^n A$ is $(-1)^w\sigma$-self-dual with
respect to $K_X(-w)$. Then $\gr^W_w\pH^n (A^{\boxtimes m})$ is
$(-1)^w\sigma^m$-self-dual with respect to $K_{[X^m/\Perm_m]}(-w)$ for all
$n,w\in\Z$. Here $W$ denotes the $\iota$-weight filtrations. In particular,
the functor $(-)^{\boxtimes m}$ carries $\rD^w_{\iota,\sigma}(X,\Qlb)$ to
$\rD^{mw}_{\iota,\sigma^m}([X^m/\Perm_m],\Qlb)$.
\end{prop}

Similar results hold for $\rD_{\iota,\dl}$.

\begin{proof}
We write $A^{n}=\pH^n A$. For $\tau\in \Perm_m$, $\tau$ acts on $X^m$ by
$(x_1,\dots,x_m)\mapsto (x_{\tau(1)},\dots,x_{\tau(m)})$. By the K\"unneth
formula,
\[\pH^n
(A^{\boxtimes m})\simeq \bigoplus_{n_1+\dots + n_m=n} A^{n_1}\boxtimes \dots
\boxtimes A^{n_m},
\]
where $\tau$ acts on the right hand side by $\prod_{\substack{i<j\\
\tau(i)>\tau(j)}}(-1)^{n_i n_j}$ times the canonical isomorphism
\[\tau^* (A^{n_{\tau(1)}}\boxtimes \dots
\boxtimes A^{n_{\tau(m)}}) \simto A^{n_1}\boxtimes \dots \boxtimes
A^{n_m}.
\]
Note that $W_w \pH^n(A^{\boxtimes m}) \subseteq \pH^n (A^{\boxtimes m})$ is
the perverse subsheaf given by
\[\sum_{\substack{n_1+\dots + n_m=n\\w_1+\dots +w_m=w}} W_{w_1}A^{n_1}\boxtimes \dots
\boxtimes W_{w_m}A^{n_m}.
\]
So
\[\gr^W_w\pH^n (A^{\boxtimes m})\simeq \bigoplus_{\substack{n_1+\dots + n_m=n\\w_1+\dots +w_m=w}} \gr^W_{w_1}A^{n_1}\boxtimes \dots
\boxtimes \gr^W_{w_1}A^{n_m}.
\]
Thus the $(-1)^{w_i}\sigma$-symmetric isomorphisms
\[\gr^W_{w_i}A^{n_i}\simto (D_X
\gr^W_{w_i} A^{n_i})(-w_i)\]
induce a $(-1)^{w}\sigma^m$-symmetric isomorphism
\[\gr^W_w\pH^n
(A^{\boxtimes m})\simto (D_{X^m}\gr^W_w\pH^n (A^{\boxtimes m}))(-w),\]
compatible with the actions of $\Perm_m$.
\end{proof}

We conclude this subsection with a symmetry criterion in terms of traces of
squares of Frobenius, analogue of the Frobenius-Schur indicator theorem in
representation theory (cf.\ \cite[Section 13.2, Proposition 39]{SerreRL},
\cite[Proposition II.6.8]{BD}). This criterion will be used to show a result
on the independence of $\ell$ of symmetry (Corollary \ref{c.indep}). We
refer the reader to \cite[Theorem 1.9.6]{Katz} for a related criterion on
the symmetry of the geometric monodromy. For a groupoid $\cC$, we let
$\lvert\cC\rvert$ denote the set of isomorphism classes of its objects.

\begin{prop}\label{p.square}
Let $X$ be a Deligne-Mumford stack of finite presentation over $\F_q$,
connected and geometrically unibranch of dimension $d$. Let $\cF$ be a
semisimple lisse $\Qlb$-sheaf on $X$, punctually $\iota$-pure of weight
$w\in \Z$. Consider the series
\[L^{(2)}(T)=
\exp\left(\sum_{m\ge 1}\quad\sum_{x\in \lvert X(\F_{q^m})\rvert}\frac{\iota\tr(\Fr_x^2\mid \cF_{\bar x})}{\# \Aut(x)}\frac{T^m}{m}\right),
\]
where $\bar x$ denotes a geometric point above $x$, and $\Fr_x=\Fr_{q^m}$.
Then the series $L^{(2)}(T)$ converges absolutely for $\lvert T\rvert<
q^{-w-d}$ and extends to a rational function satisfying
\[-\mathrm{ord}_{T=q^{-w-d}}L^{(2)}(T)=\dim \rH^0(X,(\Sym^2 (\cF\spcheck))(-w))-\dim \rH^0(X,(\wedge^2 (\cF\spcheck))(-w)).\]
In particular, if $\cF$ is simple and $\sigma$-self-dual (resp.\ not
self-dual) with respect to $\Qlb(-w)$, then
\[-\mathrm{ord}_{T=q^{-w-d}} L^{(2)}(T)=\sigma\quad \text{(resp.\ $=0$)}.\]
\end{prop}

\begin{proof}
For $x\in X(\F_{q^m})$,
\[\tr(\Fr_x^2\mid \cF_{\bar x})=\tr(\Fr_x\mid \Sym^2\cF_{\bar x})-\tr(\Fr_x\mid \wedge^2\cF_{\bar x}).\]
Thus $L^{(2)}(T)=L_\iota(X,\Sym^2\cF,T)/L_\iota(X,\wedge^2 \cF,T)$ (see
\cite[Definition 4.1]{SunL} for the definition of $L$-series
$L_\iota(X,-,T)$). Note that $\cF\otimes \cF$ is lisse punctually
$\iota$-pure of weight $2w$, and semisimple by a theorem of Chevalley
\cite[IV, Proposition 5.2]{Chev}. The same holds for $\Sym^2\cF$ and
$\wedge^2 \cF$. For $\cG$ lisse punctually $\iota$-pure of weight $2w$ on
$X$, the series $L_\iota(X,\cG,T)$ converges absolutely for $\lvert T\rvert<
q^{-w-d}$ and extends to a rational function
\[\iota \prod_i \det(1-T\Fr_q\mid \rH^i_c(X_{\Fqb},\cG))^{(-1)^{i+1}}
\]
by \cite[Theorem 4.2]{SunL}. Only the factor $i=2d$ may contribute to poles
on the circle $|T|=q^{-w-d}$, by \cite[Theorem 1.4]{SunL}. For every dense
open substack $U$ of $X$ such that $U_\red$ is regular, we have
\[\rH^{2d}_c(X_{\Fqb},\cG)\simeq \rH^{0}(U_{\Fqb},\cG\spcheck)\spcheck(-d)
 \simeq \rH^0(X_{\Fqb},\cG\spcheck)\spcheck(-d)
\]
Here in the second assertion we used the hypothesis on $X$, which implies
that the homomorphism $\pi_1(U)\to \pi_1(X)$ is surjective. Therefore,
\[-\mathrm{ord}_{T=q^{-w-d}}
L_\iota(X,\cG,T)=\dim \rH^{0}(X,\cG\spcheck(-w))\]
for $\cG$ semisimple.
\end{proof}

\subsection{Variant: Semisimple complexes over a separably closed field}\label{s.var}

In this subsection let $k$ be a separably closed field. We establish
variants over $k$ of results of Subsection \ref{ss.pc}. The main result is a
preservation result under proper direct image (Theorem \ref{t.main}). We
also include an example of Gabber (Remark \ref{r.Gabber}) showcasing the
difference in parity of Betti numbers between zero and positive
characteristics.

One key point in Subsection \ref{ss.pc} is the relative hard Lefschetz
theorem for pure perverse sheaves. If $k$ has characteristic zero, a
conjecture of Kashiwara states that all semisimple perverse sheaves satisfy
relative hard Lefschetz. Kashiwara's conjecture was proved by Drinfeld
\cite{Drinfeld} assuming de Jong's conjecture for infinitely many primes
$\ell$, which was later proved by Gaitsgory for $\ell>2$ \cite{Gaitsgory}.

Drinfeld's proof uses the techniques of reduction from $k$ to finite fields
in \cite[Section 6]{BBD}. The reduction holds in fact without restriction on
the characteristic of $k$ and provides a class of semisimple perverse
sheaves over $k$ for which the relative hard Lefschetz theorem holds. Let us
briefly recall the reduction. Let $X$ be a Deligne-Mumford stack of finite
presentation over $k$. There exist a subring $R\subseteq k$ of finite type
over $\Z[1/\ell]$ and a Deligne-Mumford stack $X_R$ of finite presentation
over $\Spec(R)$ such that $X\simeq X_R \otimes_Rk$. For extra data
$(\cT,\cL)$ on $X$, we have an equivalence \cite[6.1.10]{BBD}
\begin{equation}\label{e.BBD}
\rD^b_{\cT,\cL}(X,\Qlb)\simto \rD^b_{\cT,\cL}(X_s,\Qlb)
\end{equation}
for every geometric point $s$ above a closed point of $\Spec(R)$, provided
that $R$ is big enough (relative to the data $(\cT,\cL)$). Here $X_s$
denotes the base change of $X_R$ by $s\to\Spec(R)$. Note that $s$ is the
spectrum of an algebraic closure of a finite field. Each $A\in
\rD^b_c(X,\Qlb)$ is contained in $\rD^b_{\cT,\cL}(X,\Qlb)$ for some
$(\cT,\cL)$.

\begin{definition}[Admissible semisimple complexes]\label{d.ass}
Let $A$ be a semisimple perverse $\Qlb$-sheaf on $X$. If $k$ is an algebraic
closure of a finite field, we say that $A$ is \emph{admissible} if there
exists a Deligne-Mumford stack $X_0$ of finite presentation over a finite
subfield $k_0$ of $k$, an isomorphism $X\simeq X_0\otimes_{k_0} k$, and a
perverse $\Qlb$-sheaf $A_0$ on $X_0$ such that $A$ is isomorphic to the
pullback of $A_0$. More generally, if $k$ has characteristic $>0$, we say
that $A$ is \emph{admissible} if the images of $A$ under the equivalences
\eqref{e.BBD}, for all geometric points $s$ over a closed point of
$\Spec(R)$, are admissible, for some $R$ big enough. If $k$ has
characteristic zero, we adopt the convention that every semisimple perverse
$\Qlb$-sheaf is admissible.

We say that a complex $B\in \rD^b_c(X,\Qlb)$ is \emph{admissible semisimple}
if $B\simeq \bigoplus_i (\pH^i B)[-i]$ and, for each $i$, its $i$-th
perverse cohomology sheaf $\pH^i B$ is admissible semisimple.
\end{definition}

\begin{remark}
In the case where $k$ is the algebraic closure of a finite field $k_0$, for
$X_0$ as above, the pullback of an $\iota$-pure complex on $X_0$ to $X$ is
admissible semisimple by the decomposition theorems \cite[Th\'eor\`emes
5.3.8, 5.4.5]{BBD} (\cite{Sun} for the case of stacks). Conversely, if a
semisimple perverse $\Qlb$-sheaf $A$ on $X$ is the pullback of $A_0$ on
$X_0$ as above, then we may take $A_0$ to be pure (of weight 0 for example)
by Lafforgue's theorem \cite[Corollaire VII.8]{Lafforgue} mentioned in
Remark \ref{r.Laff} (2).
\end{remark}

\begin{remark}
Following \cite[6.2.4]{BBD}, we say that a simple perverse $\Qlb$-sheaf on
$X$ is \emph{of geometric origin} if it belongs to the class of simple
perverse $\Qlb$-sheaves generated from the constant sheaf $\Qlb$ on
$\Spec(k)$ by taking composition factors of perverse cohomology sheaves
under the six operations. By \cite[Lemme 6.2.6]{BBD} (suitably extended),
simple perverse $\Qlb$-sheaves of geometric origin are admissible.
\end{remark}

The operations that preserve purity also preserve admissible semisimple
complexes. The details are given below.

\begin{remark}[Preservation of admissible semisimple complexes]\label{r.pure}
Let $f\colon X\to Y$ be a morphism.
\begin{itemize}
\item The full subcategory of $\rD^b_c$ consisting of objects $A$ such
    that the composition factors of $\pH^i A$ are admissible for all $i$
    is stable under the operations $\rR f_*$, $\rR f_!$, $f^*$, $\rR f^!$,
    $\otimes$, $\rR \cHom$, and $(-)^{\boxtimes m}$.

\item $D_X\colon \rD^b_c(X,\Qlb)^\op\to \rD^b_c(X,\Qlb)$ preserves
    admissible semisimple complexes.

\item If $f$ is an open immersion, $f_{!*}$ preserves admissible
    semisimple perverse sheaves.

\item Assume that $f$ is a closed immersion and let $A\in
    \rD^b_c(X,\Qlb)$. Then $A$ is admissible semisimple if and only if
    $f_*A$ is admissible semisimple.

\item If $f$ is smooth, $f^*$ preserves admissible semisimple complexes.

\item If $X$ and $Y$ are regular, $f^*$ preserves admissible semisimple
    complexes in $\rD^b_{\lisse}$.

\item The functors
\begin{gather*}
-\boxtimes-\colon \rD^b_c(X,\Qlb)\times
    \rD^b_c(X',\Qlb)\to \rD^b_c(X\times X',\Qlb),\\
(-)^{\boxtimes m}\colon \rD^b_c(X,\Qlb)\to \rD^b_c([X^m/\Perm_m],\Qlb),\quad m\ge 0
\end{gather*}
    preserve admissible semisimple complexes.

\item For a proper morphism $f$, $\rR f_*$ preserves admissible semisimple
    complexes.

\end{itemize}
These properties reduce to the corresponding properties for pure complexes
over a finite field (Remark \ref{r.purew}). Since the equivalences
\eqref{e.BBD} are compatible with these operations, this reduction is clear
in positive characteristic. The reduction in characteristic zero is more
involved. The case of $\rR f_*$ is done in \cite{Drinfeld} and the other
cases can be done similarly.
\end{remark}

\begin{definition}[$\rD_\sigma$]\label{d.Dsig}
We denote by $\rD_\sigma(X,\Qlb)\subseteq \Ob(\rD^b_c(X,\Qlb))$ (resp.\
$\rD_\sd(X,\Qlb)\subseteq \Ob(\rD^b_c(X,\Qlb)$)) the subset consisting of
admissible semisimple complexes $A$ such that $\pH^i A$ is
$(-1)^i\sigma$-self-dual (resp.\ self-dual) with respect to $K_X$, for all
$i$. We denote by $\rD_\dl(X,\Qlb)\subseteq\Ob\big(\rD^b_c(X,\Qlb)\times
\rD^b_c(X,\Qlb)\big)$ the subset consisting of pairs $(A,B)$ such that both
$A$ and $B$ are admissible semisimple, and that $\pH^i A$ is isomorphic to
$D_X\pH^i B$.
\end{definition}

By definition, $\rD_\sd(X,\Qlb)=\Delta^{-1}(\rD_\dl(X,\Qlb))$, where
$\Delta\colon \rD^b_c(X,\Qlb)\to \rD^b_c(X,\Qlb)\times \rD^b_c(X,\Qlb)$ is
the diagonal embedding.

\begin{example}\label{ex.pt}
For $X=\Spec(k)$, every object $A$ of $\rD^b_c(X,\Qlb)$ is admissible
semisimple and belongs to $\rD_\sd(X,\Qlb)$. Let $d_i=\dim\cH^i(A)$. Then
\begin{itemize}
\item $A$ is $1$-self-dual with respect to $\Qlb$ if and only if it is
    self-dual with respect to $\Qlb$, namely if $d_i=d_{-i}$ for all $i$.
    (Recall that in general self-dual objects are not necessarily
    $1$-self-dual.)

\item $A$ is $-1$-self-dual with respect to $\Qlb$ if and only if
    $d_i=d_{-i}$ for all $i$ and $d_0$ is even.

\item $A\in \rD_1(X,\Qlb)$ (resp.\ $\in \rD_{-1}(X,\Qlb)$) if and only if
    $d_i$ is even for $i$ odd (resp.\ even).

\item For $A,B\in \rD^b_c(X,\Qlb)$, $(A,B)\in \rD_\dl(X,\Qlb)$ if and only
    if $A\simeq B$.
\end{itemize}
\end{example}

The main result of this subsection is the following.

\begin{theorem}\label{t.main}
Let $f\colon X\to Y$ be a proper morphism of Deligne-Mumford stacks of
finite presentation over $k$, where $Y$ has finite inertia. Then $\rR f_*$
preserves $\rD_\sigma$ and $\rD_\dl$.
\end{theorem}

\begin{cor}
Let $f\colon X\to Y$ be a proper morphism, where $Y$ has finite inertia.
Then $\rR f_*$ preserves $\rD_\sd$.
\end{cor}

The strategy for proving Theorem \ref{t.main} is the same as Theorem
\ref{t.mainf}. Let us recall that the operations listed in Remark
\ref{r.pure} that preserve admissible semisimple complexes also preserve
$\sigma$-self-dual complexes (Remark \ref{p.sd}). With the exception of
proper direct image $\rR f_*$, they also preserve perversity, hence they
preserve $\rD_\sigma$. The details are given below.

\begin{remark}[Preservation of $\rD_\sigma$, easy part]\label{r.Dsig}
Let $f\colon X\to Y$ be a morphism, and let $n\in \Z$.
\begin{enumerate}
\item If $A\in \rD_\sigma(X,\Qlb)$, then $A[n]\in
    \rD_{(-1)^n\sigma}(X,\Qlb)$.

\item $D_X$ preserves $\rD_\sigma(X,\Qlb)$.

\item Assume that $X$ is regular and let $\cF$ be a lisse $\Qlb$-sheaf on
    $X$, admissible semisimple. Then there exists a nondegenerate
    $\sigma$-symmetric pairing $\cF\otimes \cF\to \Qlb$ if and only if
    $\cF$ belongs to $\rD_\sigma(X,\Qlb)$.

\item If $f$ is smooth, then $f^*$ preserves $\rD_\sigma$.

\item If $X$ and $Y$ are regular, then $f^*$ preserves
    $\rD_\sigma\cap\rD^b_{\lisse}$.

\item Assume that $f$ is a closed immersion and let $A\in
    \rD^b_c(X,\Qlb)$. Then $A\in \rD_\sigma(X,\Qlb)$ if and only if
    $f_*A\in \rD_\sigma(Y,\Qlb)$.

\item The exterior tensor product functors induce functors
\begin{gather*}
    -\boxtimes -\colon\rD_\sigma(X,\Qlb)\times \rD_{\sigma'}(X',\Qlb)\to\rD_{\sigma
    \sigma'}(X\times X',\Qlb),\\
    (-)^{\boxtimes m}\colon \rD_\sigma(X,\Qlb)\to \rD_{\sigma^m}([X^m/\Perm_m],\Qlb), \quad m\ge 0.
\end{gather*}
For $(-)^{\boxtimes m}$, the reduction to perverse sheaves is nontrivial
and is similar to Proposition \ref{p.boxtimesf}.
\end{enumerate}
Similar properties hold for $\rD_\dl$.
\end{remark}

By Proposition \ref{p.23}, the two-out-of-three property holds for
$\rD_\sigma$ and $\rD_\dl$.

We state a relative hard Lefschetz theorem over an arbitrary field $F$ in
which $\ell$ is invertible.

\begin{prop}[Relative hard Lefschetz]\label{p.rhl}
Let $f\colon X\to Y$ be a projective morphism of Deligne-Mumford stacks of
finite presentation over $F$, and let $\eta\in\rH^2(X,\Qlb(1))$ be the first
Chern class of an $f$-ample invertible sheaf on $X$. Let $A$ be a perverse
$\Qlb$-sheaf on $X$ whose pullback to $X\otimes_F \bar F$ is admissible
semisimple. Then, for $i\ge 0$, the morphism
\[\pH^{-i}(\eta^i\otimes \id_A)\colon \pR^{-i}f_*A \to \pR^{i}f_*A(i)\]
is an isomorphism.
\end{prop}

That the morphism is an isomorphism can be checked on $X\otimes_F \bar F$.
Thus we are reduced to the case where $F=k$ is separably closed. As
mentioned earlier, the relative hard Lefschetz theorem in this case is
obtained by reduction to the finite field case (Proposition \ref{p.rhlf}).

Combining Proposition \ref{p.rhl} with Lemma \ref{l.final}, we obtain the
following preservation result under projective direct image.

\begin{prop}\label{p.proj}
Let $f\colon X\to Y$ be a projective morphism (over $k$). Then $\rR f_*$
preserves $\rD_\sigma$ and $\rD_\dl$.
\end{prop}

The proof of Theorem \ref{t.mainf} can now be repeated verbatim to prove
Theorem \ref{t.main}.

Over an arbitrary field $F$ in which $\ell$ is invertible, we may exploit
the relative hard Lefschetz theorem to get analogues for split complexes
that are geometrically admissible semisimple.

\begin{prop}\label{p.karb}
Let $f\colon X\to Y$  be a proper morphism of separated Deligne-Mumford
stacks of finite type over $F$. Assume that $X$ is regular. Let $\cF$ be a
lisse $\Qlb$-sheaf on $X$, whose pullback to $X_{\bar F}$ is admissible
semisimple. Then we have $\rR f_* \cF \simeq \bigoplus_i (\pR^i f_*
\cF)[-i]$.
\end{prop}

\begin{proof}
By \cite[Th\'eor\`eme 16.6]{LMB}, there exists a finite surjective morphism
$g_1\colon X'\to X$ where $X'$ is a scheme. Up to replacing $X'$ by its
normalization, we may assume that $X'$ is normal. By de Jong's alterations
\cite[Theorem 4.1]{dJ}, there exists a proper surjective morphism $g_2\colon
X''\to X'$, generically finite, such that $X''$ is regular and
quasi-projective over $F$. Let $g=g_1 g_2\colon X''\to X$. By the relative
hard Lefschetz theorem (Proposition \ref{p.rhl}) and Deligne's decomposition
theorem \cite{Deligne}, we have
\[\rR (fg)_* g^*\cF\simeq \bigoplus_i (\pR^i(fg)_* g^*\cF)[-i].
\]
Note that $g^*\cF\simeq g^*\cF\otimes \rR g^! \Qlb \simeq \rR g^!\cF$.
Consider the composite
\[\alpha\colon \cF\to \rR g_*g^*\cF\simeq \rR g_! \rR g^!\cF \to \cF\]
of the adjunction morphisms. Since $\alpha$ is generically multiplication by
the degree of $g$, $\alpha$ is an isomorphism. It follows that $\cF$ is a
direct summand of $\rR g_* g^* \cF$, so that $\rR f_* \cF$ is a direct
summand of $\rR (fg)_* g^*\cF$.
\end{proof}

\begin{remark}[Gabber]\label{r.Gabber}
Let $X$ be a proper smooth algebraic space over $k$ and let $\cF$ be a lisse
$\Qlb$-sheaf on $X$ with finite monodromy, $-1$-self-dual with respect to
$\Qlb$. Then $\cF$ is admissible semisimple (since each simple factor is of
geometric origin), and $\cF$ belongs to $\rD_{-1}$. By Theorem \ref{t.main},
$b_n(\cF)\coloneqq\dim \rH^n(X,\cF)$ is even for $n$ even.

If $k$ has characteristic $0$, then $b_n(\cF)$ is even for all $n$. To see
this, we may assume $k=\C$, $X$ connected, and $\cF$ simple. Let $G$ be the
monodromy group of $\cF$ and let $f\colon Y\to X$ be the corresponding
Galois \'etale cover. Then
\[\rH^n(X,\cF)\simeq \rH^n(Y,f^*\cF)^G\simeq (\rH^n(Y,\Qlb)\otimes_{\Qlb} V)^G, \]
where $V$ is the representation of $G$ corresponding to $\cF$. Thus
$b_n(\cF)$ is the multiplicity of $V\spcheck$ in the representation
$\rH^n(Y,\Qlb)$ of $G$. Since the complex representation $\rH^n(Y(\C),\C)$
of $G$ has a real structure $\rH^n(Y(\C),\R)$, it admits a $G$-invariant
nondegenerate symmetric bilinear form. In other words, it is $1$-self-dual.
The same holds for $\rH^n(Y,\Qlb)$. Therefore, the multiplicity of
$V\spcheck$, which is $-1$-self-dual, is necessarily even (cf.\
\cite[Section 13.2, Theorem 31]{SerreRL}, \cite[Proposition II.6.6 (i),
(ii), (iii)]{BD}).

By contrast, if $k$ has characteristic $2$ or $3$, then $b_n(\cF)$ may be
odd for $n$ odd, as shown by the following counterexample. Let $E$ be a
supersingular elliptic curve over $k$ and let $G$ be its automorphism group.
Let $X'\to X$ be a finite \'etale cover of connected projective smooth
curves over $k$ of Galois group $G$, which exists by \cite[Theorem 7.4]{PS},
as explained in \cite[Section 3]{Partsch}. Let $f\colon Y=(X'\times E)/G\to
X$ be the projection, where $G$ acts diagonally on $X'\times E$. Then
$\cF=\rR^1 f_*\Qlb$ is a $-1$-self-dual simple lisse sheaf of rank $2$ on
$X$, of monodromy $G$. Note that $f^*\cF$ is a $-1$-self-dual simple lisse
sheaf on $Y$ of monodromy $G$, since $\pi_1(Y)$ maps onto $\pi_1(X)$. We
claim that $b_1(f^*\cF)$ and $b_1(\cF)$ are not of the same parity. Indeed,
consider the Leray spectral sequence for $(f, f^*\cF)$:
\[E_2^{pq}=\rH^p(X,R^q f_* f^*\cF)\Rightarrow \rH^{p+q}(Y,f^*\cF).\]
By the projection formula,
\[f_* f^*\cF\simeq f_*\Qlb \otimes \cF\simeq \cF,\quad \rR^1 f_*f^*\cF\simeq \rR^1f_*\Qlb\otimes \cF=\cF\otimes \cF,\]
so we have an exact sequence
\[0\to \rH^1(X,\cF)\to \rH^1(Y,f^*\cF)\to \rH^0(X,\cF\otimes \cF)\to E_2^{20}=0.\]
Since $\dim\rH^0(X,\cF\otimes \cF)=1$, we get $b_1(f^*\cF)=b_1(\cF)+1$.
(That the Leray spectral sequence degenerates at $E_2$ also follows from a
general theorem of Deligne \cite{DeligneL}.) By the
Grothendieck-Ogg-Shafarevich formula, we have
$b_1(\cF)=(2g-2)\text{rk}(\cF)=4g-4$ is even, where $g$ is the genus of $X$.
It follows that $b_1(f^*\cF)=4g-3$ is odd.
\end{remark}

\section{Symmetry in Grothendieck groups}\label{s.G}

In Section \ref{s.pc}, we studied the behavior of $\sigma$-self-dual pure
perverse sheaves under operations that preserve purity. Mixed Hodge theory
suggests that one may expect results for more general operations in the
mixed case. This section confirms such expectations in a weak sense, by
working in Grothendieck groups. We work over a finite field $k=\F_q$. In
Subsection \ref{ss.G1}, we review operations on Grothendieck groups. In
Subsection \ref{ss.G2}, we define certain subgroups of the Grothendieck
groups and state the main result of this section (Theorem \ref{t.sixop}),
which says that these subgroups are preserved by Grothendieck's six
operations, and contains the finite field case of Theorem \ref{t.i3}. The
proof is a bit involved and is given in Subsection \ref{ss.G3}.

\subsection{Operations on Grothendieck groups}\label{ss.G1}

In this subsection, we review Grothendieck groups and operations on them.
The six operations are easily defined. The action of the middle extension
functor $f_{!*}$ on Grothendieck groups is more subtle, and we justify our
definition with the help of purity.

\begin{construction}[Six operations on Grothendieck groups]\label{c.lambda}
Let $X$ be a Deligne-Mumford stack of finite presentation over a field. We
let $\rK(X,\Qlb)$ denote the Grothendieck group of $\rD^b_c(X,\Qlb)$, which
is a free Abelian group generated by the isomorphism classes of simple
perverse $\Qlb$-sheaves. For an object $A$ of $\rD^b_c(X,\Qlb)$, we let
$[A]$ denote its class in $\rK(X,\Qlb)$. The usual operations on derived
categories induce maps between Grothendieck groups. More precisely, for a
morphism $f\colon X\to Y$ of Deligne-Mumford stacks of finite presentation
over a field, we have $\Z$-(bi)linear maps
\begin{gather*}
-\otimes-,\ \cHom(-,-)\colon \rK(X,\Qlb)\times \rK(X,\Qlb)\to \rK(X,\Qlb),\qquad D_X\colon \rK(X,\Qlb)\to \rK(X,\Qlb),\\
-\boxtimes- \colon \rK(X,\Qlb)\times \rK(Y,\Qlb)\to \rK(X\times Y,\Qlb),\\
f^*,f^!\colon\rK(Y,\Qlb)\to \rK(X,\Qlb), \qquad f_*,f_!\colon \rK(X,\Qlb)\to \rK(Y,\Qlb).
\end{gather*}
Tensor product ($-\otimes -$) endows $\rK(X,\Qlb)$ with a ring structure.
The map $f^*$ is a ring homomorphism.

The Grothendieck ring is equipped with the structure of a $\lambda$-ring as
follows. Readers not interested in this structure may skip this part as it
is not used in the proof of the theorems in the introduction. For $m\ge 0$,
we have a map
\[(-)^{\boxtimes m}\colon \rK(X,\Qlb)\to \rK([X^m/\Perm_m],\Qlb),\]
which preserves multiplication and satisfies $(n[A])^{\boxtimes m}= n^m
[A]^{\boxtimes m}$ (with the convention $0^0=1$) for $n\ge 0$ and
$(-[A])^{\boxtimes m}=(-1)^m [\cS]\otimes [A]^{\boxtimes m}$, where $\cS$ is
the lisse sheaf of rank $1$ on $[X^m/\Perm_m]$ given by the sign character
$\Perm_m\to \Qlb^\times$. The maps $\lambda^m\colon \rK(X,\Qlb)\to
\rK(X,\Qlb)$ given by $\lambda^m(x)=(-1)^m p_*\Delta^* (-x)^{\boxtimes m}$,
where $\Delta\colon X\times B\Perm_m\to [X^m/\Perm_m]$ is the diagonal
morphism and $p\colon X\times B\Perm_m\to X$ is the projection, endow
$\rK(X,\Qlb)$ with the structure of a special $\lambda$-ring. The map $f^*$
is a $\lambda$-ring homomorphism. We refer the reader to
\cite[Section~4]{Grothendieck} and \cite{AT} for the definitions of special
$\lambda$-ring and $\lambda$-ring homomorphism.
\end{construction}

\begin{remark}\label{r.mid}
For a separated quasi-finite morphism $f\colon X\to Y$, the functor
$f_{!*}\colon \Perv(X,\Qlb)\to \Perv(Y,\Qlb)$ is \emph{not exact} in
general. There exists a unique homomorphism $f_{!*}\colon \rK(X,\Qlb)\to
\rK(Y,\Qlb)$ such that $f_{!*}[A]=[f_{!*} A]$ for $A$ perverse
\emph{semisimple}. As we shall see in Lemma \ref{l.interexact}, over a
finite field this identity also holds for $A$ pure perverse.
\end{remark}

We note the following consequence of Lemma \ref{l.cover} (applied to
semisimple perverse sheaves).

\begin{lemma}\label{l.coverK}
Let $(X_\alpha)_{\alpha\in I}$ be a finite Zariski open covering of $X$, and
let $A\in \rK(X,\Qlb)$. Then
    \[\sum_{\substack{J\subseteq I\\\text{$\#J$ \emph{even}}}} j_{J!*}j_J^* A
    =
    \sum_{\substack{J\subseteq I\\\text{$\#J$ \emph{odd}}}} j_{J!*}j_J^* A,
    \]
where $j_J\colon \bigcap_{\alpha\in J}X_\alpha\to X$ is the open immersion.
\end{lemma}

In the rest of this section we work over a finite field $k=\F_q$. We first
recall the following injectivity, which will be used in the proof of
Corollary \ref{c.R}.

\begin{lemma}\label{l.inj}
The homomorphism $\rK(X,\Qlb)\to \Map(\coprod_{m\ge 1} \lvert
X(\F_{q^m})\rvert,\Qlb)$ sending $A$ to $x\mapsto \tr(\Fr_x\mid A_{\bar x})$
is injective.
\end{lemma}

As in \cite[Th\'eor\`eme 1.1.2]{LaumonTF}, this injectivity follows from
Chebotarev's density theorem \cite[Theorem~7]{Serre}, which extends to the
case of Deligne-Mumford stacks as follows.

\begin{lemma}\label{l.Ch}
Let $Y\to X$ be a Galois \'etale cover of Galois group $G$ of irreducible
Deligne-Mumford stacks of dimension $d$ of finite presentation over $\F_q$.
Let $R\subseteq G$ be a subset stable under conjugation. Then
\[\lim_{T\to (q^{-d})^-} \sum_{m\ge 1}\sum_{x} \frac{1}{\# \Aut(x)}\frac{T^m}{m}/\log\frac{1}{T-q^{-d}}=\# R/\# G,
\]
where $x$ runs through isomorphism classes of $X(\F_{q^m})$ such that the
image $F_x$ of $\Fr_x$ in $G$ (well-defined up to conjugation) lies in $R$.
\end{lemma}

\begin{proof}
For a character $\chi\colon G\to \Qlb$ of a $\Qlb$-representation of $G$,
consider the $L$-series
\[L(X,\iota\chi,T)=L_\iota(X,\cF_\chi,T)=
\exp\left(\sum_{m\ge 1}\quad\sum_{x\in \lvert X(\F_{q^m})\rvert}\frac{\iota\chi(F_x)}{\# \Aut(x)}
\frac{T^m}{m}\right)
\]
associated to the corresponding lisse $\Qlb$-sheaf $\cF_\chi$ on $X$
\cite[Definition 4.1]{SunL}. The series $L(X,\iota\chi,T)$ converges
absolutely for $\lvert T\rvert< q^{-d}$ and extends to a rational function
\[\iota \prod_i \det(1-T\Fr_q\mid \rH^i_c(X_{\Fqb},\cF_\chi))^{(-1)^{i+1}}
\]
by \cite[Theorem 4.2]{SunL}. As $\rH^{2d}_c(X_{\Fqb},\cF_\chi)\simeq
\rH^{0}(U_{\Fqb},\cF_\chi\spcheck)\spcheck(-d)$ for a dense open substack
$U$ of $X$ such that $U_\red$ is regular, $-\mathrm{ord}_{T=q^{-d}}
L(X,\iota\chi,T)=\dim \rH^{0}(U,\cF_\chi\spcheck)$ is the multiplicity of
the identity character in $\chi$, so that
\[\lim_{T\to (q^{-d})^-}\sum_{m\ge 1}\quad\sum_{x\in \lvert X(\F_{q^m})\rvert}\frac{\iota\chi(F_x)}{\# \Aut(x)}\frac{T^m}{m}/\log\frac{1}{T-q^{-d}}=\sum_{g\in G} \iota \chi(g) /\# G.
\]
This equality extends to an arbitrary class function $\chi\colon G\to \Qlb$.
It then suffices to take $\chi$ to be the characteristic function of $R$.
\end{proof}

Next we discuss purity.

\begin{notation}
For $w\in \R$, we let $\rK^w_\iota(X,\Qlb)\subseteq \rK(X,\Qlb)$ denote the
subgroup generated by perverse sheaves $\iota$-pure of weight $w$ on $X$. We
put $\rK^\Z_\iota(X,\Qlb)\coloneqq \bigoplus_{w\in \Z} \rK_\iota^w(X,\Qlb)$.
\end{notation}

The group $\rK^w_\iota(X,\Qlb)$ is a free Abelian group generated by the
isomorphism classes of simple perverse sheaves $\iota$-pure of weight $w$ on
$X$. We have $\bigoplus_{w\in \R}\rK^w_\iota(X,\Qlb)\subseteq \rK(X,\Qlb)$
and the $\lambda$-subring $\rK^\Z_\iota(X,\Qlb)\subseteq \rK(X,\Qlb)$ is
stable under Grothendieck's six operations and duality. For $w\in \Z$, the
group $\rK^w(X,\Qlb)\coloneqq \bigcap_\iota \rK^w_\iota(X,\Qlb)$ is a free
Abelian group generated by the isomorphism classes of perverse sheaves pure
of weight $w$ on $X$.

\begin{remark}\label{r.Lafforgue}
In fact, we have $\bigoplus_{w\in \R}\rK^w_\iota(X,\Qlb)= \rK(X,\Qlb)$, as
every $\Qlb$-sheaf on $X$ is $\iota$-mixed by Lafforgue's theorem
\cite[Corollaire VII.8]{Lafforgue} mentioned in Remark \ref{r.Laff} (2).
\end{remark}

For a subset $I\subseteq \R$, we let $\Perv^I_\iota(X,\Qlb)\subseteq
\Perv(X,\Qlb)$ denote the full subcategory of perverse sheaves $\iota$-mixed
of weights contained in $I$. Lemmas \ref{l.interexact} and \ref{l.nine}
below, which justify the definition of the map $f_{!*}$ in Remark
\ref{r.mid}, are taken from \cite[Lemme 2.9, Corollaire 2.10]{M2}.

\begin{lemma}\label{l.interexact}
Let $f\colon X\to Y$ be a separated quasi-finite morphism. For $w\in \R$,
the functor
\[
f_{!*}\colon
\Perv^{\{w,w+1\}}_\iota(X,\Qlb)\to \Perv^{\{w,w+1\}}_\iota(Y,\Qlb)
\]
is exact. In particular, $f_{!*}[A]=[f_{!*}A]$ for $A$ in
$\Perv^{\{w,w+1\}}_\iota(X,\Qlb)$.
\end{lemma}

\begin{proof}
As the assertion is local for the \'etale topology on $Y$ and trivial for
$f$ proper quasi-finite, we may assume that $f$ is an open immersion. Let
$i\colon Z\to Y$ be the closed immersion complementary to $f$. We proceed by
induction on the dimension $d$ of $Z$. Let $0\to A_1\to A_2\to A_3 \to 0$ be
a short exact sequence in $\Perv^{\{w,w+1\}}_\iota(X,\Qlb)$. As in Gabber's
proof of his theorem on independence on $\ell$ for middle extensions
\cite[Theorem~3]{Gabber}, up to shrinking $Z$, we may assume that $Z$ is
smooth equidimensional and that $\cH^n i^* \rR f_* A_j$ is  lisse for every
$j$ and every $n$. It follows that the distinguished triangle
\[i_*\rR i^! f_{!*}A_j \to f_{!*}A_j\to \rR f_* A_j\to\]
induces isomorphisms $f_{!*}A_j\simto \prescript{P}{}\tau^{\le -d-1} \rR f_*
A_j$ and $\prescript{P}{}\tau^{\ge -d}\rR f_* A_j\simto i_* \rR i^!
f_{!*}A_j[1]$ for every $j$. Here $P$ denotes the $t$-structure obtained by
gluing $(\rD^b_c(X,\Qlb),0)$ and the canonical $t$-structure on
$\rD^b_c(Z,\Qlb)$. Thus $\PR^{-d-1} f_* A_j\simeq i_* \cH^{-d-1}i^*
f_{!*}A_j$ has punctual $\iota$-weights $\le w-d$, while $\PR^{-d} f_*
A_j\simeq i_*\cH^{-d+1}\rR i^! f_{!*}A_j$ has punctual $\iota$-weights $\ge
w-d+1$. Therefore, the morphism $\PR^{-d-1} f_* A_3\to \PR^{-d} f_* A_1$ is
zero. Applying Lemma \ref{l.nine} below, we get a distinguished triangle
\[
\prescript{P}{}\tau^{\le -d-1} \rR f_* A_1\to \prescript{P}{}\tau^{\le
-d-1} \rR f_* A_2\to \prescript{P}{}\tau^{\le -d-1} \rR f_* A_3\to.\]
Taking perverse cohomology sheaves, we get the exactness of the sequence
\[0\to f_{!*}A_1\to f_{!*}A_2\to f_{!*} A_3\to 0.\]
\end{proof}

\begin{lemma}\label{l.nine}
Let $P$ be a $t$-structure on a triangulated category $\cD$ and let
$A\xrightarrow{a} B\xrightarrow{b} C\xrightarrow{c} A[1]$ be a distinguished
triangle such that $\PH^0c\colon \PH^0C\to \PH^1A$ is zero. Then there
exists a unique nine-diagram of the form
\begin{equation}\label{e.nine}
\xymatrix@C=3em{\Ptau^{\le 0} A \ar[r]^{\Ptau^{\le 0} a}\ar[d] & \Ptau^{\le 0} B \ar[r]^{\Ptau^{\le
0}b}\ar[d] & \Ptau^{\le 0} C \ar[r]^{c_0}\ar[d]^u\ar@{}[rd]|{(*)} & (\Ptau^{\le 0} A)[1]\ar[d]\\
A\ar[r]^a\ar[d] & B\ar[r]^b\ar[d] & C\ar[r]^c\ar[d]\ar@{}[rd]|{(**)} & A[1]\ar[d]^v\\
\Ptau^{\ge 1} A \ar[r]^{\Ptau^{\ge 1} a}\ar[d] & \Ptau^{\ge 1} B \ar[r]^{\Ptau^{\ge 1} b}\ar[d]
& \Ptau^{\ge 1} C \ar[r]^{c_1}\ar[d] & (\Ptau^{\ge 1} A)[1]\ar[d]\\
(\Ptau^{\le 0} A)[1] \ar[r]^{(\Ptau^{\le 0} a)[1]} & (\Ptau^{\le 0} B)[1] \ar[r]^{(\Ptau^{\le
0}b)[1]} & (\Ptau^{\le 0} C)[1] \ar[r]^{c_0[1]} & (\Ptau^{\le 0} A)[2],
}
\end{equation}
where the columns are the canonical distinguished triangles.
\end{lemma}

By a \emph{nine-diagram} in a triangulated category (cf.\ \cite[Proposition
1.1.11]{BBD}), we mean a diagram
\[
\xymatrix@C=2em{ A \ar[r] \ar[d] & B \ar[r] \ar[d] & C \ar[r] \ar[d] & A[1] \ar@{-->}[d] \\
A' \ar[r] \ar[d] & B' \ar[r] \ar[d] & C' \ar[r] \ar[d] & A'[1] \ar@{-->}[d] \\
A'' \ar[r] \ar[d] & B'' \ar[r] \ar[d]
& C'' \ar[r] \ar[d] & A''[1] \ar@{-->}[d] \\
A[1] \ar@{-->}[r] & B[1] \ar@{-->}[r] & C[1] \ar@{-->}[r] & A[2] \ar@{}[ul]|{-},
}
\]
in which the square marked with ``$-$'' is anticommutative and all other
squares are commutative, the dashed arrows are induced from the solid ones
by translation, and the rows and columns in solid arrows are distinguished
triangles.

\begin{proof}
First note that $vcu$ is the image of $\PH^0c$ under the isomorphism
\[\Hom(\PH^0C, \PH^1A)\simto \Hom(\Ptau^{\le 0} C, (\Ptau^{\ge 1} A)[1]).\]
Hence $vcu=0$. Moreover, $\Hom(\Ptau^{\le 0} C, \Ptau^{\ge 1} A)=0$. Thus by
\cite[Proposition 1.1.9]{BBD}, there exist a unique $c_0$ making $(*)$
commutative and a unique $c_1$ making $(**)$ commutative. This proves the
uniqueness of \eqref{e.nine}. It remains to show that \eqref{e.nine} thus
constructed is a nine-diagram. To do this, we extend the upper left square
of \eqref{e.nine} into a nine-diagram
\begin{equation}\label{e.nine2}
\xymatrix@C=3em{\Ptau^{\le 0} A \ar[r]^{\Ptau^{\le 0} a}\ar[d] & \Ptau^{\le 0} B \ar[r]\ar[d] & C_0 \ar[r]\ar[d] & (\Ptau^{\le 0} A)[1]\ar[d]\\
A\ar[r]^a\ar[d] & B\ar[r]^b\ar[d] & C\ar[r]^c\ar[d]\ar@{}[rd]|{(***)} & A[1]\ar[d]\\
\Ptau^{\ge 1} A \ar[r]^{\Ptau^{\ge 1} a}\ar[d] & \Ptau^{\ge 1} B \ar[r]\ar[d]
& C_1 \ar[r]\ar[d] & (\Ptau^{\ge 1} A)[1]\ar[d]\\
(\Ptau^{\le 0} A)[1] \ar[r]^{(\Ptau^{\le 0} a)[1]} & (\Ptau^{\le 0} B)[1] \ar[r] & C_0[1] \ar[r] & (\Ptau^{\le 0} A)[2].
}
\end{equation}
By the first and third rows of \eqref{e.nine2}, $C_0\in \PcD^{\le 0}$ and
$C_1\in \PcD^{\ge 0}$. Taking $\PH^0$ of $(***)$, we obtain a commutative
diagram
\[\xymatrix{\PH^0 C\ar[r]^0\ar[d]_e & \PH^1 A\ar@{=}[d]\\
\PH^0 C_1 \ar[r]^d & \PH^1 A,}
\]
where $e$ is an epimorphism and $d$ is a monomorphism. Thus $\PH^0 C_1=0$,
so that $C_1\in \PcD^{\ge 1}$. Further applying \cite[Proposition
1.1.9]{BBD}, we may identify \eqref{e.nine2} with \eqref{e.nine}.
\end{proof}

\subsection{Statement and consequences of main result}\label{ss.G2}

In this subsection, we define a subgroup $\rK_{\iota,\sigma}$ of the
Grothendieck group and state its preservation by Grothendieck's six
operations (Theorem \ref{t.sixop}), which contains the finite field case of
Theorem \ref{t.i3}. We then give a number of consequences and discuss the
relationship with independence of $\ell$ and Laumon's theorem on Euler
characteristics.

\begin{definition}[$\rK_{\iota,\sigma}$]\label{d.K}
We define $\rK_{\iota,\sigma}^w(X,\Qlb)\subseteq \rK_\iota^w(X,\Qlb)$
(resp.\ $\rK_{\iota,\sd}^w(X,\Qlb)\subseteq \rK_\iota^w(X,\Qlb)$), for
$w\in\Z$, to be the subgroup generated by $[B]$, for $B$ perverse,
$\iota$-pure of weight $w$, and $(-1)^w\sigma$-self-dual (resp.\ self-dual)
with respect to $K_X(-w)$. We put
\[
\rK_{\iota,\sigma}(X,\Qlb)=\bigoplus_{w\in \Z} \rK_{\iota,\sigma}^w(X,\Qlb)\quad
\text{(resp.}\ \rK_{\iota,\sd}(X,\Qlb)=\bigoplus_{w\in \Z}
\rK_{\iota,\sd}^w(X,\Qlb)).
\]
We define the \emph{twisted dualizing map}
\[
\Dbar_{\iota,X}\colon \rK^\Z_\iota(X,\Qlb)\to \rK^\Z_\iota(X,\Qlb)
\]
to be the direct sum of the group automorphisms $\Dbar_{\iota,X}^w\colon
\rK^w_\iota(X,\Qlb)\to \rK^w_\iota(X,\Qlb)$ sending $[A]$ to $[(D_X
A)(-w)]$. We let $\rK_{\iota,\dl}^w(X,\Qlb)\subseteq\rK_{\iota}^w(X,\Qlb)^2$
denote the graph of $\Dbar_{\iota,X}^w$. We put
\[\rK_{\iota,\dl}(X,\Qlb)=\bigoplus_{w\in \Z}\rK^w_{\iota,\dl}(X,\Qlb).\]
\end{definition}

We note that $\Dbar_{\iota, X}\Dbar_{\iota, X}=\id$, and
$\rK_{\iota,\dl}(X,\Qlb)\subseteq \rK^\Z_{\iota}(X,\Qlb)^2$ is the graph of
$\Dbar_{\iota,X}$.

\begin{example}\label{e.K}
For $X=\Spec(\F_q)$, an element $A\in \rK(X,\Qlb)$ is determined by the
determinant
\[P(A,T)\coloneqq \det (1-T\Fr_q\mid A_{\Fqb})\in \Qlb(T).\]
Assume $A\in \rK^w_\iota(X,\Qlb)$, $w\in \Z$. For $\lambda\in \Qlb$
satisfying $\lvert \iota(\lambda)\rvert =q^{w/2}$, we let $m_\lambda$ and
$m'_\lambda$ denote the order at $T=1/\lambda$ of $P(A,T)$ and
$P(\Dbar_{\iota,X}A,T)$, respectively. We then have
$m_{\lambda}=m'_{q^w/\lambda};$ in other words,
\begin{equation}\label{e.Kdet}
\iota P(A,T)=\bar\iota P(\Dbar_{\iota,X}A,T).
\end{equation}
We have $\rK^w_{\iota,(-1)^{w+1}}(X,\Qlb)\subseteq
\rK^w_{\iota,(-1)^w}(X,\Qlb)=\rK^w_{\iota,\sd}(X,\Qlb)$. The following
conditions are equivalent:
\begin{enumerate}
\item $A\in \rK^w_{\iota,(-1)^w}(X,\Qlb)=\rK^w_{\iota,\sd}(X,\Qlb)$;
\item $m_\lambda=m_{q^w/\lambda}$ for all $\lambda$;
\item $\iota
    P(A,T)\in \R(T)$.
\end{enumerate}
Furthermore, the following conditions are equivalent:
\begin{enumerate}
\item $A\in\rK^w_{\iota,(-1)^{w+1}}(X,\Qlb)$;
\item $m_\lambda=m_{q^w/\lambda}$ for all $\lambda$ and $m_{q^{w/2}}$,
    $m_{-q^{w/2}}$ are even;
\item $\iota P(A,T)\in \R(T)$, the rank $b=\sum_\lambda m_\lambda \in \Z$
    of $A$ is even, and $\det(\Fr_q \mid A_{\Fqb})=q^{wb/2}$.
\end{enumerate}
\end{example}

\begin{remark}\label{r.K0}
Let $w\in\Z$.
\begin{enumerate}
\item By definition, $\rK_{\iota,\sigma}(X,\Qlb)\subseteq \rK(X,\Qlb)$
    (resp.\ $\rK_{\iota,\sd}(X,\Qlb)\subseteq \rK(X,\Qlb)$) is generated
    by the image of $\rD^w_{\iota,\sigma}(X,\Qlb)$ (resp.\
    $\rD^w_{\iota,\sd}(X,\Qlb)$), as defined in Definition \ref{d.Diota}.
    Moreover, $\rK_{\iota,\dl}(X,\Qlb)\subseteq \rK(X,\Qlb)^2$ is
    generated by the image of $\rD^w_{\iota,\dl}(X,\Qlb)$.

\item By Remark \ref{r.ss}, in the definition of $\rK^w_{\iota,\sigma}$,
    one may restrict to semisimple perverse sheaves. This also holds for
    $\rK^w_{\iota,\sd}$. Thus $\rK^w_{\iota,\sigma}(X,\Qlb)$ (resp.\
    $\rK^w_{\iota,\sd}(X,\Qlb)$) is generated by $[A]+[(D_X A)(-w)]$ for
    $A$ simple perverse $\iota$-pure of weight $w$, and $[B]$ for $B$
    simple perverse $\iota$-pure of weight $w$ and
    $(-1)^{w}\sigma$-self-dual (resp.\ self-dual) with respect to
    $K_X(-w)$.

\item By Proposition \ref{p.schur}, we have
    $\rK^w_{\iota,\sd}(X,\Qlb)=\rK^w_{\iota,1}(X,\Qlb)+
    \rK^w_{\iota,-1}(X,\Qlb)$.

\item $\rK_{\iota,\dl}(X,\Qlb)$ is generated by $([B],[(D_X A)(-w)])$ for
    $B$ simple perverse $\iota$-pure of weight $w$. Thus
    $\rK_{\iota,\sd}(X,\Qlb)=\Delta^{-1} (\rK_{\iota,\dl}(X,\Qlb))$, where
    $\Delta\colon \rK^\Z_\iota(X,\Qlb)\to \rK^\Z_\iota(X,\Qlb)^2$ is the
    diagonal embedding. In other words, for $A\in \rK^\Z_\iota(X,\Qlb)$,
    $A$ belongs to $\rK_{\iota,\sd}(X,\Qlb)$ if and only if
    $A=\Dbar_{\iota,X}A$.

\item For $A\in \rK_\iota^\Z(X,\Qlb)$ and $n\in \Z$, we have
    $\Dbar_{\iota,X} (A(n))=(\Dbar_{\iota,X} A)(n)$. For $A\in
    \rK_{\iota,\sigma}^w(X,\Qlb)$, we have $A(n)\in
    \rK_{\iota,\sigma}^{w-2n}(X,\Qlb)$.

\item Let $A$ be a perverse sheaf on $X$, $\iota$-pure of weight $w$. By
    Corollary \ref{c.even}, $[A]\in \rK_{\iota,\sigma}^w(X,\Qlb)$ if and
    only if the semisimplification of $A$ is $(-1)^w\sigma$-self-dual with
    respect to $K_X(-w)$. Similar results hold for $\rK_{\iota,\sd}$ and
    $\rK_{\iota,\dl}$.
\end{enumerate}
\end{remark}

\begin{remark}
Although we do not need it in the sequel, let us give two more descriptions
of $\rK_{\iota,\sigma}^w$. In our definition of $\rK_{\iota,\sigma}^w$, we
consider self-dual perverse sheaves $B$ only and do not take the bilinear
form $B\simeq D_X B (-w)$ as part of the data. Alternatively, we can also
include the form and consider the Grothendieck group $\GS$ of symmetric
spaces in $\Perv_\iota^{\{w\}}$ (equipped with the duality $D_X(-w)$ and the
evaluation map modified by a factor of $(-1)^{w}\sigma$). The
Grothendieck-Witt group $\GW$ is a quotient of $\GS$, equipped with a
homomorphism $\GW\to \rK_{\iota}^w$. We refer the reader to \cite[page
280]{QSS}, \cite[Section 2.2]{Schlichting} for the definition of the
Grothendieck-Witt group of an Abelian category with duality (generalizing
Quillen's definition \cite[(5.1)]{Quillen} for representations). In our
situation, the canonical maps
\[\GS\to \GW\to \rK_{\iota,\sigma}^w\]
are isomorphisms. In fact, by definition, $\rK_{\iota,\sigma}^w$ is the
image of $\GW$. Moreover, since we work over the algebraically closed field
$\Qlb$, symmetric spaces with isomorphic underlying objects are isometric
\cite[3.4 (3)]{QSS}.
\end{remark}

We now consider preservation of $\rK_{\iota,\sigma}$ and $\rK_{\iota,\dl}$
by cohomological operations. The preservation of $\rK_{\iota,\dl}$ is
equivalent to the commutation with the twisted dualizing map
$\Dbar_{\iota}$. The main result of this section is the following
generalization of Theorem \ref{t.i3}.

\begin{theorem}\label{t.sixop}
Let $f\colon X\to Y$ be a morphism between Deligne-Mumford stacks of finite
inertia and finite presentation over $\F_q$. Then Grothendieck's six
operations induce maps
\begin{gather*}
-\otimes-,\ \cHom(-,-)\colon \rK_{\iota,\sigma}(X,\Qlb)\times \rK_{\iota,\sigma'}(X,\Qlb)\to \rK_{\iota,\sigma\sigma'}(X,\Qlb),\\
f^*,f^!\colon
\rK_{\iota,\sigma}(Y,\Qlb)\to \rK_{\iota,\sigma}(X,\Qlb), \qquad f_*,f_!\colon \rK_{\iota,\sigma}(X,\Qlb)\to \rK_{\iota,\sigma}(Y,\Qlb).
\end{gather*}
Moreover, Grothendieck's six operations on $\rK^\Z_\iota$ commute with the
twisted dualizing map $\Dbar_\iota$.
\end{theorem}

The proof will be given in the next section. We now make a list of  pure
cases in which the preservation has already been established. Most items of
the list below follow from Remark \ref{r.Diota}.

\begin{remark}[Preservation of $\rK_{\iota,\sigma}$, pure cases]\label{r.K}
Let $f\colon X\to Y$ be a morphism, $w,w'\in\Z$.
\begin{enumerate}
\item $D_X$ carries $\rK^w_{\iota,\sigma}(X,\Qlb)$ to
    $\rK^{-w}_{\iota,\sigma}(X,\Qlb)$.

\item If $f$ is smooth, then $f^*$ preserves $\rK_{\iota,\sigma}^w$.

\item If $f$ is an open immersion, then $f_{!*}$ preserves
    $\rK^w_{\iota,\sigma}$.

\item The functor $-\boxtimes-$ carries
    $\rK_{\iota,\sigma}^w(X,\Qlb)\times \rK_{\iota,\sigma'}^{w'}(X',\Qlb)$
    to $\rK_{\iota,\sigma \sigma'}^{w+w'}(X\times X',\Qlb)$, and the
    functor $(-)^{\boxtimes m}$, $m\ge 0$, carries
    $\rK_{\iota,\sigma}^w(X,\Qlb)$ to
    $\rK_{\iota,\sigma^m}^{mw}([X^m/\Perm_m],\Qlb)$. For the latter we
    used Proposition \ref{p.boxtimesf}.

\item Assume that $f$ is a closed immersion and let $A\in \rK(X,\Qlb)$.
    Then $A\in \rK_{\iota,\sigma}^w(X,\Qlb)$ if and only if
    $f_*A\in\rK^w_{\iota,\sigma}(Y,\Qlb)$.

\item Assume that $f$ is proper.  If $f$ is projective or $Y$ has finite
    inertia, then $f_*$ preserves $\rK^w_{\iota,\sigma}$, by Proposition
    \ref{p.projf} and Theorem \ref{t.mainf}.
\end{enumerate}
Similar properties hold for $\rK_{\iota,\dl}$.
\end{remark}

The following Zariski local nature of $\rK_{\iota,\sigma}^w$ will be used in
the proof of Theorem \ref{t.sixop}. It follows from the Zariski local nature
of $\sigma$-self-dual perverse sheaves (Proposition \ref{p.23}). It also
follows from Remark \ref{r.K} (3) and Lemma \ref{l.coverK}.

\begin{remark}[Zariski local nature]\label{r.ZarK}
Let $(X_\alpha)_{\alpha\in I}$ be a Zariski open covering of $X$ and let
$A\in \rK(X,\Qlb)$. Then $A\in \rK_{\iota,\sigma}^w(X,\Qlb)$ if and only if
$A\res{X_\alpha}\in \rK_{\iota,\sigma}^w(X_\alpha,\Qlb)$ for every $\alpha$.
The same holds for $\rK_{\iota,\dl}$.
\end{remark}

We now turn to consequences of Theorem \ref{t.sixop}. The ring part of the
following two corollaries follow from the two assertions of Theorem
\ref{t.sixop} applied to $a_X^*$ (recall $a_X\colon X\to \Spec(\F_q)$) and
$-\otimes-$. For the $\lambda$-ring part, we apply Remark \ref{r.K} (4) to
the map $(-)^{\boxtimes m}$ and Theorem \ref{t.sixop} to the maps $\Delta^*$
and $p_*$ in the definition of $\lambda^m$ in Construction \ref{c.lambda}.

\begin{cor}\label{c.lsubring}
Assume that $X$ has finite inertia. Then $\rK_{\iota,1}(X,\Qlb)$ is a
$\lambda$-subring of $\rK(X,\Qlb)$. In particular, $\rK_{\iota,1}(X,\Qlb)$
contains the class $[\Qlb]$ of the constant sheaf $\Qlb$ on $X$.
\end{cor}

\begin{cor}\label{c.lringhom}
Assume that $X$ has finite inertia. Then $\Dbar_{\iota,X}\colon
\rK^\Z_\iota(X,\Qlb)\to \rK^\Z_\iota(X,\Qlb)$ is a $\lambda$-ring
homomorphism. In particular, $\Dbar_{\iota,X}[\Qlb]=[\Qlb]$.
\end{cor}

Another consequence of Theorem \ref{t.sixop} is the following pointwise
characterization of $\rK_{\iota,\dl}$ and $\rK_{\iota,\sd}$. We let
$\rK_{\iota,\bar \iota}(X,\Qlb)\subseteq \rK(X,\Qlb)^2$ (resp.\
$\rK_{\iota,\R}(X,\Qlb)\subseteq \rK(X,\Qlb)$) denote the subgroup
consisting of elements $(A,A')$ (resp.\ $A$) such that for every morphism
$x\colon\Spec(\F_{q^m})\to X$ and every geometric point $\bar x$ above $x$,
we have
\[\iota\tr(\Fr_x,A_{\bar x})=\bar \iota \tr(\Fr_x,A'_{\bar x})\quad \text{(resp.}\ \iota\tr(\Fr_x,A_{\bar x})\in\R\text{)}.
\]
The notation $\rK_{\iota,\bar \iota}$ and $\rK_{\iota,\R}$ will only be used
in Corollary \ref{c.R} and Remark \ref{r.Gabberind} below.

\begin{cor}\label{c.R}
Assume that $X$ has finite inertia. Let $A\in \rK^\Z_\iota(X,\Qlb)$. Then
for every $m\ge 1$, every morphism $x\colon \Spec(\F_{q^m})\to X$ and every
geometric point $\bar x$ above $x$, we have
\begin{equation}\label{e.R}
\iota\tr(\Fr_x,A_{\bar x}) = \bar\iota \tr(\Fr_x,(\Dbar_{\iota,X}A)_{\bar x}).
\end{equation}
Moreover, $\rK_{\iota,\dl}(X,\Qlb)=\rK_{\iota,\bar\iota}(X,\Qlb)\cap
\rK_{\iota}^\Z(X,\Qlb)^2$. In particular,
$\rK_{\iota,\sd}(X,\Qlb)=\rK_{\iota,\R}\cap \rK_{\iota}^\Z(X,\Qlb)$.
\end{cor}

\begin{proof}
By the second assertion of Theorem \ref{t.sixop} applied to $x^*$, we have
$\Dbar_{\iota,\Spec(\F_{q^m})} x^* A=x^*\Dbar_{\iota,X} A$. Thus
\eqref{e.Kdet} in Example \ref{e.K} implies \eqref{e.R}. It follows that
$\rK_{\iota,\dl}(X,\Qlb)\subseteq \rK_{\iota,\bar\iota}(X,\Qlb)\cap
\rK_{\iota}^\Z(X,\Qlb)^2$. The inclusion in the other direction follows from
the injectivity of the homomorphism $\rK(X,\Qlb)\to \Map(\coprod_{m\ge 1}
\lvert X(\F_{q^m})\rvert,\Qlb)$ (Lemma \ref{l.inj}). The last assertion of
Corollary \ref{c.R} follows from the second one.
\end{proof}

\begin{remark}\label{r.Gabberind}
Corollary \ref{c.R} also follows from \cite[Lemma 1.8.1 1), 4)]{Katz}, which
in turn follows from Gabber's theorem on independence of $\ell$ for middle
extensions \cite[Theorem~3]{Gabber}. By Gabber's theorem on independence of
$\ell$ for Grothendieck's six operations \cite[Theorem~2]{Gabber} (see
\cite[3.2]{Zind} for a different proof and \cite[Proposition 5.8]{Zind} for
the case of stacks), $\rK_{\iota,\bar\iota}$ and $\rK_{\iota,\R}$ in
Corollary \ref{c.R} are stable under the six operations. Thus the second
assertion of Theorem \ref{t.sixop} follows from Gabber's theorems on
independence of $\ell$. We will not use Gabber's theorems on independence of
$\ell$ in our proof of Theorem \ref{t.sixop}.
\end{remark}

\begin{remark}
The pointwise characterization of $\rK_{\iota,\sd}$ in Corollary \ref{c.R}
does not extend to $\rK_{\iota,\sigma}$. For instance, if $X$ is regular and
geometrically connected and if $f\colon E\to X$ is a family of elliptic
curves with nonconstant $j$-invariant, then $\cF=\rR^1 f_*\Qlb$ is a
geometrically simple lisse $\Qlb$-sheaf on $X$ by \cite[Lemme
3.5.5]{WeilII}, so that $[\cF]\in \rK^1_{\iota,1}(X,\Qlb)\backslash
\rK^1_{\iota,-1}(X,\Qlb)$, but for every closed point $x$ of $X$, we have
$[\cF_x]\in \rK^1_{\iota,1}(x,\Qlb)\subseteq \rK^1_{\iota,-1}(x,\Qlb)$.
\end{remark}

\begin{remark}\label{r.Laumon}
Let $f\colon X\to Y$ be as in Theorem \ref{t.sixop} and let
$\rI(Y,\Qlb)\subseteq \rK(Y,\Qlb)$ be the ideal generated by
$[\Qlb(1)]-[\Qlb]$. A theorem of Laumon \cite{Laumon} (\cite[Theorem
3.2]{IZ} for the case of Deligne-Mumford stacks) states that $f_*\equiv f_!$
modulo $\rI(Y,\Qlb)$. This is equivalent to the congruence $D_Y f_* \equiv
f_* D_X$ (and to $D_Y f_! \equiv f_! D_X$) modulo $\rI(Y,\Qlb)$. Thus the
second assertion of Theorem \ref{t.sixop} can be seen as a refinement of
Laumon's theorem.
\end{remark}

In the case of $\rK_{\iota,\sigma}$, we have the following result on
independence of $(\ell,\iota)$. Let $\ell'\neq q$ be a prime number and let
$\iota'\colon \overline{\Q}_{\ell'}\to \cC$ be an embedding.

\begin{cor}\label{c.indep}
Assume that $X$ has finite inertia. Let $A\in \rK^\Z_\iota(X,\Qlb)$, $A'\in
\rK^\Z_{\iota'}(X,\overline{\Q}_{\ell'})$. Assume that $A$ and $A'$ are
compatible in the sense that for every morphism $x\colon\Spec(\F_{q^m})\to
X$ and every geometric point $\bar x$ above $x$, we have
\[\iota\tr(\Fr_x,A_{\bar x})=\iota' \tr(\Fr_x,A'_{\bar x}).
\]
Then $A$ belongs to $\rK_{\iota,\sigma}(X,\Qlb)$ if and only if $A'$ belongs
to $\rK_{\iota',\sigma}(X,\overline{\Q}_{\ell'})$.
\end{cor}

\begin{proof}
Let $(X_\alpha)_{\alpha\in I}$ be a stratification of $X$. Then
$A=\sum_{\alpha\in I} j_{\alpha!}j_\alpha^* A$ and $A'=\sum_{\alpha\in I}
j_{\alpha!}j_\alpha^* A'$, where $j_\alpha\colon X_\alpha\to X$ is the
immersion. Thus, by Theorem \ref{t.sixop}, up to replacing $X$ by a stratum,
we may assume that $X$ is regular and $A$ belongs to the subgroup generated
by lisse $\Qlb$-sheaves, that is, $A=\sum_\cF n_\cF [\cF]$, where $\cF$ runs
over isomorphism classes of simple lisse $\Qlb$-sheaves. For each $\cF$
appearing in the decomposition, let $\cF'$ be the companion of $\cF$
\cite{Dr} (\cite{Zhcomp} for the case of stacks), namely the simple lisse
Weil $\overline{\Q}_{\ell'}$-sheaf such that
\[\iota\tr(\Fr_x,\cF_{\bar x})=\iota' \tr(\Fr_x,\cF'_{\bar x})
\]
for all $x$ and $\bar x$ as above. Since $A'=\sum_\cF n_\cF [\cF']$, each
$\cF'$ is an honest $\overline{\Q}_{\ell'}$-sheaf. By Corollary \ref{c.R},
we have $(\Dbar\cF)'=\Dbar(\cF')$. Therefore, we may assume that $A=[\cF]$
and $A'=[\cF']$. In this case, the assertion follows from the symmetry
criterion in terms of squares of Frobenius (Proposition \ref{p.square}).
\end{proof}

\subsection{Proof of main result}\label{ss.G3}

The situation of Theorem \ref{t.sixop} is quite different from that of
Gabber's theorem on independence of $\ell$ \cite[Theorem~2]{Gabber}. In
Gabber's theorem, the preservation by $-\otimes -$ and $f^*$ is trivial and
the preservation by $f_!$ follows from the Grothendieck trace formula. The
key point of Gabber's theorem is thus the preservation by $D_X$. The
preservation by middle extensions \cite[Theorem~3]{Gabber} follows from the
preservation by the six operations. In Theorem \ref{t.sixop}, the stability
under each of the six operations is nontrivial, but the preservation by
$D_X$ and middle extensions is easy. To prove Theorem \ref{t.sixop}, we will
first deduce that $f_!$ preserves $\rK_{\iota,\sigma}$ and $\rK_{\iota,\dl}$
in an important special case from the preservation by middle extensions.

\begin{prop}\label{p.NCD}
Let $X$ be a regular Deligne-Mumford stack of finite presentation over
$\F_q$ and let $D=\sum_{\alpha\in I}D_\alpha$ be a strict normal crossing
divisor, with $D_\alpha$ regular. Assume that there exists a finite \'etale
morphism $f\colon Y\to X$ such that $f^{-1}(D_\alpha)$ is defined globally
by $t_\alpha\in \Gamma(Y,\mathcal{O}_Y)$ for all $\alpha\in I$. Let $\cF$ be
a lisse $\Qlb$-sheaf on $U=X-D$, tamely ramified along $D$. Assume that
$[\cF]\in \rK^\Z_{\iota}(U,\Qlb)$. Then
$\Dbar_{\iota,X}[j_!\cF]=j_!\Dbar_{\iota,U}[\cF]$ and, if $[\cF]\in
\rK_{\iota,\sigma}(U,\Qlb)$, $[j_!\cF]$ belongs to
$\rK_{\iota,\sigma}(X,\Qlb)$, where $j\colon U\to X$ is the open immersion.
\end{prop}

[Added after publication: the assumption on the existence of $Y$ is
superfluous. The proof has been corrected accordingly.]

\begin{proof}
We will prove the case of $\rK_{\iota,\sigma}$. The case of
$\rK_{\iota,\dl}$ is similar.

We apply the construction of \cite[1.7.9, 1.7.10]{WeilII} to our setting as
follows. For $J\subseteq I$, let $U_J=X-\bigcup_{\beta \in I-J} D_\beta$ and
let $D_J^*=\bigcap_{\beta\in J}D_\beta\cap U_J$. Let $p_\beta\colon
T_\beta\to D_{\{\beta\}}^*$ be the normal bundle of $D_{\{\beta\}}^*$ in
$U_\beta=U_{\{\beta\}}$. If $D_\beta$ is defined by $t_\beta=0$, then for
each locally constant constructible sheaf of sets $\cG$ on $U$, tamely
ramified along $D_{\{\beta\}}^*$, there exists an integer $n$ invertible in
$\F_q$ such that $\cG$ extends to $\cG'$ on the cover
$U_\beta[t_\beta^{1/n}]$ of $U_\beta$, and we let $\cG[D_{\{\beta\}}^*]$
denote the restriction of $\cG'$ to $D_{\{\beta\}}^*$, which is locally
constant constructible and equipped with an action of $\mu_n$. In general we
have a locally constant constructible sheaf $\cG[D_{\{\beta\}}^*]'$ on
$T_\beta$, obtained by applying the above construction to the deformation to
the normal bundle and the pullback of $\cG$. Extending this construction to
$\Qlb$-sheaves by taking limits, we obtain a lisse $\Qlb$-sheaf
$\cF[D^*_{\{\beta\}}]'$ on $T_\beta$ endowed with an action of
$G_{\{\beta\}}=\hat\Z_L(1)$. Here $L$ denotes the set of primes invertible
in $\F_q$.

Let us first show that for all $J\subseteq I$, $[j_J^*Rj_*\cF]\in
\rK_{\iota,\sigma}(D_J^*,\Qlb)$, where $j_J\colon D_J^*\to X$ is the
immersion. We proceed by induction on $\# J$. The assertion is trivial for
$J$ empty, as $j_\emptyset=j$. For $J$ nonempty, choose $\beta\in J$.
Consider the diagram with Cartesian square
\[\xymatrix{&& D_J^*\ar[d]^{j'_{\beta,J'}}\\
& D_{\{\beta\}}^*\ar[d]_{i'_\beta}\ar[r]^{j'_{\beta}}\ar[rd]^{j_{\{\beta\}}} & D_\beta\ar[d]^{i_\beta}\\
U\ar[r]^{j^\beta} & U_\beta \ar[r]_{j_\beta} & X,}
\]
where $U_\beta=X-\bigcup_{\alpha\in I-\{\beta\}} D_\alpha$,
$D_{\{\beta\}}=D_\beta\cap U_\beta$, and $J'=J-\{\beta\}$. By \cite[Lemme
3.7]{Zint} (or by direct computation using \cite[1.7.9]{WeilII}), the base
change morphism
\[i_\beta^*\rR j_* \cF\to \rR j'_{\beta*} i_\beta'^* \rR j^\beta_*\cF\]
is an isomorphism, so that
\[j_J^*\rR j_*\cF\simeq (j'_{\beta,J'})^*i_\beta^*\rR j_* \cF\simeq (j'_{\beta,J'})^*Rj'_{\beta*} i'^*_\beta\rR j^\beta_* \cF.\]
Since $\#J'<\# J$, by induction hypothesis applied to $j'_{\beta}$ and $J'$,
it suffices to show that the class of $i_\beta'^* \rR j^\beta_*\cF \simeq
j_{\{\beta\}}^* \rR j_* \cF$ is in
$\rK_{\iota,\sigma}(D_{\{\beta\}}^*,\Qlb)$. For this, we may assume $\cF$
$\iota$-pure of weight $w\in \Z$. Let $H_\beta<=\hat \Z_L(1)=G_{\{\beta\}}$
be an open subgroup whose action on $V=\cF[D^*_{\{\beta\}}]'$ is unipotent.
Let $N\colon V(1)\to V$ be the logarithm of this action and let $M$ be the
monodromy filtration on $V$. We have $j_{\{\beta\}}^* \rR j_*\cF\in
\rD^{[0,1]}$,
\begin{gather*}
p_{\beta}^*j_{\{\beta\}}^* j_* \cF\simeq (\Ker(N)(-1))^{G_{\{\beta\}}/H_\beta},
\qquad  p_{\beta}^*j_{\{\beta\}}^* \rR^1 j_*\cF\simeq (\Coker(N))^{G_{\{\beta\}}/H_\beta}(-1),\\
\gr^M_i(\Ker(N)(-1))\simeq \begin{cases}P_i(V,N)& i\le 0,\\ 0& i> 0,\end{cases}\qquad
\gr^M_i(\Coker(N))\simeq \begin{cases}P_{-i}(V,N)(-i)& i\ge 0,\\ 0& i< 0.\end{cases}
\end{gather*}
By \cite[Corollaire 1.8.7, Remarque 1.8.8]{WeilII}, $P_i(V,N)$ is pure of
weight $w+i$ for $i\le 0$. Moreover, $P_i(V,N)$ is
$(-1)^{w+i}\sigma$-self-dual by Proposition \ref{p.Pi}. It follows that
$[j_{\{\beta\}}^* j_* \cF], [j_{\{\beta\}}^* \rR^1 j_*\cF]\in
\rK_{\iota,\sigma}$.

Next we show that, if $\cF$ is $\iota$-pure of weight $w\in \Z$, then for
all $n\ge 0$ and $J\subseteq I$, $[j_J^*\rR j_{n*}j^n_{!*} \cF]\in
\rK_{\iota,\sigma}(D_J^*,\Qlb)$. Here $U\xrightarrow{j^n} U_n
\xrightarrow{j_n} X$ are immersions, $U_n=X-\bigcup_{K\subseteq I,\ \#K\ge
n} D_K^*$, $j^n_{!*}\cF\coloneqq (j^n_{!*}(\cF[d]))[-d]$, $d=\dim(X)$ (a
function on $\pi_0(X)$). The proof is similar to Gabber's proof of
independence of $\ell$ for middle extensions \cite[Theorem 3]{Gabber}. We
proceed by induction on $n$. For $n=0$, $U_0=U$ and the assertion is shown
in the preceding paragraph. For $n\ge 1$, consider the immersions
$U_n-U_{n-1}\xrightarrow{i_n} U_n \xleftarrow{j_{n-1}^n} U_{n-1}$ and the
distinguished triangle
\[i_{n*}\rR i_n^! j^n_{!*} \cF \to j^n_{!*} \cF \to \rR (j_{n-1}^n )_* j^{n-1}_{!*} \cF\to.\]
The second and third arrows of the triangle induce isomorphisms $j^n_{!*}
\cF \simto \tau^{\le n-1}\rR (j_{n-1}^n )_* j^{n-1}_{!*} \cF$ and $\tau^{\ge
n}\rR (j_{n-1}^n )_* j^{n-1}_{!*} \cF\simto i_{n*}\rR i_n^! j^n_{!*}
\cF[1]$. By induction hypothesis, the left hand side of
\[[i_n^*\rR (j_{n-1}^n )_*
j^{n-1}_{!*} \cF]=[i_n^*j^n_{!*} \cF]-[\rR i_n^!j^n_{!*}\cF]
\]
belongs to $\rK_{\iota,\sigma}$. Moreover, the first term of the right hand
side belongs to $\bigoplus_{w'\le w+n-1} \rK^{w'}_\iota$ and the second term
belongs to $\bigoplus_{w'\ge w+n+1}\rK^{w'}_\iota$. It follows that both
terms belong to $\rK_{\iota,\sigma}$. Thus, by the preceding paragraph,
$[j_J^* R(j_n i_n)_*\rR i_n^! j^n_{!*} \cF]\in \rK_{\iota,\sigma}$.
Moreover, by induction hypothesis, $[j_J^* \rR (j_{n-1})_* j^{n-1}_{!*}
\cF]\in \rK_{\iota,\sigma}$. Therefore, $[j_J^* \rR j_{n*}j^n_{!*}
\cF]=[j_J^* R(j_n i_n)_*\rR i_n^! j^n_{!*} \cF]+[j_J^* \rR (j_{n-1})_*
j^{n-1}_{!*} \cF]\in \rK_{\iota,\sigma}$.

Taking $n=1+\# I$ so that $U_n=X$ in the preceding paragraph, we get that
for $\cF$ $\iota$-pure and $J\subseteq I$, $[j_J^*j_{!*}\cF]\in
\rK_{\iota,\sigma}(D_J^*,\Qlb)$. Here $j_{!*}\cF\coloneqq
(j_{!*}(\cF[d]))[-d]$.

Finally, we show the proposition by induction on $\# I$. The assertion is
trivial for $I$ empty. For $I$ nonempty, we may assume $\cF$ $\iota$-pure.
We have
\[[j_{!*}\cF]=[j_!\cF]+\sum_{\emptyset \neq J\subseteq I} [j_{J!}j_J^*j_{!*}\cF].\]
By the preceding paragraph and induction hypothesis, for $\emptyset \neq
J\subseteq I$, we have $[j_{J!}j_J^*j_{!*}\cF] \in
\rK_{\iota,\sigma}(X,\Qlb)$. Moreover, $[j_{!*}\cF]\in
\rK_{\iota,\sigma}(X,\Qlb)$. It follows that $[j_!\cF]\in
\rK_{\iota,\sigma}(X,\Qlb)$.
\end{proof}

\begin{lemma}\label{l.Chow}
Let $X$ be a Noetherian Deligne-Mumford stack of separated diagonal. Then
there exist a finite group $G$ and a $G$-equivariant dominant open immersion
$V\to W$ of schemes, such that the induced morphism $[V/G]\to [W/G]$ fits
into a commutative diagram
\[\xymatrix{[V/G]\ar[r]\ar[rd]_j & [W/G]\ar[d]^f\\
&X,}
\]
where $j$ is an open immersion and $f$ is quasi-finite, proper, and
surjective.
\end{lemma}

\begin{proof}
By \cite[16.6.3]{LMB}, there exists a finite group $G$ acting on a scheme
$V$ fitting into a Cartesian square
\[\xymatrix{V\ar[r]\ar[d] & Z\ar[d]^g\\ [V/G]\ar[r]^j & X,}\]
where $Z$ is a scheme, $g$ is finite surjective, and $j$ is a dense open
immersion. It then suffices to take $W$ to be the schematic closure of $V$
in $(Z/X)^G$ (fiber product over $X$ of copies of $Z$ indexed by $G$)
endowed with the action of $G$ by permutation of factors.
\end{proof}

\begin{proof}[Proof of Theorem \ref{t.sixop}.]
We will prove the preservation of $\rK_{\iota,\sigma}$. The commutation with
$\Dbar_{\iota}$ is similar.

(1) Let us first show the case of $f_!$ for an open immersion $f$. Since $Y$
has finite inertia, there exists a Zariski open covering $(Y_\alpha)$ of $Y$
with $Y_\alpha$ separated. By the Zariski local nature of
$\rK_{\iota,\sigma}$ (Remark \ref{r.ZarK}), we may assume $Y$ separated. We
may assume $f$ dominant.

We proceed by induction on $d=\dim X$. For $d<0$ (i.e.\ $X=\emptyset)$, the
assertion is trivial. For $d\ge 0$, let $A\in \rK_{\iota,\sigma}(X,\Qlb)$.
Note that if $A'\in \rK(X,\Qlb)$ is such that the support of $A-A'$ has
dimension $< d$, then, by induction hypothesis, to show $f_!A\in
\rK_{\iota,\sigma}(Y,\Qlb)$, it suffices to show $A'\in
\rK_{\iota,\sigma}(X,\Qlb)$ and $f_!A'\in \rK_{\iota,\sigma}(Y,\Qlb)$. This
applies in particular to $A'=j_!j^*A$, where $j\colon U\to X$ is a dominant
open immersion. In this case $f_!A'=(fj)_!j^*A$. This allows us to shrink
$X$.

Applying Lemma \ref{l.Chow} to $Y$, we obtain a finite group $G$, a
$G$-equivariant dominant open immersion of schemes $V\to W$, and a
commutative diagram
\[\xymatrix{[V/G]\ar[r]\ar[rd] & [W/G]\ar[d]^p\\
&Y,}
\]
where $p$ is proper quasi-finite surjective, and the oblique arrow is an
open immersion. Let $j\colon U\to X$, where $U=X\cap [V/G]=[V'/G]$. By the
remark above, it suffices to show that $j_!$ and $(fj)_!$ preserve
$\rK_{\iota,\sigma}$. In the case of $j_!$, up to replacing $Y$ by $X$ and
$p$ by its restriction to $X$, we are reduced to the case of $(fj)_!$. Since
$(fj)_!=f'_!p_*$, where $f'\colon U\to [W/G]$, we are reduced to the case of
$f'_!$. Thus, changing notation, we are reduced to the case of $f_!$, where
$f\colon X=[V/G]\to [W/G]=Y$ is given by a $G$-equivariant open immersion of
schemes $V\to W$.

The reduction of this case to the case where $V$ is the complement of a
$G$-strict normal crossing divisor of $W$ is similar to parts of
\cite[Section~3]{Zind}. We may assume $V$ reduced. Shrinking $V$, we may
assume $V$ normal and $A=[\cF]$, where $\cF=\cF_\cO\otimes_\cO \Qlb$, where
$\cF_\cO$ is a lisse $\cO$-sheaf, $\cO$ is the ring of integers of a finite
extension of $\mathbb{Q}_\ell$. Applying \cite[Lemme 3.5]{Zind}, we obtain a
$G$-stable dense open subscheme $U$ of $V$ and an equivariant morphism
$(u,\alpha)\colon (U',G')\to (U,G)$, where $\alpha$ is surjective and $u$ is
a Galois \'etale cover of group $\Ker(\alpha)$ trivializing $\cF_\cO/\fm
\cF_\cO$, where $\fm$ is the maximal ideal of $\cO$. By Nagata
compactification, this can be completed into a commutative diagram
\[\xymatrix{(U',G')\ar[d]_{(u,\alpha)}\ar[r]^{(f',\id)} &
(W',G')\ar[d]^{(w,\alpha)}\\
(U,G)\ar[r]& (W,G),
}
\]
where $w$ is proper and $f'$ is an open immersion. Since $[u/\alpha]$ is an
isomorphism and Remark \ref{r.K} (6) applies to $[w/\alpha]_*$, shrinking
$X$ and changing notation, we are reduced to the case where $\cF_\cO/\fm
\cF_\cO$ is constant on every connected component. We may assume $W$
reduced. Let $k'$ be a finite extension of $k$ such that the irreducible
components of $W\otimes_k k'$ are geometrically irreducible. Up to replacing
$W$ by $W\otimes_k k'$ and $G$ by $G\times \Gal(k'/k)$, we may assume that
the irreducible components of $W$ are geometrically irreducible and there
exists a $G$-equivariant morphism $W\to \Spec(k')$. Shrinking $V$, we may
assume $V$ regular. Moreover, we may assume that $G$ acts transitively on
$\pi_0(V)$. Let $V_0$ be an irreducible component of $V$ and let $G_0$ be
the decomposition group. Then $f$ can be decomposed as $X\simeq [V_0/G_0]\to
[W/G_0]\xrightarrow{g} [W/G]=Y$, where $g$ is finite. Changing notation, we
may assume $V$ irreducible. Up to replacing $W$ by the closure of $V$, we
may assume $W$ irreducible, thus geometrically irreducible.

Applying Gabber's refinement \cite[Lemme 3.8]{Zind} (see also
\cite[4.4]{Vidal}) of de Jong's alterations \cite{dJ2}, we obtain a diagram
with Cartesian square
\[\xymatrix{(U',G')\ar[r] & (V',G')\ar[d]_{(v,\alpha)}\ar[r] & (W',G')\ar[d]^{(w,\alpha)}\\
&(V,G)\ar[r] & (W,G),}
\]
where $(w,\alpha)$ is a Galois alteration, $W'$ is regular, quasi-projective
over $k$, and $U'$ is the complement of a $G'$-strict normal crossing
divisor of $W'$. As $\cF$ is lisse and $[V/G]$, $[V'/G']$ are regular,
$A'=[v/\alpha]^*A$ belongs to $\rK_{\iota,\sigma}([V'/G'],\Qlb)$, so that
$[v/\alpha]_* A$ belongs to $\rK_{\iota,\sigma}([V/G],\Qlb)$. Moreover, the
support of $A- [v/\alpha]_*A'$ has dimension $<d$. Thus it suffices to show
that $f_![v/\alpha]_*A' =[w/\alpha]_*f'_!A'$ belongs to
$\rK_{\iota,\sigma}(Y,\Qlb)$, where $f'\colon [W'/G']\to [V'/G']$. Let
$j'\colon [U'/G']\to [V'/G']$. It suffices to show that $j'_!j'^*A'$ and
$(f'j')_!j'^*A'$ belong to $\rK_{\iota,\sigma}$. Changing notation, we are
reduced to showing $f_! [\cF]\in \rK_{\iota,\sigma}([W/G],\Qlb)$ for
$f\colon [V/G]\to [W/G]$, where $V$ is the complement of a $G$-strict normal
crossing divisor $D$ of a regular quasi-projective scheme $W$ over $k$ and
$\cF$ is a lisse sheaf on $[V/G]$ tame along $D$ such that $[\cF]\in
\rK_{\iota,\sigma}([V/G],\Qlb)$.

Note that $W$ admits a Zariski open covering by $G$-stable affine schemes.
Thus, by the Zariski local nature of $\rK_{\iota,\sigma}$ (Remark
\ref{r.ZarK}), we may assume $W$ affine. In this case, the assertion is a
special case of Proposition \ref{p.NCD}. This finishes the proof of the case
of $f_!$ for $f$ an open immersion.

(2) Next we establish the general case of $f_!$. Let $(X_\alpha)_{\alpha\in
I}$ be a Zariski open covering of $X$ with $X_\alpha$ separated. For
$J\subseteq I$, let $j_J\colon \bigcap_{\beta\in J}X_\beta \to X$ be the
open immersion. Let $A\in \rK_{\iota,\sigma}(X,\Qlb)$. Then
$A=\sum_{\emptyset \neq J\subseteq I} (-1)^{1+\# J}j_{J!}j_J^*A$ by Lemma
\ref{l.coverK}. Thus we may assume $X$ separated. Applying Nagata
compactification \cite{CLO} to the morphism $\bar X\to \bar Y$ of coarse
spaces, we obtain a diagram with Cartesian squares
\[\xymatrix{X\ar[r]^{f_1} & \bar X\times_{\bar Y} Y\ar[r]^{f_2}\ar[d] & \bar Z\times_{\bar Y} Y \ar[r]^{f_3}\ar[d]
& Y\ar[d]\\
& \bar X\ar[r]^{g_2} & \bar Z\ar[r]^{g_3} & \bar Y, }
\]
where $f_1$ is proper and quasi-finite, $g_2$ is an open immersion, and
$g_3$ is proper. Thus $f_!=f_{3*}f_{2!}f_{1*}$ preserves
$\rK_{\iota,\sigma}$.

The case of $f_*=D_Y f_! D_X$ follows immediately.

(3) Next we establish the case of $f^*$. The argument is similar to the
deduction of the congruence $f^*\equiv f^!$ modulo $\rI(X,\Qlb)$
\cite[Corollary 9.5]{Zheng} from Laumon's theorem mentioned in Remark
\ref{r.Laumon}. Let $B\in \rK_{\iota,\sigma}(Y,\Qlb)$. If $f$ is a closed
immersion, then $B=j_!j^* B+f_*f^* B$, where $j$ is the complementary open
immersion. It follows that $f_*f^*B\in \rK_{\iota,\sigma}(Y,\Qlb)$, so that
$f^*B\in \rK_{\iota,\sigma}$. In the general case, let
$(Y_\alpha)_{\alpha\in I}$ be a  stratification of $Y$ such that each
$Y_\alpha$ is the quotient stack of an affine scheme by a finite group
action. For each $\alpha$, form the Cartesian square
\[\xymatrix{X_\alpha\ar[d]_{f_\alpha} \ar[r]^{j'_\alpha} & X\ar[d]^f\\
Y_\alpha\ar[r]^{j_\alpha} & Y.}
\]
Then $f^*B=\sum_{\alpha\in I}f^*j_{\alpha!}j_\alpha^* B=\sum_{\alpha\in I}
j'_{\alpha!} f_\alpha^*j_\alpha^* B$. Thus we may assume $Y=[Y'/H]$, where
$Y'$ is an affine scheme endowed with an action of a finite group $H$.
Similarly, we may assume $X=[X'/G]$, where $X'$ is an affine scheme endowed
with an action of a finite group $G$. Up to changing $X'$ and $G$, we may
further assume that $f=[f'/\gamma]$, for $(f,\gamma)\colon (X',G)\to
(Y',H)$, by \cite[Proposition 5.1]{Zind}. In this case $f'$ can be
decomposed into $G$-equivariant morphisms $X'\xrightarrow{i}
Z'\xrightarrow{p} Y'$ where $i$ is a closed immersion and $p$ is an affine
space. Thus $f^*\simeq [i/\id]^*[p/\gamma]^*$ preserves
$\rK_{\iota,\sigma}$.

The assertions for the other operations follow immediately: $f^!=D_X f^*
D_Y$, $-\otimes-=\Delta_X^*(-\boxtimes-)$, $\cHom(-,-)=D(-\otimes D-)$.
\end{proof}

\section{Variant: Horizontal complexes}\label{s.6}
In this section, let $k$ be a field finitely generated over its prime field.
This includes notably the case of a number field. Many results in previous
sections over finite fields can be generalized to Annette Huber's horizontal
complexes \cite{Huber}, as extended by Sophie Morel \cite{Morel}, over $k$.
In Subsection \ref{ss.h}, after briefly reviewing horizontal complexes, we
discuss symmetry and decomposition of pure horizontal complexes, and prove
analogues of results of Subsection \ref{ss.pc}. In Subsection \ref{ss.hK},
we discuss symmetry in Grothendieck groups of horizontal complexes, and give
analogues of results of Section \ref{s.G}. This section stems from a
suggestion of Takeshi Saito.

\subsection{Symmetry and decomposition of pure horizontal complexes}\label{ss.h}
Let $X$ be a Deligne-Mumford stack of finite presentation over $k$. Huber
\cite{Huber} (see also Morel \cite{Morel}) defines a triangulated category
$\rD^b_h(X,\Qlb)$ of horizontal complexes. For a finite extension $\Lambda$
of $\Z_\ell$, $\rD^b_h(X,\Lambda)$ is the $2$-colimit of the categories
$\rD^b_c(X_R,\Lambda)$, indexed by triples $(R,X_R,u)$, where $R\subseteq k$
is a subring of finite type over $\Z[1/\ell]$ such that $k=\Frac(R)$, $X_R$
is a Deligne-Mumford stack of finite presentation over $\Spec(R)$, and
$u\colon X\to X_R \otimes_R k$ is an isomorphism. We may restrict to $R$
regular and $X_R$ flat over $\Spec(R)$. We have Grothendieck's six
operations on $\rD^b_h(X,\Lambda)$.

The triangulated category $\rD^b_h(X,\Qlb)$ is equipped with a canonical
$t$-structure and a perverse $t$-structure. We let $\Sh_h(X,\Qlb)$ and
$\Perv_h(X,\Qlb)$ denote the respective hearts. The pullback functors via
$X\to X_R$ induce a conservative functor $\eta^*\colon \rD^b_h(X,\Qlb)\to
\rD^b_c(X,\Qlb)$ $t$-exact for the canonical $t$-structures and the perverse
$t$-structures, and compatible with the six operations. Moreover, $\eta^*$
induces \emph{fully faithful} exact functors $\Sh_h(X,\Qlb)\to \Sh(X,\Qlb)$
and $\Perv_h(X,\Qlb)\to \Perv(X,\Qlb)$ \cite[Propositions 2.3, 2.5]{Morel}.
Every object of $\Perv_h(X,\Qlb)$ has finite length.

\begin{remark}
The functor $\eta^*\colon \Perv_h(X,\Qlb)\to \Perv(X,\Qlb)$ preserves
indecomposable objects. By the description of simple objects, the functor
also preserves simple objects. Thus, via the functor, $\Perv_h(X,\Qlb)$ can
be identified with a full subcategory of $\Perv(X,\Qlb)$ stable under
subquotients. The subcategory is \emph{not} stable under extensions in
general.
\end{remark}

By restricting to closed points of $\Spec(R[1/\ell])$, we get a theory of
weights for horizontal complexes. Weight filtration does not always exist,
but this will not be a problem for us. The analogue of Remark \ref{r.purew}
for the preservation of pure complexes holds. Moreover, the analogues of
\cite[Th\'eor\`emes 5.3.8, 5.4.5]{BBD} hold for the decomposition of the
pullbacks of pure horizontal complexes to $\bar k$. In other words, the
functor
\[\bar\eta^*\colon \rD^b_h(X,\Qlb)\xrightarrow{\eta^*} \rD^b_c(X,\Qlb)\to \rD^b_c(X_{\bar k},\Qlb)\]
obtained by composing $\eta^*$ with the pullback functor carries pure
complexes to admissible semisimple complexes (Definition \ref{d.ass}).
Indeed, both theorems follow from \cite[Proposition 5.1.15 (iii)]{BBD},
which has the following analogue, despite the fact that the analogue of
\cite[Proposition 5.1.15 (ii)]{BBD} does not hold in general.

\begin{prop}
Let $K,L\in \rD^b_h(X,\Qlb)$, with $K$ mixed of weights $\le w$ and $L$
mixed of weights $\ge w$. Then $\Ext^i(\bar\eta^*K,\bar\eta^*L)^{\Gal(\bar
k/k)}=0$ for $i>0$. In particular, the map $\Ext^i(\eta^*K,\eta^*L)\to
\Ext^i(\bar\eta^* K, \bar\eta^* L)$ is zero.
\end{prop}

\begin{proof}
The second assertion follows from the first one, as the map factors through
$E^i\coloneqq \Ext^i(\bar\eta^*K,\bar\eta^*L)^{\Gal(\bar k/k)}$. For the
first assertion, consider the horizontal complex $\cE=\rR a_{X*}
\rR\cHom(K,L)$ on $\Spec(k)$, which has weight $\ge 0$. Thus $E^i\simeq
\Gamma(\Spec(k),\cH^i\cE)=0$ for $i>0$.
\end{proof}

Since pure horizontal perverse sheaves are geometrically semisimple, Lemma
\ref{l.indec} on the support decomposition applies (cf.\ \cite[Corollaire
5.3.11]{BBD}).

General preservation properties of $\sigma$-self-dual complexes listed in
Remark \ref{p.sd} still hold for $\rD^b_h$. The two-out-of-three property
(Proposition \ref{p.23}) holds for $\sigma$-self-dual horizontal perverse
sheaves. The trichotomy for indecomposable horizontal perverse sheaves
(Proposition \ref{p.schur}) also holds.

We say that a horizontal complex of $\Qlb$-sheaves $A$ is \emph{split} if it
is a direct sum of shifts of horizontal perverse sheaves, or, in other
words, $A\simeq \bigoplus_i(\pH^i A)[-i]$. Definition \ref{d.Diota} can be
repeated as follows.

\begin{definition}[$\rD^w_{h,\sigma}$]
Let $w\in \Z$. We denote by $\rD^w_{h,\sigma}(X,\Qlb)\subseteq\Ob
(\rD^b_h(X,\Qlb))$ (resp.\
$\rD^w_{h,\sd}(X,\Qlb)\subseteq\Ob(\rD^b_c(X,\Qlb))$) the subset consisting
of split pure horizontal complexes $A$ of weight $w$ such that $\pH^i A$ is
$(-1)^{w+i}\sigma$-self-dual (resp.\ self-dual) with respect to $K_X(-w-i)$
for all $i$. We denote by $\rD^w_{h,\dl}(X,\Qlb)\subseteq\Ob\big(
\rD^b_h(X,\Qlb)\times\rD^b_h(X,\Qlb)\big)$ the subset consisting of pairs
$(A,B)$ of split pure horizontal complexes of weight $w$ such that $\pH^i A$
is isomorphic to $(D_X\pH^iB)(-w-i)$ for all $i$.
\end{definition}

The analogue of Remark \ref{r.Diota} holds for the preservation of
$\rD^w_{h,\sigma}$ and $\rD^w_{h,\dl}$. The two-out-of-three property,
analogue of Remark \ref{r.23f}, also holds for $\rD^w_{h,\sigma}$ and
$\rD^w_{h,\dl}$. We have the following analogue of Proposition
\ref{p.boxtimesf}, which holds with the same proof as before, and a similar
result for $\rD^w_{h,\dl}$.

\begin{prop}
Let $m\ge 0$. Let $A$ be a mixed horizontal complex on $X$ such that for all
$n\in \Z$, $\pH^n A$ admits a weight filtration $W$, and that $\gr^W_w\pH^n
A$ is $(-1)^w\sigma$-self-dual with respect to $K_X(-w)$, for all $w\in \Z$.
Then for all $n\in \Z$, $\pH^n(A^{\boxtimes m})$ admits a weight filtration
$W$, and $\gr^W_w\pH^n(A^{\boxtimes m})$ is $(-1)^w\sigma^m$-self-dual with
respect to $K_{[X^m/\Perm_m]}(-mw)$ for all $w\in \Z$. Moreover, the functor
$(-)^{\boxtimes m}$ carries $\rD^w_{h,\sigma}(X,\Qlb)$ to
$\rD^{mw}_{h,\sigma^m}([X^m/\Perm_m],\Qlb)$.
\end{prop}

We have the following analogues of Theorems \ref{t.mainf} and Corollary
\ref{c.decomp}, which hold with the same proofs as before.

\begin{theorem}\label{t.mainh}
Let $f\colon X\to Y$ be a proper morphism of Deligne-Mumford stacks of
finite presentation over $k$. Assume that $Y$ has finite inertia. Then $\rR
f_*$ preserves $\rD^w_{h,\sigma}$ and $\rD^w_{h,\dl}$.
\end{theorem}

\begin{cor}
Assume that $Y$ has finite inertia. Then $Rf_*$ preserves split pure
complexes of weight $w$.
\end{cor}

The analogue of Corollary \ref{c.mainf} also holds for $\rD^b_{h,\sd}$.

Theorem \ref{t.i2} is a special case of Theorem \ref{t.mainh}. Applying it
to $a_{X}$, we obtain Theorem \ref{t.i1}. Indeed, as we remarked in the
introduction, in Theorem \ref{t.i1} we may assume that $k$ is finitely
generated over its prime field. The horizontal perverse sheaf $\IC_X$, pure
of weight $d$, is $(-1)^d$-self-dual with respect to $K_X(-d)$ by Example
\ref{e.IC}, so $\IC_X[-d]\in\rD_{h,1}^0$, hence $\rR
a_{X*}\IC_X[-d]\in\rD_{h,1}^0$ by Theorem \ref{t.i2}, which proves Theorem
\ref{t.i1}.

\subsection{Symmetry in Grothendieck groups of horizontal complexes}\label{ss.hK}
Let $X$ be a Deligne-Mumford stack of finite presentation over $k$. We let
$\rK_h(X,\Qlb)$ denote the Grothendieck group of $\rD^b_h(X,\Qlb)$, which is
a free Abelian group generated by the isomorphism classes of simple
horizontal perverse $\Qlb$-sheaves. The functor $\eta^*$ induces an
injection $\rK_h(X,\Qlb)\to \rK(X,\Qlb)$, which identifies $\rK_h(X,\Qlb)$
with a $\lambda$-subring of $\rK(X,\Qlb)$. The operations on Grothendieck
groups in Construction \ref{c.lambda} and Remark \ref{r.mid} induce
operations on $\rK_h$.

For $w\in \Z$, we let $\rK^w_h(X,\Qlb)\subseteq \rK_h(X,\Qlb)$ denote the
subgroup generated by pure horizontal perverse sheaves of weight $w$ on $X$.
Let $\rK^\Z_h(X,\Qlb)=\bigoplus_{w\in \Z}\rK^w_h(X,\Qlb)\subseteq
\rK_h(X,\Qlb)$. The analogue of Lemma \ref{l.interexact} holds, which
further justifies the definition of the map $f_{!*}$ in Remark \ref{r.mid}.

We repeat Definition \ref{d.K} as follows.

\begin{definition}[$\rK_{h,\sigma}$]
We define $\rK_{h,\sigma}^w(X,\Qlb)\subseteq \rK_h^w(X,\Qlb)$ (resp.\
$\rK_{h,\sd}^w(X,\Qlb)\subseteq \rK_h^w(X,\Qlb)$), for $w\in\Z$, to be the
subgroup generated by $[B]$, for $B$ perverse, $\iota$-pure of weight $w$,
and $(-1)^w\sigma$-self-dual (resp.\ self-dual) with respect to $K_X(-w)$.
We put
\[
\rK_{h,\sigma}(X,\Qlb)=\bigoplus_{w\in \Z} \rK_{h,\sigma}^w(X,\Qlb)\quad
\text{(resp.}\ \rK_{\iota,\sd}(X,\Qlb)=\bigoplus_{w\in \Z}
\rK_{h,\sd}^w(X,\Qlb)).
\]
We define the \emph{twisted dualizing map}
\[
\Dbar_{h,X}\colon \rK^\Z_h(X,\Qlb)\to \rK^\Z_h(X,\Qlb)
\]
to be the direct sum of the group automorphisms $\Dbar_{h,X}^w\colon
\rK^w_h(X,\Qlb)\to \rK^w_h(X,\Qlb)$ sending $[A]$ to $[(D_X A)(-w)]$. We let
$\rK_{h,\dl}^w(X,\Qlb)\subseteq\rK_{h}^w(X,\Qlb)^2$ denote the graph of
$\Dbar_{h,X}^w$. We put
\[\rK_{h,\dl}(X,\Qlb)=\bigoplus_{n\in Z}\rK^w_{h,\dl}(X,\Qlb).\]
\end{definition}

The subgroup $\rK_\orth(X,\Qlb)$ in the introduction is $\rK_{h,1}(X,\Qlb)$.
If $k$ is a finite field, $\rK_{h,\sigma}(X,\Qlb)$ is the intersection
$\bigcap_\iota\rK_{\iota,\sigma}(X,\Qlb)$ of the subgroups
$\rK_{\iota,\sigma}$ of Definition \ref{d.K}, where $\iota$ runs over
embeddings $\Qlb \hookrightarrow \C$.

The analogue of Remark \ref{r.K0} holds for the definition of
$\rK_{h,\sigma}$, $\rK_{h,\sd}$, and $\rK_{h,\dl}$. The analogues of Remarks
\ref{r.K} and \ref{r.ZarK} hold for the preservation and Zariski local
nature of $\rK_{h,\sigma}$ and $\rK_{h,\dl}$ with the same proofs.

The following analogue of Theorem \ref{t.sixop} holds with the same proof.
In particular, the analogue of Proposition \ref{p.NCD} holds.

\begin{theorem}\label{t.Kh}
Let $X$ and $Y$ be Deligne-Mumford stacks of finite inertia and  finite
presentation over $k$ and let $f\colon X\to Y$ be a morphism. Then
Grothendieck's six operations induce maps
\begin{gather*}
-\otimes-,\ \cHom(-,-)\colon \rK_{h,\sigma}(X,\Qlb)\times \rK_{h,\sigma'}(X,\Qlb)\to \rK_{h,\sigma\sigma'}(X,\Qlb),\\
f^*,f^!\colon
\rK_{h,\sigma}(Y,\Qlb)\to \rK_{h,\sigma}(X,\Qlb), \qquad f_*,f_!\colon \rK_{h,\sigma}(X,\Qlb)\to \rK_{h,\sigma}(Y,\Qlb).
\end{gather*}
Moreover, Grothendieck's six operations on $\rK^\Z_h$ commute with the
twisted dualizing map $\Dbar_h$.
\end{theorem}

The analogues of Corollaries \ref{c.lsubring} and \ref{c.lringhom} hold. The
relationship with Laumon's theorem (Remark \ref{r.Laumon}) also holds.

Theorem \ref{t.i3} is a special case of Theorem \ref{t.Kh}.

\appendix

\section{Appendix. Symmetry and duality in categories}\label{s.app}

In the appendix, we collect some general symmetry properties in categories
with additional structures. Tensor product equips the derived category of
$\ell$-adic sheaves with a symmetric structure. We discuss symmetry of
pairings in symmetric categories in Subsection \ref{s.smc}. The category of
perverse sheaves is not stable under tensor product, but is equipped with a
duality functor. We study symmetry in categories with duality in Subsections
\ref{ss.dual} and \ref{ss.fun}. We discuss the relation of the two points of
view in Subsection \ref{ss.rel}. We then study effects of translation on
symmetry in Subsection \ref{s.str}, which is applied in the main text to the
Lefschetz pairing. In Subsection \ref{s.nil}, we study symmetry of primitive
parts under a nilpotent operator, which is applied in the main text to the
monodromy operator. The results of the appendix are formal but are used in
the main text. The presentation here is influenced by \cite{QSS},
\cite[Section 12]{Riou}, and \cite{Schlichting}. Recall $\sigma,\sigma'\in
\{\pm 1\}$.

\subsection{Symmetric categories}\label{s.smc}

In this subsection, we discuss symmetry of pairings in symmetric categories.

\begin{definition}[Symmetric category]\label{d.scat}
A \emph{symmetric category} is a category $\cC$ endowed with a bifunctor
$-\otimes -\colon \cC\times \cC\to \cC$ and a natural isomorphism (called
the \emph{symmetry constraint}) $c_{AB}\colon A\otimes B\to B\otimes A$, for
objects $A$ and $B$ of $\cC$, satisfying $c_{AB}^{-1}=c_{BA}$. We say that
the symmetric category $\cC$ is \emph{closed} if for every object $A$ of
$\cD$, the functor $-\otimes A\colon \cC\to \cC$ admits a right adjoint,
which we denote by $\cHom(A,-)$.
\end{definition}

In our applications, we mostly encounter symmetric \emph{monoidal}
categories (see, for example, \cite[Section VII.7]{MacLane} for the
definition), but the associativity and unital constraints are mostly
irrelevant to the results of this article.

To deal with signs, we need the following additive variant of Definition
\ref{d.scat}.

\begin{definition}
A \emph{symmetric additive category} is a symmetric category $(\cD,\otimes)$
such that $\cD$ is an additive category and $-\otimes- \colon \cD\times
\cD\to \cD$ is an additive bifunctor (namely, a bifunctor additive in each
variable). A \emph{closed symmetric additive category} is a closed symmetric
category $(\cD,\otimes)$ such that $\cD$ is an additive category.
\end{definition}

A closed symmetric additive category is necessarily a symmetric additive
category and the internal Hom functor $\cHom(-,-)\colon \cD^{\op}\times
\cD\to \cD$ is an additive bifunctor.

\begin{definition}\label{d.sympair}
Let $(\cC,\otimes, c)$ be a symmetric category. Assume that $\cC$ is an
additive category if $\sigma=-1$. Let $A$, $B$, $K$ be objects of $\cC$.
\begin{enumerate}
\item We define the \emph{transpose} of a pairing $g\colon B\otimes A \to
    K$ to be the composite
    \[g^T\colon A\otimes B \xrightarrow{c} B\otimes A\xrightarrow{g} K.
    \]
We call $\sigma g^T$ the \emph{$\sigma$-transpose} of $g$.

\item We say that a pairing $f\colon A\otimes A \to K$ is
    \emph{$\sigma$-symmetric} if $f=\sigma f^T$.
\end{enumerate}
\end{definition}

We have $(g^T)^T=g$. We will often say ``symmetric'' instead of
``$1$-symmetric''.

Note that for a pair of pairings $f\colon A\otimes B\to K$ and $g\colon
B\otimes A\to K$ in a symmetric additive category, $(2f,2g)$ is a sum of a
pair of $1$-transposes and a pair of $-1$-transposes:
\[(2f,2g)=(f+g^T,g+f^T)+(f-g^T,g-f^T).\]

\begin{remark}
Let $(\cC,\otimes,c)$ be a symmetric category such that $\cC$ is an additive
category. Then $(\cC,\otimes,-c)$ is another symmetric category. The
$-1$-transpose in $(\cC,\otimes,c)$ of a pairing $g\colon A\otimes B\to K$
is the transpose in $(\cC,\otimes,-c)$ of $g$.
\end{remark}

Next we consider effects of functors on symmetry.

\begin{definition}\label{d.sf}
Let $\cC$ and $\cD$ be symmetric categories. A \emph{right-lax symmetric
functor} (resp.\ \emph{symmetric functor}) from $\cC$ to $\cD$ is a functor
$G\colon \cC\to \cD$ endowed with a natural transformation (resp.\ natural
isomorphism) of functors $\cC\times \cC\to \cD$ given by morphisms
$G(A)\otimes G(B)\to G(A\otimes B)$ in $\cD$ for objects $A,B$ of $\cC$,
such that the following diagram commutes
\[\xymatrix{G(A)\otimes G(B)\ar[r]\ar[d]_{c_{GA,GB}} & G(A\otimes B)\ar[d]^{G(c_{A,B})}\\
G(B)\otimes G(A)\ar[r]& G(B\otimes A).}
\]
\end{definition}

Between symmetric monoidal categories, one has the notions of symmetric
monoidal functors and lax symmetric monoidal functors, which are compatible
with the associativity constraints and unital constraints. In our
applications we will need to consider symmetric functors between symmetric
monoidal categories that are \emph{not} symmetric monoidal functors. For
example, if $f$ is an open immersion, then $f_!$ is a symmetric functor
compatible with the associativity constraint, but not compatible with the
unital constraints except in trivial cases. Again we emphasize that the
compatibility with the associativity and unital constraints is irrelevant to
the results in this article.

\begin{example}
Let $\cC$ and $\cD$ be symmetric categories. Let $F\colon \cC\to \cD$ be a
functor admitting a right adjoint $G\colon \cD\to \cC$. Then every symmetric
structure on $F$ induces a right-lax symmetric structure on $G$, given by
the morphism $G(A)\otimes G(B)\to G (A\otimes B)$ adjoint to
\[F(G(A)\otimes G(B))\simto F(G(A))\otimes F(G(B))\to A\otimes B.\]
This construction extends to \emph{left-lax} symmetric structures on $F$ and
provides a bijection between left-lax symmetric structures on $F$ and
right-lax symmetric structures on $G$. Since we do not need this extension,
we omit the details.
\end{example}

\begin{example}
Let $\cC$ be a symmetric \emph{monoidal} category. Then $\cC\times \cC$ is a
symmetric monoidal category and the functor $-\otimes -\colon \cC\times
\cC\to \cC$ is a symmetric monoidal functor and, in particular, a symmetric
functor. The symmetric structure of the functor is given by the isomorphisms
$(A\otimes A')\otimes (B\otimes B')\simto (A\otimes B)\otimes (A'\otimes
B')$ for objects $A,A',B,B'$ of $\cC$.
\end{example}

\begin{construction}\label{c.rl}
Let $\cC$ and $\cD$ be symmetric categories and let $G\colon \cC\to \cD$ be
a right-lax symmetric functor. Let $A,B,K$ be objects of $\cC$. A pairing
$A\otimes B\to K$ induces a pairing $G(A)\otimes G(B)\to G(A\otimes B)\to
G(K)$.
\end{construction}

The following lemma follows immediately from the definitions.

\begin{lemma}
Let $\cC$ and $\cD$ be symmetric categories and let $G\colon \cC\to \cD$ be
a right-lax symmetric functor. Let $A,B,K$ be objects of $\cC$. Let
$A\otimes B\to K$ and $B\otimes A\to K$ be transposes of each other. Then
the induced pairings $GA\otimes GB\to GK$ and $GB\otimes GA\to GK$ are
transposes of each other.
\end{lemma}

\subsection{Categories with duality}\label{ss.dual}
In this subsection, we study symmetry in categories with duality.

\begin{definition}[Duality]\label{d.d}
Let $\cC$ be a category. A \emph{duality} on $\cC$ is a functor $D\colon
\cD^{\op}\to \cD$ endowed with a natural transformation $\ev\colon
\id_\cC\to DD$ such that the composite $D\xrightarrow{\ev D} DDD
\xrightarrow{D \ev} D$ is equal to $\id_D$. The duality $(D,\ev)$ is said to
be \emph{strong} if $\ev$ is a natural isomorphism.
\end{definition}

We are mostly interested in strong dualities in the main text. However, for
the proofs of many results on strong dualities, it is necessary to consider
general dualities (for example, the duality $D_{\rR f_* K_X}$ in the proof
of Remark \ref{p.sd} (3) is not strong in general). Our terminology here is
consistent with \cite{Schlichting}*{Definition 3.1}. Some authors refer to a
strong duality simply as ``duality'' \cite{QSS}.

The underlying functor of a strong duality is an equivalence of categories.
If $\cC$ is an additive category, we say that a duality on $\cC$ is
\emph{additive} if the underlying functor is additive. By an \emph{additive
category with duality}, we mean an additive category equipped with an
additive duality.

A basic example of duality is provided by the internal Hom functor in a
closed symmetric category. We will discuss this in detail in Subsection
\ref{ss.rel}. By analogy with this case, we sometimes refer to morphisms
$B\to DA$ in a category with duality as \emph{forms}. We have the following
notion of symmetry for forms.

\begin{definition}[Symmetry of forms]\label{d.symD}
Let $(\cC,D,\ev)$ be a category with duality for $\sigma=1$ (resp.\ additive
category with duality for $\sigma=-1$) and let $A,B$ be objects of $\cD$.
\begin{enumerate}
\item We define the \emph{transpose} of a morphism $g\colon B\to D A$ to
    be the composite
    \[\xymatrix{A\xrightarrow{\ev} D D A
    \xrightarrow{D g}  D B.}
    \]
    We call $\sigma g^T$ the \emph{$\sigma$-transpose} of $g$.

\item We say that a morphism $f\colon A\to D A$ is
    \emph{$\sigma$-symmetric} if $f=\sigma f^T$.
\end{enumerate}
\end{definition}

Again we will often say ``symmetric'' instead of ``$1$-symmetric''. The
terminology above is justified by the following lemma.

\begin{lemma}
We have $(g^T)^T=g$. Moreover, the map $\Hom_\cC(B,DA)\to \Hom_\cC(A,DB)$
carrying $g$ to $g^T$ is a bijection.
\end{lemma}

\begin{proof}
The first assertion follows from the commutativity of the following diagram
\[\xymatrix{& DA\\
&DDDA\ar[u]_{D\ev} & DDB\ar[ul]_{D(g^T)}\ar[l]^{DDg}\\
DA\ar[ur]_{\ev D}\ar[ruu]^{\id} & B.\ar[l]^{g}\ar[ur]_{\ev}}
\]
For the second assertion, note that the map carrying $h\colon A\to DB$ to
$h^T$ is the inverse of the map $g\mapsto g^T$, by the first assertion.
\end{proof}

\begin{remark}
Let $(\cC,D,\ev)$ be an additive category with duality. Then $(\cC,D,-\ev)$
is another additive category with duality. The $-1$-transpose in
$(\cC,D,\ev)$ of a morphism $g\colon B\to DA$ is the transpose in
$(\cC,D,-\ev)$ of $g$. This allows us in the sequel to omit the
$-1$-symmetric case in many results without loss of generality.
\end{remark}

We will be especially interested in objects $A$ that admit isomorphisms
$A\simto DA$.

\begin{definition}\label{d.sd}
Let $(\cC,D,\ev)$ be a category with duality and let $A$ be an object of
$\cC$.
\begin{enumerate}
\item We say that $A$ is \emph{self-dual} if there exists an isomorphism
    $A\simto DA$.

\item Assume that $(\cC,D,\ev)$ is an additive category with duality if
    $\sigma=-1$. We say that $A$ is \emph{$\sigma$-self-dual} if there
    exists a $\sigma$-symmetric isomorphism $A\simto DA$.
\end{enumerate}
\end{definition}

We warn the reader that being $1$-self-dual is more restrictive than being
self-dual. If $A$ is $1$-self-dual or $-1$-self-dual, then $\ev\colon A\to
DDA$ is an isomorphism.

\begin{remark}\label{r.selfd}
Let $(\cC,D,\ev)$ be an additive category with duality.
\begin{enumerate}
\item The classes of self-dual objects and $\sigma$-self-dual objects of
    $\cD$ are stable under finite products.

\item If $A$ is an object of $\cD$ such that $\ev\colon A\to DDA$ is an
    isomorphism, then $A\oplus DA$ is $1$-self-dual and $-1$-self-dual. In
    fact, the isomorphism $A\oplus DA \xrightarrow{\sigma\ev\oplus \id}
    DDA\oplus DA\simeq D(A\oplus DA)$ is $\sigma$-symmetric.

\item There are self-dual objects that are neither $1$-self-dual nor
    $-1$-self-dual (Corollary \ref{c.even}).
\end{enumerate}
\end{remark}

We close this subsection with a couple of lemmas on $\sigma$-self-dual
objects. They are used in Subsection \ref{ss.gp2} but not in the rest of
this appendix.

A $\sigma$-symmetric isomorphism $f\colon A\simto DA$ induces an involution
on $\End(A)$ carrying $g\in \End(A)$ to $f^{-1} (Dg)f$. If $\cC$ is a
$k$-linear category and $D$ is a $k$-linear functor, then the involution is
$k$-linear.

\begin{lemma}\label{l.KS}
Let $(\cC,D,\ev)$ be an additive category with duality. Let $A$ be an object
of $\cD$ such that $R=\End(A)$ is a local ring and that $2$ is invertible in
$R$.
\begin{enumerate}
\item If $A$ is self-dual with respect to $D$, then $A$ is $1$-self-dual
    or $-1$-self-dual with respect to $D$.

\item If $A$ is both $1$-self-dual and $-1$-self-dual with respect to $D$,
    then every symmetric (resp.\ $-1$-symmetric) isomorphism $f\colon A\to
    DA$ induces a nontrivial involution on the residue division ring of
    $R$.
\end{enumerate}
\end{lemma}

This is essentially \cite{QSS}*{Proposition 2.5}. We recall the proof in our
notation. It will be apparent from the proof that the additional assumption
in \cite{QSS} that $D$ is a strong duality is not used.

\begin{proof}
(1) Since $A\simeq DA$, we have $\End(A)\simeq \Hom(A,DA)$. The image
$M\subseteq \Hom(A,DA)$ of the maximal ideal of $R=\End(A)$ is the
complement of the set of isomorphisms. For any $f\in \Hom(A,DA)$, we have
$2f=(f+f^T)+(f-f^T)$, where $f+f^T$ is symmetric and $f-f^T$ is
$-1$-symmetric. If $f$ is an isomorphism, then $2f$ is an isomorphism, so
that either $f+f^T$ or $f-f^T$ is an isomorphism.

(2) Let $g\colon A\to DA$ be a $-1$-symmetric (resp.\ symmetric)
isomorphism. Then $h=f^{-1}g$ is a unit of $R$ whose image under the
involution induced by $f$ is $-h$. Thus the involution is nontrivial on the
residue field of $R$.
\end{proof}

\begin{remark}\leavevmode\label{r.KS}
\begin{enumerate}
\item If $\cC$ is an Abelian category and $A$ is an indecomposable object
    of finite length, then $\End(A)$ is a local ring \cite[Lemma
    7]{Atiyah}.

\item Let $k$ be a separably closed field of characteristic $\neq 2$.
    Assume that $\cC$ is a $k$-linear category, $D$ is a $k$-linear
    functor, and $R=\End(A)$ is a finite $k$-algebra. Then any $k$-linear
    involution on $R$ is trivial on the residue field. It follows then
    from Lemma \ref{l.KS} that exactly one of the following holds: $A$ is
    $1$-self-dual; $A$ is $-1$-self-dual; $A$ is not self-dual.
\end{enumerate}
\end{remark}

\begin{lemma}\label{l.ss}
Let $(D,\ev)$ be a \emph{strong} duality on an Abelian category $\cC$. Let
$A$ be a $\sigma$-self-dual object of finite length. Then the
semisimplification $A^\rss$ of $A$ is $\sigma$-self-dual.
\end{lemma}

Note that by assumption $D$ is an equivalence of categories, hence an exact
functor.

\begin{proof}
We fix a $\sigma$-symmetric isomorphism $f\colon A\to DA$. For any
sub-object $N$ of $A$, we let $N^\perp$ denote the kernel of the morphism
$A\xrightarrow[\sim]{f} DA\to DN$. Then we have $A/N^\perp\simeq DN$, so
that $N^\rss\oplus (A/N^\perp)^\rss$ is $\sigma$-self-dual by Remark
\ref{r.selfd}. If $N$ is totally isotropic, namely $N\subseteq N^\perp$,
then $f$ induces a $\sigma$-symmetric isomorphism $N^\perp/N\simto
D(N^\perp/N)$ (cf.\ \cite[Lemma 5.2]{QSS}). Now let $N$ be a maximal totally
isotropic sub-object of $A$. By \cite[Theorem 6.12]{QSS}, $N^\perp/N$ is
semisimple. Therefore, $A^\rss\simeq N^\rss\oplus (A/N^\perp)^\rss\oplus
N^\perp/N$ is $\sigma$-self-dual.
\end{proof}

\subsection{Duality and functors}\label{ss.fun}
In this subsection, we study symmetry of functors between categories with
duality.

Given categories with duality $(\cC,D_\cC,\ev)$ and $(\cD,D_\cD,\ev)$, and
functors $F,G\colon \cC\to \cD$, we sometimes refer to natural
transformations $GD_\cC\to D_\cD F$ as \emph{form transformations}. If $\cC$
is the category with one object $*$ and one morphism $\id$, and if we
identify functors $\{*\}\to \cD$ with objects of $\cD$, then a form
transformation is simply a form in $\cD$. Form transformations are composed
as follows.

\begin{construction}\label{c.comp} Let
$(\cB,D_\cB,\ev)$, $(\cC,D_\cC,\ev)$, and $(\cD,D_\cD,\ev)$ be categories
with duality. Let $F,G\colon \cC\to \cD$ and $F',G'\colon \cB\to \cC$ be
functors. Let $\alpha\colon FD_\cC\to D_\cD G$ and $\alpha'\colon F'D_\cB\to
D_\cC G'$ be natural transformations. We define the \textit{composite} of
$\alpha$ and $\alpha'$ to be
\[\alpha\alpha'\colon FF'D_\cB\xrightarrow{F\alpha'} FD_\cC G'\xrightarrow{\alpha
G'}D_\cD GG'.\]
\end{construction}

As the name suggests, form transformations act on forms. This can be seen as
the case $\cB=\{*\}$ of the preceding construction, as follows.

\begin{construction}\label{c.F}
Let $(\cC,D_\cC,\ev)$ and $(\cD,D_\cD,\ev)$ be categories with duality. Let
$F,G\colon \cC\to \cD$ be functors and let $\alpha\colon FD_\cC \to D_\cD G$
be a natural transformation. Let $A, B$ be objects of $\cC$ and let $f\colon
A\to D_\cC B$ be a morphism in $\cC$. The action of $\alpha$ on $f$ is the
composite
\[\alpha f\colon FA \xrightarrow{Ff} FD_\cC B \xrightarrow{\alpha_B} D_\cD GB.\]
\end{construction}

We have the following notion of symmetry for form transformations.

\begin{definition}[Symmetry of form transformations]\label{d.symF}
Let $(\cC,D_\cC,\ev)$ and $(\cD,D_\cD,\ev)$ be categories with duality. Let
$F,G\colon \cC\to \cD$ be functors. Assume that $(\cD,D_\cD,\ev)$ is an
additive category with duality if $\sigma=-1$.
\begin{enumerate}
\item We define the \emph{transpose} of a natural transformation
    $\beta\colon GD_\cC\to D_\cD F$ to be the composite
\[\beta^T\colon FD_\cC\xrightarrow{\ev FD_\cC}  \\
    D_\cD D_\cD FD_\cC\xrightarrow{D_\cD\beta D_\cC}  D_\cD GD_\cC D_\cC\xrightarrow{D_\cD G\ev} D_\cD G.
\]
We call $\sigma\beta^T$ the $\sigma$-transpose of $\beta$.

\item We say that a natural transformation $\alpha\colon FD_\cC \to D_\cD
    F$ is \emph{$\sigma$-symmetric} if $\alpha=\sigma\alpha^T$.
\end{enumerate}
\end{definition}

Again we will often say ``symmetric'' instead of ``$1$-symmetric''. The
terminology above is justified by the following easy lemma.

\begin{lemma}
We have $(\beta^T)^T=\beta$. Moreover, the map $\Nat(GD_\cC,D_\cD
F)\to\Nat(FD_\cC,D_\cD G)$ carrying $\beta$ to $\beta^T$ is a bijection.
\end{lemma}

The transpose $\alpha=\beta^T$ is uniquely characterized by the
commutativity of the diagram
\begin{equation}\label{e.formtrans}
\xymatrix{F\ar[r]^-{F\ev}\ar[d]_{\ev F} & FD_\cC D_\cC\ar[d]^{\alpha D_\cC}\\
D_\cD D_\cD F\ar[r]^-{D_\cD\beta} &D_\cD GD_\cC.}
\end{equation}
If $\cC=\{*\}$, Definition \ref{d.symF} reduces to Definition \ref{d.symD}.

\begin{remark}
A more direct analogue of Definition \ref{d.symD} (1) for functors is as
follows. Let $\cC$ be a category and let $(\cD,D_\cD,\ev)$ be a category
with duality. Let $G\colon \cC\to \cD$ and $H\colon \cC\to \cD^{\op}$ be
functors. Then the map $\Nat(G,D_\cD H)\to \Nat(H,D_\cD G)$ carrying
$\gamma\colon G\to D_\cD H$ to $\gamma^*\colon H\xrightarrow{\ev H} D_\cD
D_\cD H \xrightarrow{D_\cD\gamma} D_\cD G$ is a bijection. Indeed, $\gamma$
is a collection $(\gamma_A\colon GA\to D_\cD HA)_A$ of forms in $\cD$ and
$\gamma^*$ is characterized by $(\gamma^*)_A=(\gamma_A)^T$ for all objects
$A$ of $\cC$, so that $(\gamma^*)^*=\gamma$. As one of the referee points
out, this operation does not lead to a notion of symmetry, since $G$ and $H$
do not have the same variance.
\end{remark}

\begin{remark}
In the situation of Definition \ref{d.symF}, we have a bijection
\begin{equation}\label{e.symF}
\Nat(G D_\cC,D_\cD F)\simto \Nat(F,D_\cD G D_\cC)
\end{equation}
carrying $\beta$ to $\beta^*$. Note that $\beta^*$ is the composite
$F\xrightarrow{F\ev}FD_\cC D_\cC\xrightarrow{\beta^T D_\cC} D_\cD G D_\cC$,
and $\beta^T$ is the composite $FD_\cC\xrightarrow{\beta^*D_\cC} D_\cD G
D_\cC D_\cC \xrightarrow{D_\cD G\ev} D_\cD G$.

If equip the functor category $\Fun(\cC,\cD)$ with the duality carrying $G$
to $D_\cD G D_\cC$ and evaluation transformation given by $G\xrightarrow{\ev
G\ev} D_\cD D_\cD G D_\cC D_\cC$, then natural transformations $F\to D_\cD G
D_\cC$ are forms in this category with duality and Definition \ref{d.symF}
of transposes of forms applies. Definition \ref{d.symF} is compatible with
Definition \ref{d.symD} via the bijection \eqref{e.symF} in the sense that
we have $(\beta^*)^T=(\beta^T)^*$.
\end{remark}

Composition of form transformations is compatible with transposition.

\begin{lemma}\label{l.comp}
Let $(\cB,D_\cB,\ev)$, $(\cC,D_\cC,\ev)$, $(\cD,D_\cD,\ev)$ be categories
with duality. Let $F,G\colon \cC\to \cD$ and $F',G'\colon \cB\to \cC$ be
functors. Let $\alpha\colon FD_\cC\to D_\cD G$ and $\alpha'\colon F'D_\cB\to
D_\cC G'$ be natural transformations and let $\alpha\alpha'\colon
FF'D_\cB\to D_\cD GG'$ be the composite. Then
$(\alpha\alpha')^T=\alpha^T\alpha'^T$.
\end{lemma}

\begin{proof}
In the diagram
\[\xymatrix{FF'\ar[r]^-\ev\ar[rd]^\ev\ar[d]_\ev & FF'D_\cB D_\cB\ar[rd]^{\alpha'} \\
D_\cD D_\cD FF'\ar[rd]^{\alpha^T} & FD_\cC D_\cC F'\ar[r]^{\alpha'^T}\ar[d]^{\alpha}& FD_\cC G'D_\cB\ar[d]^{\alpha}\\
&D_\cD GD_\cC F'\ar[r]^{\alpha'^T} & D_\cD GG'D_\cB,}
\]
all inner cells commute. It follows that the outer hexagon commutes.
\end{proof}

Taking $\cB=\{*\}$, we obtain the following compatibility of transposition
with action of form transformations.

\begin{lemma}\label{l.F}
Let $(\cC,D_\cC,\ev)$ and $(\cD,D_\cD,\ev)$ be categories with duality. Let
$F,G\colon \cC\to \cD$ be functors equipped with a natural transformation
$\alpha\colon FD_\cC \to D_\cD G$. Let $f\colon A\to D_\cC B$ be a morphism
in $\cC$. Then $(\alpha f)^T=\alpha^T f^T$.
\end{lemma}

The following consequence of Lemma \ref{l.F} is used many times in Section
\ref{s.gp}.

\begin{lemma}\label{l.selfdual}
Let $(\cC,D_\cC,\ev)$ and $(\cD,D_\cD,\ev)$ be categories with duality. Let
$F\colon \cC\to \cD$ be a functor endowed with a symmetric natural
\emph{isomorphism} $\alpha\colon FD_\cC\simto D_\cD F$.
\begin{enumerate}
\item $F$ carries $1$-self-dual objects of $\cC$ to $1$-self-dual objects
    of $\cD$.

\item If $F$ is fully faithful, then the converse holds: any object $A$ of
    $\cC$ such that $FA$ is $1$-self-dual is $1$-self-dual.
\end{enumerate}
\end{lemma}

\begin{proof}
For (1), let $f\colon A\simto D_\cC A$ be a symmetric isomorphism. By Lemma
\ref{l.F}, $\alpha f$ is symmetric. The assertion follows from the fact that
$\alpha f$ is an isomorphism. For (2), let $g\colon FA\simto D_\cD FA$ be a
symmetric isomorphism. Since $F$ is fully faithful, there exists a unique
morphism $f\colon A\to D_\cC A$ such that $\alpha f=g$. Note that $f$ is an
isomorphism. Since $\alpha f^T=g$, we have $f^T=f$.
\end{proof}

The following lemma is used in Subsection \ref{ss.gp1} to show the symmetry
of the middle extension functor.

\begin{lemma}\label{l.inter}
Let $(\cC,D_\cC,\ev)$ and $(\cD,D_\cD,\ev)$ be categories with duality.
Assume that $\cD$ is an Abelian category and $D_\cD$ carries epimorphisms in
$\cD$ to monomorphisms. Let $E,G\colon \cC\to \cD$ be functors endowed with
natural transformations $\alpha\colon E\to G$ and $\beta \colon GD_\cC \to
D_\cD E$ such that the composite $ED_\cC\xrightarrow{\alpha D_\cC}
GD_\cC\xrightarrow{\beta} D_\cD E$ is symmetric and such that the image
functor $F\colon \cC\to \cD$ of $\alpha$ fits into a commutative diagram
\[\xymatrix{FD_\cC \ar[r]^\gamma\ar[d] & D_\cD F\ar[d]\\
GD_\cC \ar[r]^\beta & D_\cD E.}
\]
Then the natural transformation $\gamma\colon FD_\cC\to D_\cD F$ is
symmetric.
\end{lemma}

\begin{proof}
In fact, in the diagram
\[\xymatrix{E\ar@{->>}[rd]\ar[r]^{\ev}\ar[ddd]_{\ev} & ED_\cC D_\cC\ar[rr]^{\alpha}\ar[rd] && GD_\cC D_\cC\ar[ddd]^\beta\\
& F\ar[d]_{\ev}\ar[r]_{\ev} & FD_\cC D_\cC \ar[ru]\ar[d]^\gamma\\
& D_\cD D_\cD F\ar[r]^\gamma & D_\cD F D_\cC\ar@{^{(}->}[rd]\\
D_\cD D_\cD E\ar[r]^\beta\ar[ru] & D_\cD G D_\cC \ar[rr]^\alpha\ar[ru] && D_\cD E D_\cC,}
\]
the outer square commutes by the symmetry of $\beta\alpha$ and all inner
cells except the inner square commute. It follows that the inner square
commutes.
\end{proof}

We conclude this subsection with another example of form transformation,
which will be used to handle the sign of the Lefschetz pairing (see Lemma
\ref{l.final}). We refer to \cite[Remark 10.1.10 (ii)]{KS} for the
convention on distinguished triangles in the opposite category of a
triangulated category.

\begin{lemma}\label{l.trunc}
Let $\cD$ be a triangulated category equipped with a $t$-structure $P$. Let
$(D,\ev)\colon \cD^{\op}\to \cD$ be a duality on the underlying category of
$\cD$. Assume that $D$ underlies a \emph{right $t$-exact} triangulated
functor. We consider $\tau=\Ptau^{\ge a}$ and $\tau'=\Ptau^{\le -a}$ as
functors $\cD\to \cD$. Then the form transformations $\tau D \to D \tau'$
and $\tau' D \to D \tau$ induced by the diagrams
\begin{gather}
\tau D \to \tau D \tau' \xleftarrow{\sim} D \tau',\label{e.tau0}\\
\tau' D \xleftarrow{\sim} \tau'D\tau \to D\tau\label{e.tau}
\end{gather}
are transposes of each other.
\end{lemma}

The second arrow in \eqref{e.tau0} is an isomorphism by the assumption that
$D$ carries $\Ptau^{\le -a}$ to $\Ptau^{\ge -a}$. To see that the first
arrow in \eqref{e.tau} is an isomorphism, consider, for any object $A$ of
$\cD$, the distinguished triangle
\[D\Ptau^{\ge a} A \xrightarrow{f} DA \to D \Ptau^{\le a-1} A\to.\]
By assumption, $D \Ptau^{\le a-1} A$ is in $\PcD^{\ge 1-a}$. Thus, by Lemma
\ref{l.nine}, $\Ptau^{\le -a} f$ is an isomorphism.

\begin{proof}
The commutativity of \eqref{e.formtrans} follows from the commutativity of
the diagram
\[\xymatrix{\tau\ar[r]\ar[d] & \tau DD \ar[r] & \tau D\tau' D\ar[d] & D\tau' D\ar[l]_\sim\ar[d]\ar@{=}[rdd]\\
DD\tau \ar[d] && DD\tau D\tau' D\ar[d] & DDD\tau' D\ar[dd]\ar[rd]\ar[l]_\sim\\
D\tau'D \tau D \ar[r] & D\tau' D\tau DD\ar[r] & D\tau' D\tau D \tau' D && D\tau' D\ar[dd]^\simeq \\
D\tau' D\ar[u]^\simeq \ar@{=}[rrd] \ar[r] & D\tau' DDD\ar[rd]\ar[u]_\simeq\ar[rr] && D\tau'DD\tau' D\ar[rd]\ar[lu]_\simeq \\
&& D\tau'D\ar[rr]^\sim && D\tau'\tau' D.}\]
\end{proof}

\begin{remark}\label{r.trunc}
For any truncation functor $\tau=\Ptau^{[a,b]}$ with dual truncation functor
$\tau'=\Ptau^{[-b,-a]}$, combining the two form transformations in the
lemma, we obtain a form transformation $\gamma_\tau\colon \tau D \to D
\tau'$ whose transpose is $\gamma_\tau'$. The form transformation
$\gamma_\tau$ is an isomorphism if $D$ is \emph{$t$-exact}.
\end{remark}

\subsection{Duality in closed symmetric categories}\label{ss.rel}
In this subsection, we study dualities given by internal Hom functors in
closed symmetric categories. Let $(\cC,\otimes, c)$ be a closed symmetric
category (Definition \ref{d.scat}).

\begin{construction}
Let $K$ be an object of $\cC$. We let $D_K$ denote the functor
$\cHom(-,K)\colon \cD^{op}\to \cD$. For an object $A$, the composite
\[A\otimes D_K A\xrightarrow{c} D_K A\otimes A \xrightarrow{\adj} K,\]
where $\adj$ denotes the adjunction morphism, corresponds by adjunction to a
morphism $A\to D_K D_K A$. This defines a natural transformation $\ev\colon
\id_\cD\to D_K D_K$, which makes $D_K$ a duality on $\cD$. The latter
follows by adjunction from the commutativity of the diagram
\[\xymatrix{D_K A\otimes D_K D_K A\ar[r]^c & D_K D_K A\otimes D_K A\ar[rdd]^{\adj} \\
D_K A\otimes A\ar[u]^{\id\otimes \ev}\ar[r]^c \ar[rrd]_{\adj} & A\otimes D_K A\ar[u]^{\ev\otimes \id}\\
&& K.}
\]
\end{construction}

We defined transposes of pairings in symmetric categories (Definition
\ref{d.symD}) and in categories with duality (Definition \ref{d.sympair}).
The two definitions are compatible via the above construction, by the
following lemma.

\begin{lemma}\label{lemma2.1.12}
Let $A,B,K$ be objects of $\cC$. We put $D=D_K$. Then the following diagram
commutes
\[\xymatrix{\Hom(B\otimes A, K)\ar[d]_\simeq\ar[rr]^{-\circ c} && \Hom(A\otimes B, K)\ar[d]^\simeq\\
\Hom(B,DA)\ar[r]^D & \Hom(DDA,DB)\ar[r]^-{-\circ \ev(A)}& \Hom(A,DB).}
\]
\end{lemma}

\begin{proof}
Let $f\in \Hom( B,DA)$. The two images of $f$ in $\Hom(A\otimes B,K)$ are
the two composite morphisms in the commutative diagram
\[\xymatrix{A\otimes B\ar[r]^c\ar[d]_{\id \otimes f} & B\otimes A\ar[d]^{f\otimes
\id}\\
A\otimes DA\ar[r]^c & DA\otimes A\ar[r]^-{\adj} & K.}
\]
\end{proof}

Following Definition \ref{d.sd}, we say $A$ is \emph{self-dual} with respect
to $K$ if $A\simeq D_K A$. We say $A$ is \emph{$\sigma$-self-dual} with
respect to $K$ if there exists a $\sigma$-symmetric isomorphism $A \simto
D_K A$, or, in other words, if there exists a $\sigma$-symmetric pairing
$A\otimes A\to K$ that is \emph{perfect} in the sense that it induces an
isomorphism $A\simto D_KA$.

\begin{definition}
A \emph{dualizing object} of $\cC$ is an object $K$ of $\cC$ such that the
evaluation transformation $\ev\colon \id_\cC\to D_K D_K$ is a natural
isomorphism, or, in other words, that $(D_K,\ev)$ is a strong duality.
\end{definition}

\begin{remark}
Let $\cB$ be a closed symmetric \emph{monoidal} category and let $K$ be an
object of $\cB$. The associativity constraint induces an isomorphism
$\cHom(A,D_K B)\simeq D_K(A\otimes B)$ for objects $A,B$ of $\cD$. In
particular, if $K$ is a dualizing object, then $\cHom(A,B)\simeq
\cHom(A,D_KD_K B)\simeq D_K(A\otimes D_K B)$.
\end{remark}

We close this subsection with the construction of two symmetric form
transformations.

\begin{construction}
For a morphism $f\colon K\to L$ of $\cC$, the natural transformation
$D_f\colon \id_\cC D_K\to D_L \id_\cC$ is symmetric.  This follows from the
commutativity of the diagram
\[\xymatrix{A\otimes D_K A\ar[r]^c\ar[d]_{\id\otimes D_f} & D_K A\otimes A\ar[r]^-\adj\ar[d]_{D_f\otimes \id} & K\ar[d]^f\\
A\otimes D_L A\ar[r]^c & D_L A\otimes A \ar[r]^-\adj & L.}
\]
The action of $D_f$ on forms (Construction \ref{c.F}) carries $A\otimes B\to
K$ to the composite $A\otimes B\to K\xrightarrow{f} L$.
\end{construction}

\begin{construction}\label{c.image}
Let $\cC$ and $\cD$ be closed symmetric categories and let $G\colon \cC\to
\cD$ be a right-lax symmetric functor (Definition \ref{d.sf}). For objects
$A$, $K$ of $\cC$, consider the morphism
\begin{multline*}
G\cHom(A,K)\xrightarrow{\adj} \cHom(GA,G\cHom(A,K)\otimes GA) \to \cHom(GA,G(\cHom(A,K)\otimes A))\\
\xrightarrow{\adj} \cHom(GA,GK).
\end{multline*}
This induces a symmetric natural transformation $GD_K \to D_{GK} G$ (cf.\
\cite[Th\'eor\`eme 12.2.5]{Riou}), whose action on forms carries $A\otimes
B\to K$ to the pairing $GA\otimes GB\to GK$ of Construction \ref{c.rl}.
\end{construction}

\subsection{Symmetry and translation}\label{s.str}
The derived category of $\ell$-adic sheaves is equipped with a shift functor
$A\mapsto A[1]$ and the Tate twist functor $A\mapsto A(1)$. In this
subsection, we study the effects of such translation functors on symmetry.
Lemma \ref{l.final} is used in the main text to handle the symmetry of the
Lefschetz pairing.

Recall that a \emph{category with translation} \cite[Definition 10.1.1
(i)]{KS} is a category $\cD$ equipped with an equivalence of categories
$T\colon \cD\to \cD$.  We let $T^{-1}\colon \cD\to \cD$ denote a
quasi-inverse of $T$. For an integer $n$, we will often write $[n]$ for
$T^n$. Recall that a \emph{functor of categories of translation}
\cite[Definition 10.1.1 (ii)]{KS} $(\cD,T)\to (\cD',T')$ is a functor
$F\colon \cD\to \cD'$ endowed with a natural isomorphism $\eta\colon
FT\simto T'F$. Recall that a \emph{morphism of functors of categories with
translation} $(F,\eta)\to (G,\xi)$ is a natural transformation $\alpha\colon
F\to G$ of functors such that the following diagram commutes
    \[\xymatrix{FT\ar[r]^\eta_\sim\ar[d]_{\alpha T} & T'F\ar[d]^{T' \alpha}\\
    GT\ar[r]^\xi_\sim & T'G.}
    \]

Our first goal is to define duality on categories with translation, variant
of Definition \ref{d.d}. We endow $\cD^{\op}$ with the translation functor
$(T^{\op})^{-1}\colon \cD^{\op}\to\cD^{\op}$. We endow $F^\op\colon
\cA^\op\to \cA'^\op$ with the isomorphism $F^\op (T^\op)^{-1}\simto
(T'^\op)^{-1} F^\op$ induced by
\[
\eta^\op\colon T'^\op
F^\op \simto F^\op T^\op.
\]

\begin{definition}[Duality on a category with translation]\label{d.pdt}
Let $(\cD,T)$ be a category with translation. A \emph{duality on $(\cD,T)$}
is a functor of categories with translation $(D,\eta)\colon
(\cD^\op,(T^{\op})^{-1}) \to (\cD,T)$ endowed with a structure of duality on
the underlying functor $D\colon \cD^\op\to \cD$ such that $\ev\colon \id_\cD
\to DD^\op$ is a morphism of functors of categories with translation. This
means that the diagram
\[\xymatrix{T\ar[r]^{\ev}\ar[d]_\ev & DD^\op
T\ar[d]^{D(T^\op)^{-1}\eta^\op T}\\
TDD^{\op}&\ar[l]_{\eta D^\op}  D(T^\op)^{-1}D^\op}
\]
commutes. In other words, the isomorphisms $\eta^{-1}\colon TD\simto
D(T^\op)^{-1}$ and $T^{-1}\eta T^\op\colon T^{-1} D \simto D T^\op$ are
\emph{transposes} of each other in the sense of Definition \ref{d.symF}.
\end{definition}

The above definitions have obvious additive variants. An \emph{additive
category with translation} is defined to be a category with translation
whose underlying category is additive. For additive categories with
translation $\cD$ and $\cD'$, a \emph{functor of additive categories with
translation} $\cD\to \cD'$ is defined to be a functor of categories with
translation whose underlying functor is additive. An \emph{additive duality}
on an additive category with translation is a duality on the category with
translation such that the underlying functor is additive.

As in the case without translation, a basic example of dualities on
categories with translation is provided by closed symmetric categories with
translation (see Construction \ref{c.Homtrans} below). Our next goal is to
define symmetric categories with translation, variant of Definition
\ref{d.scat}. Note that in the example of $\ell$-adic sheaves, the shift and
twist functors differ in signs with regard to tensor products. To deal with
the two cases simultaneously, we let $\epsilon=\pm 1$. The case
$\epsilon=-1$ of the following definition corresponds to \cite[Definition
10.1.1 (v)]{KS}. For a more general notion, see \cite[D\'efinition
I.1.4.4]{Verdier}.

\begin{definition}\label{d.bif}
Let $\cD,\cD',\cD''$ be additive categories with translation. An
    \emph{$\epsilon$-bifunctor of additive categories with translation} $F\colon \cD\times \cD'\to
    \cD''$ is an additive bifunctor endowed with functorial isomorphisms $F(A[1],B)\simeq
    F(A,B)[1]$ and $F(A,B[1])\simeq F(A,B)[1]$ for objects $A$ of $\cD$ and $B$ of
    $\cD'$,
    such that the following diagram $\epsilon$-commutes
    \[\xymatrix{F(A[1],B[1])\ar[r]^\sim\ar[d]_\simeq & F(A,B[1])[1]\ar[d]^\simeq\\
    F(A[1],B)[1]\ar[r]^\sim & F(A,B)[2].}
    \]
\end{definition}

It then follows that the following diagram $\epsilon^{mn}$-commutes
    \[\xymatrix{F(X[m],Y[n])\ar[r]^\sim\ar[d]_\simeq & F(X,Y[n])[m]\ar[d]^\simeq\\
    F(X[m],Y)[n]\ar[r]^\sim & F(X,Y)[m+n].}
    \]

\begin{definition}[Symmetric category with translation]\label{d.smac}
An \emph{$\epsilon$-symmetric additive category with translation} is an
additive category with translation $\cD$ endowed with a symmetric structure
$\otimes$ and a structure of additive $\epsilon$-bifunctor of categories
with translation on $-\otimes -\colon \cD\times \cD\to \cD$, such that the
symmetry constraint, when restricted to each variable, is a morphism of
functors of categories with translation $\cD\to \cD$. We say that an
$\epsilon$-symmetric additive category with translation is \emph{closed} if
its underlying symmetric category is closed.
\end{definition}

For $\epsilon=1$, Definitions \ref{d.bif} and \ref{d.smac} make sense
without assuming that the categories in question are additive.

\begin{example}
Let $\cC$ be a symmetric \emph{monoidal} category and let $X$ be a
\emph{dualizable} object of $\cC$, that is, there exists an object $B$ of
$\cC$ such that $A\otimes B\simeq \one$. Then $-\otimes A$ endows $\cC$ with
the structure of a $1$-symmetric category with translation. This applies in
particular to the Tate twist functor on the Abelian category of perverse
$\Qlb$-sheaves.
\end{example}

\begin{example}\label{e.topos}
The derived category of any commutatively ringed topos is a closed
$-1$-symmetric additive category with translation. Similarly, the derived
category of $\Qlb$-sheaves is a closed $-1$-symmetric additive category with
translation.
\end{example}

Let $\cD$ be an $\epsilon$-symmetric additive category with translation.

\begin{lemma}\label{l.trans}
The diagram
\[\xymatrix{A[m]\otimes B[n] \ar[d]_c^\simeq \ar[r]^\sim & (A\otimes B[n])[m]\ar[r]^\sim
& (A\otimes B)[m+n]\ar[d]^{c[m+n]}_\simeq\\
 B[n]\otimes A[m] \ar[r]^\sim & (B\otimes A[m])[n]\ar[r]^\sim& (B\otimes A)[m+n]}
\]
$\epsilon^{mn}$-commutes for all objects $A$, $B$ of $\cD$ and all integers
$m$, $n$. Here $c$ denotes the symmetry constraint.
\end{lemma}

\begin{proof}
In the diagram
\[\xymatrix{A[m]\otimes B[n]\ar[r]^\sim\ar[d]_c & (A\otimes B[n])[m]\ar[r]^\sim
\ar[d]^{c[m]}
& (A\otimes B)[m+n]\ar[d]^{c[m+n]}\\
B[n]\otimes A[m]\ar[r]^\sim\ar[rd]^\sim & (B[n]\otimes A)[m]\ar[r]^{\sim} & (B\otimes A)[m+n]\\
&(B\otimes A[m])[n]\ar[ru]^\sim,}
\]
the upper squares commute by functoriality, and the lower triangle
$\epsilon^{mn}$-commutes by the definition of $\epsilon$-bifunctor of
additive categories with translation.
\end{proof}

\begin{construction}\label{c.shift}
Let $A,B,K$ be objects of $\cD$. A pairing $A\otimes B \to K$ induces a
pairing
\[(A[m])\otimes (B[n])\simeq (A\otimes B[n])[m]\simeq (A\otimes B)[m+n] \to
K[m+n].
\]
\end{construction}

Lemma \ref{l.trans} implies the following.

\begin{lemma}\label{l.shift}
Let $A,B,K$ be objects of $\cD$. Let $A\otimes B\to K$ and $B\otimes A\to K$
be pairings that are $\sigma$-transposes of each other. Then the induced
pairings $(A[m])\otimes (B[n]) \to K[m+n]$ and $(B[n])\otimes (A[m])\to
K[m+n]$ are $\epsilon^{mn} \sigma$-transposes of each other.
\end{lemma}

Let $\cD$ be a \emph{closed} $\epsilon$-symmetric additive category with
translation.

\begin{construction}\label{c.Homtrans}
Consider the isomorphisms
\begin{gather*}
\alpha_n\colon \cHom(A[-n],B)\simto
\cHom(A,B)[n],\\
\beta_n\colon \cHom(A,B[n])\simto \cHom(A,B)[n]
\end{gather*}
given by
the isomorphisms
\begin{gather*}
\Hom(C,\cHom(A[-n],B))\simeq \Hom(C\otimes A[-n],B)\simeq \Hom((C\otimes A)[-n],B),\\
\Hom(C,\cHom(A,B[n]))\simeq \Hom(C\otimes A,B[n])\simeq \Hom((C\otimes A)[-n],B),\\
\Hom((C\otimes A)[-n],B) \simeq \Hom(C[-n]\otimes A, B)\qquad\\
\qquad\simeq \Hom(C[-n],\cHom(A,B))\simeq \Hom(C,\cHom(A,B)[n])
\end{gather*}
for objects $A,B,C$ of $\cD$. We have
$\alpha_m\alpha_n=\epsilon^{mn}\alpha_{m+n}$, $\beta_m\beta_n=\beta_{mn}$,
$\alpha_m\beta_n=\epsilon^{mn}\beta_n\alpha_m$. We endow $\cHom(-,-)\colon
\cD^{\op}\times \cD\to \cD$ with the structure of $\epsilon$-bifunctor of
additive categories with translation given by $\epsilon\alpha_1$ and
$\beta_1$.\footnote{The sign convention is adopted here only to fix ideas.
Our results do not depend on the convention.} Let $\tilde
\alpha_n=(\epsilon\alpha_1)^n=\epsilon^{n(n+1)/2}\alpha_n$.

In particular, $D_A\colon \cD^{\op}\to \cD$ is endowed with the structure of
functor of additive categories with translation, which, together with
$\ev\colon \id_\cD \to D_AD_A^{\op}$, defines an additive duality on the
additive category with translation (cf.\ \cite[Proposition 3.2.1]{CH}).
\end{construction}

\begin{remark}\label{r.shift}
Construction \ref{c.shift} corresponds to the construction sending $f\colon
A\to D_K B$ to $\epsilon^{n(n-1)/2}$ times the morphism
\[
A[m]\xrightarrow{f[m]} (D_K
B)[m]\xrightarrow[\sim]{\tilde\alpha_{-n}}
D_K(B[n])[m+n]\xrightarrow{\beta^{-1}_{m+n}} D_{K[m+n]}(B[n]),
\]
where $\tilde\alpha$ and $\beta$ are as in Construction \ref{c.Homtrans}. In
fact, the following diagram $\epsilon^{n(n-1)/2}$-commutes
\[\xymatrix{\Hom(A[m]\otimes B[n],K[m+n])\ar[d]_\simeq\ar[r]^\sim
& \Hom((A\otimes B[n])[m],K[m+n])\ar[d]^\simeq
\\
\Hom(A[m],D_{K[m+n]}(B[n]))\ar[d]_\simeq^{\tilde \alpha_{-n}} &\Hom((A\otimes B)[m+n],K[m+n])\ar[r]^-\sim & \Hom(A\otimes B,K)\ar[d]^\simeq\\
\Hom(A[m],D_K(B[n])[m+n])\ar[r]^-{\beta^{-1}_{m+n}}_-\sim & \Hom(A[m],D_KB[m])\ar[r]_-\sim & \Hom(A,D_KB).}
\]

Thus, Construction \ref{c.shift} corresponds to the form transformation
$\gamma_{m,n}\colon T^m D_K\simto D_{K[m+n]}T^n$, defined to be
$\epsilon^{n(n-1)/2}$ times the isomorphism
\[
T^m D_K \xrightarrow[\sim]{\tilde \alpha_{-n}}
T^{m+n} D_K T^n \xrightarrow[\sim]{\beta^{-1}_{m+n}} D_{K[m+n]} T^{n}
\]
given by Construction \ref{c.Homtrans}. By the above, the
$\epsilon^{mn}$-transpose of $\gamma_{m,n}$ is $\gamma_{n,m}$.
\end{remark}

We combine the above discussion on translation with our previous discussion
on truncation into the following, which is applied in the proof of
Proposition \ref{p.projf} to the Lefschetz pairing.

\begin{lemma}\label{l.final}
Let $\cD$ be a closed $-1$-symmetric additive category with translation.
Assume that the underlying category with translation is further equipped
with a triangulated structure and a $t$-structure $P$. Let $K$ and $L$ be
objects of $\cD$ such that $D_L$ is a right $t$-exact triangulated functor.
For any $\sigma$-symmetric pairing $A\otimes A \to K$ and any morphism
$\xi\colon K[2n]\to L$, the pairing $\PH^n A\otimes \PH^n A\to L$ induced by
\[A[n]\otimes A[n]\simto (A\otimes A)[2n]\to K[2n]\xrightarrow{\xi} L\]
is $(-1)^n\sigma$-symmetric.
\end{lemma}

In fact, the form transformation $\PH^{n} D_K\to D_L \PH^{n}$ given by
\[\tau^{[0,0]}T^{n}D_K\xrightarrow{\gamma_{n,n}} \tau^{[0,0]}D_{K[2n]}T^{n}\xrightarrow{D_\xi} \tau^{[0,0]}D_{L}T^{n}
\xrightarrow{\gamma_\tau} D_L \tau^{[0,0]} T^{n}
\]
is $(-1)^n$-symmetric. Here $\gamma_{n,n}$ and $\gamma_\tau$ are as in
Remarks \ref{r.shift} and \ref{r.trunc}.

\begin{remark}
Let us mention in passing that Lurie's theory of stable $\infty$-categories
\cite[Chapter~1]{Lurie} provides a nicer framework for symmetric monoidal
structures in derived categories. If $(\cD,\otimes)$ is a closed symmetric
monoidal $\infty$-category such that the underlying $\infty$-category $\cD$
is stable, then $-\otimes-$ and $\cHom(-,-)$ are automatically exact in each
variable and the homotopy category of $\cD$ is a closed $-1$-symmetric
additive category with translation.
\end{remark}

\subsection{Duality and nilpotence}\label{s.nil}
In this subsection, we study symmetry of primitive parts under a (twisted)
nilpotent operator. We formulate the problem in the language of duality on
category with translation introduced in Definition \ref{d.pdt}. The main
result of this subsection is Proposition \ref{p.Pi}. This is applied in the
main text to the logarithm of the monodromy operator associated to a normal
crossing divisor to show that Grothendieck's six operations preserve
$\rK_\orth$ (see the proof of Proposition \ref{p.NCD}).

Let $(\cA,T)$ be an additive category with translation. In this subsection,
we denote $T^n A$ by $A(n)$ instead of $A[n]$. Our first goal is to define a
category of objects with nilpotent operators and record its relation with
duality.

\begin{construction}
Consider the additive category $\Nil(\cA,T)$ of pairs $(A,N)$ of an object
$A$ of $\cA$ and a morphism $N\colon A(1)\to A$ which is nilpotent in the
sense that there exists an integer $d\ge 0$ such that $N^d\coloneqq N\circ
N(1)\circ\dots\circ N(d-1)\colon A(d)\to A$ is the zero morphism. A morphism
$(A,N)\to (A',N')$ is a morphism $f\colon A\to A'$ of $\cA$ satisfying
$N'f(1)=fN$.

There are two ways to identify $\Nil(\cA,T)^\op$ and
$\Nil(\cA^\op,(T^\op)^{-1})$, which differ by a sign. We fix $\sigma=\pm 1$
and consider the isomorphism of categories
\[E_\cA=E_{(\cA,T),\sigma}\colon \Nil(\cA,T)^\op\to
\Nil(\cA^\op,(T^\op)^{-1})
\]
sending $(A,N\colon A(1)\to A)$ to $(A,\sigma N(-1)\colon A\to A(-1))$. The
composite
\[\Nil(\cA,T)\xrightarrow{E_\cA^\op}
\Nil(\cA^\op,(T^\op)^{-1})^\op\xrightarrow{E_{\cA^{\op}}} \Nil(\cA,T)
\]
equals the identity. The duality we put on $\Nil(\cA,T)$ will depend on the
choice of $\sigma$. In the main text we take $\sigma=-1$.

Let $F\colon (\cA,T)\to (\cA',T')$ be a functor of additive categories with
translation. Then $F$ induces an additive functor $\Nil_F\colon
\Nil(\cA,T)\to \Nil(\cA',T')$ carrying $(A,N\colon TA\to A)$ to
$(FA,T'FA\simeq FTA \xrightarrow{FN} FA)$ and $f\colon (A,N)\to (A',N')$ to
$Ff$. Let $\gamma\colon F\to F'$ be a morphism of functors of categories
with translation. Then $\gamma$ induces a natural transformation
$\Nil_\gamma\colon \Nil_F\to \Nil_{F'}$, which is a natural isomorphism if
$\gamma$ is an isomorphism.

The following diagrams commute
\[\xymatrix{\Nil(\cA,T)^\op \ar[r]^{\Nil_F^\op}\ar[d]_{E_{\cA}} &
\Nil(\cA',T')^\op\ar[d]^{E_{\cA'}}
& E_{\cA'}\Nil_{F'}^\op \ar[r]^{\Nil_\gamma^\op}\ar@{=}[d] & E_{\cA'}\Nil_F^\op\ar@{=}[d]\\
\Nil(\cA^\op,(T^\op)^{-1})\ar[r]^{\Nil_{F^\op}} & \Nil(\cA'^\op,(T'^\op)^{-1})
& \Nil_{F^\op}E_{\cA'}\ar[r]^{\Nil_{\gamma^\op}} & \Nil_{F^\op}E_{\cA}.
}
\]
\end{construction}

\begin{construction}\label{c.DNil}
Let $D\colon (\cA^\op,(T^{\op})^{-1}) \to (\cA,T)$ be an additive duality on
the additive category with translation (see the comment following Definition
\ref{d.pdt}). Consider the functor $D_{\Nil(\cA,T)}$, composite of
\[\Nil(\cA,T)^\op\xrightarrow[\sim]{E_\cA}\Nil(\cA^\op,(T^{\op})^{-1})\xrightarrow{\Nil_D} \Nil(\cA,T),\]
and the natural transformation
\[\id_{\Nil(\cA,T)}\xrightarrow{\Nil_\ev} \Nil_D \Nil_{D^\op}=\Nil_D \Nil_{D^\op} E_{\cA^\op}E_\cA^\op
=\Nil_D E_\cA \Nil_D^\op E_\cA^\op=D_{\Nil(\cA,T)}D_{\Nil(\cA,T)}^\op.
\]
These define an additive duality on the additive category $\Nil(\cA,T)$,
which is strong if $D$ is strong on $\cA$.
\end{construction}

In the rest of this section, let $(\cA,T)$ be an \emph{Abelian category with
translation}, namely an additive category with translation whose underlying
category $\cA$ is Abelian. Our next goal is to review the decomposition into
primitive parts. The following is a variant of \cite[Proposition 1.6.1,
1.6.14]{WeilII}, with essentially the same proof.

\begin{lemma}\label{l.mono}
Let $(A,N)$ be an object of $\Nil(\cA,T)$. Then there exists a unique finite
increasing filtration $M$ of $A$ satisfying $NM_j(1)\subseteq M_{j-2}$ and
such that for $k\ge 0$, $N^k$ induces an isomorphism $\gr^M_k A (k)\simto
\gr^M_{-k} A$.
\end{lemma}

\begin{proof}
Let $d\ge 0$ be an integer such that $N^{d+1}=0$. We proceed by induction on
$d$. We have $M_d=A$ and $M_{-d-1}=0$. For $d>0$, $M_{d-1}=\Ker (N^d)(-d)$
and $M_{-d}=\Img (N^d)$. We have $N^d=0$ on $\Ker (N^d)(-d)/\Img (N^d)$ and
let $M'$ be the corresponding filtration given by induction hypothesis. For
$-d\le i\le d-1$, $M_i$ is the inverse image in $\Ker (N^d)(-d)$ of
$M'_i\subseteq \Ker (N^d)(-d)/\Img (N^d)$.
\end{proof}

The following is an immediate consequence of the construction of the
filtration $M$.

\begin{lemma}\label{l.monof}
Let $f\colon (A,N)\to (A',N')$ be a morphism of $\Nil(\cA,T)$. Then $f$ is
compatible with the corresponding filtrations. More precisely, if $M$ and
$M'$ denote the corresponding filtrations, then $f(M_j)\subseteq M'_j$.
\end{lemma}

For $i\le 0$, let $P_i(A,N)=\Ker(N\colon \gr^M_i A(1)\to \gr^M_{i-2}
A)(-1)$. The inclusion $\Ker (N)(-1)\subseteq A$ induces an isomorphism
$\gr_i^M(\Ker(N)(-1))\simto P_i(A,N)$. We thus obtain functors
\[P_i=P_{i,\cA}\colon\Nil(\cA,T)\to \cA.\]
For all $j$, we have
\[\gr_j^M A\simeq
\bigoplus_{\substack{k\ge \lvert j\rvert\\ k\equiv j\pmod 2}}
P_{-k}(A,N)(-\tfrac{j+k}{2}).
\]

We now proceed to define form transformations on the primitive part
functors. Let $(A,N)$ be an object of $\Nil(\cA,T)$. If $M=M(A,N)$ and
$M^*=M(E_{\cA}(A,N))$, then we have the following short exact sequence in
$\cA$
\[0\to M_{-j-1} \to A \to M^*_j\to 0,\]
so that $\gr_{-j}^M A$ can be identified with $\gr_j^{M^*} A$. Moreover,
$N^{-i}\colon \gr_{-i}^M A(-i)\simto \gr_i^M A$ induces an isomorphism in
$\cA$
\[\alpha_\cA(A,N)\colon P_{i,\cA^\op}(E_\cA(A,N))(-i)\simeq \gr_{-i}^M(\Coker(N))(-i) \simto
\gr_i^M(\Ker(N)(-1))\simeq  P_{i,\cA}(A,N).
\]
This defines a natural isomorphism of functors $\alpha_\cA\colon
P_{i,\cA}^\op \simto (T^\op)^{-i}P_{i,\cA^\op}E_\cA$. By definition, we have
the following.

\begin{lemma}\label{l.Pi1}
The isomorphism $T^{-i}P_{i,\cA^\op}^\op E_\cA^\op
\xrightarrow[\sim]{\alpha_{\cA^\op}}
P_{i,\cA}E_{\cA^\op}E_{\cA}^\op=P_{i,\cA}$ is $\sigma^i\alpha_\cA^\op$.
\end{lemma}

Let $(\cA,T)$ and $(\cA',T')$ be Abelian categories with translation and let
$F\colon (\cA,T)\to (\cA',T')$ be a functor of categories with translation
such that the underlying functor $\cA\to \cA'$ is \emph{exact}. Let $(A,N)$
be an object of $(\cA,T)$ and let $M=M(A,N)$, $M'=M(\Nil_F(A,N))$. The
exactness of $F$ allows us to identify $F(M_jA)$ as a sub-object of $FA$,
and under this identification we have $F(M_j A)=M'_j(FA)$. We have an
obvious natural isomorphism $\beta_F\colon P_{i,\cA'}\Nil_F \simto F
P_{i,\cA}$. The following functoriality of $\beta$ is obvious.

\begin{lemma}\label{l.Pi2}
Let $F,F'\colon (\cA,T)\to (\cA',T')$ be functors of categories with
translation such that the underlying functors are exact and let
$\gamma\colon F\to F'$ be a morphism of functors of categories with
translation. Then the following diagram commutes
\[\xymatrix{P_{i,\cA'}\Nil_F \ar[r]^{\beta_F}_\sim \ar[d]_{\Nil_\gamma} &
F P_{i,\cA}\ar[d]^\gamma\\
P_{i,\cA'}\Nil_{F'} \ar[r]^{\beta_{F'}}_\sim &
F' P_{i,\cA}.}
\]
\end{lemma}

By construction, the isomorphisms $\alpha$ and $\beta$ have the following
compatibility.

\begin{lemma}\label{l.Pi3}
The following diagram commutes
\[\xymatrix{(T'^\op)^{-i}F^\op P_{i,\cA^\op}E_\cA  & \ar[l]_\sim F^\op
(T^\op)^{-i}P_{i,\cA^\op}E_\cA
 & \ar[l]_-{\alpha_\cA}^-\sim F^\op P_{i,\cA}^\op \ar[d]^{\beta_F^\op}_\simeq \\
(T'^\op)^{-i}P_{i,\cA'^\op}\Nil_{F^\op} E_\cA\ar@{=}[r]\ar[u]^{\beta_{F^\op}}_\simeq
& (T'^\op)^{-i}P_{i,\cA'^\op} E_{\cA'} \Nil_F^\op
 & \ar[l]_-{\alpha_{\cA'}}^-\sim P_{\cA'}^\op\Nil_F^\op.}
\]
\end{lemma}

The following is the main result of this subsection.

\begin{prop}\label{p.Pi}
Let $(\cA,T)$ be an Abelian category with translation and let $D\colon
(\cA^\op,(T^\op)^{-1})\to (\cA,T)$ be a duality such that the underlying
functor $\cA^\op\to\cA$ is exact. Then, for $i\le 0$, the composite
isomorphism
\[P_{i,\cA}D_{\Nil(\cA,T)}=P_{i,\cA}\Nil_D E_\cA \xrightarrow[\sim]{\beta_D}
DP_{i,\cA^\op} E_\cA\xrightarrow[\sim]{\alpha_\cA^{-1}} (D(T^\op)^i) P_{i,\cA}^\op
\]
is $\sigma^i$-symmetric.
\end{prop}

Note that $T^{-i}D\simeq D(T^\op)^i\colon \cA^\op \to \cA$ endowed with the
natural transformation $\id_{\cA}\xrightarrow{\ev} DD^{op}\simeq
(D(T^\op)^i)(T^{-i}D)^\op$ is a duality on $\cA$. By the proposition,
$P_{i,\cA}$ carries $\sigma'$-self-dual objects of $\Nil(\cA,T)$ to
$\sigma^i\sigma'$-self-dual objects of $\cA$.

\begin{proof}
In the diagram
\[\xymatrix@C=.5cm{P_i\ar[r]^{\Nil_\ev}\ar[dd]_{\ev} &
P_{i,\cA}\Nil_D\Nil_{D^\op}\ar@{=}[r]\ar[d]^{\beta_D}_\simeq
& P_{i,\cA}\Nil_D\Nil_{D^\op} E_{\cA^\op} E_\cA^\op \ar@{=}[r]\ar[d]^{\beta_D}_\simeq
& P_{i,\cA} \Nil_D E_\cA \Nil_D^\op E_\cA^\op\ar[d]^{\beta_D}_\simeq\\
& DP_{i,\cA^\op} \Nil_{D^\op}\ar@{=}[r] \ar[ld]_{\beta_{D^\op}}^\sim
& DP_{i,\cA^\op} \Nil_{D^\op}E_{\cA^\op} E_\cA^\op \ar@{=}[r] \ar[ld]_{\beta_{D^\op}}^\sim
& DP_{i,\cA^\op} E_\cA \Nil_D^\op E_\cA^\op\ar[ddd]^{\alpha_\cA^{-1}}\\
DD^\op P_i\ar@{=}[r]\ar[dd]_\simeq\ar[rd]_{(\alpha_\cA^\op)^{-1}}^\sim
& DD^\op P_{i,\cA} E_{\cA^\op} E_\cA^\op\ar[d]^{\alpha_{\cA^\op}^{-1}}\\
&DD^\op T^{-i}P_{i,\cA^\op}^\op E_\cA^\op\ar[d]^\simeq\\
D(T^\op)^{i}D^\op T^i P_{i,\cA}\ar[r]^{(\alpha_\cA^\op)^{-1}}
& D(T^\op)^i D^\op P_{i,\cA^\op}^\op E_\cA^\op \ar[rr]^{\beta_D^\op}_\sim
&& D(T^\op)^iP_{i,\cA}^\op \Nil_D^\op E_\cA^\op,}
\]
the triangle $\sigma^i$-commutes by Lemma \ref{l.Pi1}, the upper-left inner
cell commutes by Lemma \ref{l.Pi2}, the lower-right inner cell commutes by
Lemma \ref{l.Pi3}, and the other inner cells trivially commute.
\end{proof}

\end{document}